\documentclass{amsart}       

\usepackage{amsmath}
\usepackage{latexsym,amssymb,dsfont}
\usepackage{thmtools,mathtools,amsthm}
\usepackage[T1]{fontenc}
\usepackage[utf8]{inputenc}
\usepackage{tikz-cd}
\usepackage{quiver}
\usepackage{url}
\usepackage{extpfeil}
\usepackage{extarrows}

\usepackage[style=numeric,maxnames=10,backend=biber,
useprefix=true,hyperref=true]{biblatex}
\theoremstyle{definition}
\newtheorem{definition}{Definition}[section]

\theoremstyle{theorem}
\newtheorem{theorem}[definition]{Theorem}
\newtheorem{lemma}[definition]{Lemma}
\newtheorem{proposition}[definition]{Proposition}
\newtheorem{corollary}[definition]{Corollary}
\newtheorem{lem}[definition]{Lemma}
\newtheorem{prop}[definition]{Proposition}
\newtheorem{cor}[definition]{Corollary}

\usepackage[mathscr]{euscript}

\usepackage{mymacros}

\usepackage[bookmarksdepth=2,pdfencoding=unicode,colorlinks=true]{hyperref}
\hypersetup{allcolors=[rgb]{0.1,0.1,0.4}}
\usepackage{cleveref}

\usepackage{mymacros}

\begin{document}

\title{Internal sums for synthetic fibered \texorpdfstring{$\inftyone$}{(∞,1)}-categories}
\author[Jonathan Weinberger]{Jonathan Weinberger}
\address{Johns Hopkins University, Department of Mathematics, 3400 N Charles St, 21218 Baltimore, MD, USA}
\email{jweinb20@jhu.edu}

\date{\today}

\maketitle

\begin{abstract}
We give structural results about bifibrations of (internal) $\inftyone$-categories with internal sums. 
This includes a higher version of Moens' Theorem, characterizing cartesian bifibrations with extensive \aka~stable and disjoint 
internal sums over lex bases as Artin gluings of lex functors. We also treat a generalized version of Moens' Theorem due to Streicher which does not require the Beck--Chevalley condition. Furthermore, we show that also in this setting the Moens fibrations can be characterized via a condition due to Zawadowski.

Our account overall follows Streicher's presentation of fibered category theory \`{a} la B\'{e}nabou, generalizing the results to the internal,
higher-categorical case, formulated in a synthetic setting. Namely, we work inside simplicial homotopy type theory, which has been introduced by Riehl and Shulman as a logical system to reason about
internal $\inftyone$-categories, interpreted as Rezk objects in any given Grothendieck--Rezk--Lurie $\inftyone$-topos. \\[1.5\baselineskip]
\end{abstract}

\section{Introduction}

\subsection{Overview}
In the work at hand, we investigate the notion of \emph{internal sums} for fibrations of synthetic $\inftyone$-categories. In particular, 
we consider the case of \emph{extensive} internal sums and prove a version of Moens' Theorem which says that essentially all fibrations with extensive internal sums
arise as Artin gluings of lex functors. The setting is Riehl--Shulman's synthetic $\inftyone$-category theory~\cite{RS17} which takes place in a simplicially augmented
version of \emph{homotopy type theory (HoTT)}~\cite{hottbook,AW05,StrSimp}. We generalize several results and proofs from Streicher's notes~\cite{streicher2020fibered}
on fibered category theory \`{a} la B\'{e}nabou to this setting.

Classically, a category $\mathbb C$ with all small coproducts is called \emph{(infinitary) extensive} if for any small family of objects $(a_i)_{i \in I}$ in $\mathbb C$ the canonical maps $\prod_{i \in I} \mathbb C/a_i \to \mathbb C/\coprod_{i \in I} a_i$. This can be equivalently expressed as saying that small coproducts are \emph{stable} under arbitrary pullbacks, and \emph{disjoint} in the sense that any coproduct injection is a monomorphism, and the intersection of a pair of different coproduct injections is the initial object. If $\mathbb C$ is a finitely complete extensive category, it is called \emph{lextensive}.

The latter phrasing of extensiveness is most useful in the fibrational context, and indeed Moens or lextensive fibrations are exactly the fibrational analogue of lextensive categories.

We provide a comprehensive treatment of them in the higher setting, formulated in a type theory that is interpretable in internal $(\infty,1)$-categories inside an arbitrary $(\infty,1)$-topos.

The extensivity property is akin to the \emph{descent property}, which is part of Giraud's characterization of Grothendieck toposes, which lifts to the $(\infty,1)$-categorical case via work by Rezk and Lurie. Therefore, our study adds to the study of higher topos theory from a fibrational point of view. It can also be used to study \emph{internal higher topos theory}~\cite{MWInt,SteInt} because it is phrased in an appropriate type theory. Indeed, classically Moens' Theorem \aka~the characterization theorem of lextensive fibrations has been used to study geometric morphisms from a fibered point of view~\cite{streicher2020fibered,StrFVGM}. Thus we also hope for future applications in higher categorical logic~\cite{SteCompr,SteGSh,SteSites}, some of which suggested in~\Cref{ssec:persp}.

\subsection{Beck--Chevalley families}
Let $\mc U$ be a universe. A co-/cartesian family $P:B \to \UU$ of $\mc U$-small synthetic $\inftyone$-categories \aka~\emph{Rezk types} 
possesses functorial transport for directed arrows in the base, \ie~any arrow $u:a \to_B b$ induces a transport functor 
$u_!: P\,a \to P\,b$ or $u^*: P\,b \to P\,a$, respectively. A family is \emph{bicartesian} if it is both cocartesian and cartesian 
in which case the transport functors form an adjunction $u_! : P\,a \rightleftarrows P\,b : u^*$. A \emph{Beck--Chevalley family} 
is a bicartesian family that satisfies the \emph{Beck--Chevalley condition (BCC)}, which is phrased as follows. We denote 
cartesian arrows by $\cartarr$ and cocartesian arrows by $\cocartarr$, respectively. The BCC says that any dependent square
\[\begin{tikzcd}
	\bullet && \bullet \\
	\bullet && \bullet
	\arrow[from=1-1, to=2-1, cart]
	\arrow[from=2-1, to=2-3, cocart]
	\arrow[from=1-1, to=1-3]
	\arrow[from=1-3, to=2-3, cart]
\end{tikzcd}\]
over a pullback is itself a pullback precisely if the top horizontal map is cocartesian, too. This condition is also referred to as $P$ having \emph{internal sums}. The motivation for this can be understood from the following ($1$-categorical) situation. Let $\mathbb C$ be some not necessarily small category. The \emph{family fibration of $\mathbb C$}~\cite[Paragraph~(3.3)]{benabou-foundations}, \cite[Example~1.5]{streicher2020fibered}, \cite[Example~8.1.9(b)]{BorHandb2}, \cite[Example~1.2.1]{jacobs-cltt}, \cite[Example~2.7]{SterAng-RelCT}, is defined as the Grothendieck construction of the functor
\[ \Fam(\mathbb C): \mathbf{Set}^{\Op} \to \mathbf{Cat}\]
with
\[ \Fam(\mathbb C)(I) := \mathbb C^I,\]
\[ \Fam(\mathbb C)(u:J \to I) := \big(\mathbb C^u := u^* := (-) \circ u: \mathbb C^I \to \mathbb C^J\big).\]
One can show that the fibration $\Fam(\mathbb C) \to \mathbf{Set}$ having internal sums (namely, being a bifibration satisfying the BCC) is equivalent to $\mathbb C$ having small sums. In accordance with B\'{e}nabou's philosophy, which is alluded to in the aforementioned sources, one seeks to generalize such considerations to an arbitrary base with pullbacks in place of just the category of sets.

Another important instance of bifibrations satisfying the Beck--Chevalley condition is given by \emph{Artin gluings} of lex functors. Namely, given a lex functor $F:B \to C$ between two lex Rezk types $B$ and $C$, its \emph{Artin gluing} $\gl(F): \commaty{C}{F} \fibarr B$ is defined by pulling back the codomain fibration $\partial_1: C^{\Delta^1} \fibarr C$ along $F$:
\[\begin{tikzcd}
	{\commaty{C}{F}} && {C^{\Delta^1}} \\
	B && C
	\arrow[from=1-1, to=1-3]
	\arrow["{\gl(F)}"', two heads, from=1-1, to=2-1]
	\arrow["F"', from=2-1, to=2-3]
	\arrow["{\partial_1}", two heads, from=1-3, to=2-3]
	\arrow["\lrcorner"{anchor=center, pos=0.125}, draw=none, from=1-1, to=2-3]
\end{tikzcd}\]

\subsection{Moens families}
If $F : B \to C$ is a functor between some lex Rezk types one can show that the functor $\gl(F): \commaty{C}{F} \fibarr B$ is a lex bifibration, meaning that each fiber is a lex Rezk type and the cartesian transport maps are lex functors (which is then automatic since $u^*$ is a right adjoint). If, in addition, the functor $F$ is lex, the internal sums of $\gl(F)$ are \emph{stable} and \emph{disjoint}. The former condition means that cocartesian arrows are stable under pullback along any arrow. Disjointness is defined as the condition that the fibered diagonal $\delta_f$ of a cocartesian arrow $f$ is always cocartesian:
\[\begin{tikzcd}
	d \\
	& {d \times_e d} && d \\
	& d && e
	\arrow["f"', from=3-2, to=3-4, cocart]
	\arrow["f", from=2-4, to=3-4, cocart]
	\arrow[curve={height=-18pt}, Rightarrow, no head, from=1-1, to=2-4]
	\arrow[curve={height=12pt}, Rightarrow, no head, from=1-1, to=3-2]
	\arrow[from=2-2, to=3-2]
	\arrow[from=2-2, to=2-4]
	\arrow["{\delta_f}", dashed, from=1-1, to=2-2, cocart]
	\arrow["\lrcorner"{anchor=center, pos=0.125}, draw=none, from=2-2, to=3-4]
\end{tikzcd}\]
These are fibrational generalizations of the usual notion of stability and disjointness of sums in a category, giving rise to the notion of \emph{extensivity}. Accordingly, we call a lex cartesian family \emph{extensive} or \emph{Moens family} if it has stable and disjoint internal sums. Occasionally, we might be interested in discarding the disjointness condition in which case we call the fibration \emph{pre-Moens}. As a central result, \emph{Moens' Theorem} establishes a correspondence between lex functors and lextensive fibration.

More precisely, fix a lex type $B$. Consider the subuniverse
\[ \LexRezk \defeq \sum_{A:\UU} \isRezk(A) \times \isLex(A)\]
of lex Rezk types. This gives rise to the type $B \downarrow^\lex \LexRezk$ of lex functors $F \colon B \to C$ from $B$ to some other lex Rezk type $C$.
Moens' Theorem will establish an equivalence of $B \downarrow^\lex$ with the type $\MoensFam(B)$ via the following mutually inverse maps:
\[\begin{tikzcd}
	{\mathrm{MoensFam}(B)} &&& {B \downarrow^\lex \mathrm{LexRezk}}
	\arrow[""{name=0, anchor=center, inner sep=0}, "\Phi", curve={height=-18pt}, from=1-1, to=1-4]
	\arrow[""{name=1, anchor=center, inner sep=0}, "\Psi", curve={height=-18pt}, from=1-4, to=1-1]
	\arrow["\simeq"{description}, Rightarrow, draw=none, from=0, to=1]
\end{tikzcd}\]
\begin{align*} 
	\Phi(P:B \to \UU) & \defeq  \lambda b.(!_b)_!(\zeta_b): B \to P\,z, \\
	\Psi(F:B \to C)&  \defeq \big(\gl(F):\commaty{C}{F} \to B\big),
\end{align*} 
where$z : B$ is terminal in $B$, and $\zeta : \prod_{b:B} P\,b$ is defined by letting $\zeta_b : P\,b$ be terminal in $P\,b$ for $b:B$. The theorem is due to Moens and his dissertation work~\cite{MoensPhD}. A systematic and comprehensive discussion is given by Streicher~\cite{streicher2020fibered} in his notes on fibered category theory \`{a} la B\'{e}nabou. In fact, Streicher observed that a generalized version of Moen's Theorem holds when dropping the assumption that the bifibration in questions satisfies the Beck--Chevalley condition. One still gets a meaningful notion of extensivity in this case, namely that a dependent square as follows is a pullback if and only if the top horizontal edge is cocartesian:
\[\begin{tikzcd}
	\bullet && \bullet \\
	\bullet && \bullet
	\arrow[from=1-1, to=1-3]
	\arrow[squiggly, from=1-1, to=2-1]
	\arrow[squiggly, from=1-3, to=2-3]
	\arrow[from=2-1, to=2-3, cocart]
\end{tikzcd}\]
This corresponds to the functor from the base only preserving the terminal element but not necessarily pullbacks. The equivalences give rise to a diagram as follows
\[\begin{tikzcd}
	{\MoensFam(B)} &&& {B \downarrow^\lex {\LexRezk}} \\
	\\
	{\GMoensFam(B)} &&& {B  \downarrow^\ter {\LexRezk}}
	\arrow[""{name=0, anchor=center, inner sep=0}, "{\omega'_{(-)}}"{description}, curve={height=12pt}, from=1-1, to=1-4]
	\arrow[hook, from=1-1, to=3-1]
	\arrow[hook, from=1-4, to=3-4]
	\arrow[""{name=1, anchor=center, inner sep=0}, "{\omega'_{(-)}}"{description}, curve={height=12pt}, from=3-1, to=3-4]
	\arrow[""{name=2, anchor=center, inner sep=0}, "\gl"{description}, curve={height=12pt}, from=3-4, to=3-1]
	\arrow[""{name=3, anchor=center, inner sep=0}, "\gl"{description}, curve={height=12pt}, from=1-4, to=1-1]
	\arrow["\simeq"{description}, Rightarrow, draw=none, from=1, to=2]
	\arrow["\simeq"{description}, Rightarrow, draw=none, from=0, to=3]
\end{tikzcd}\]
where the types at hand are the $\Sigma$-types over the type $B \to \UU$ of the respective flavors of fibrations, and of lex or terminal object-preserving functors, respectively. These \emph{generalized Moens fibrations} also play a role in work by Zawadowski~\cite{zawadowski-lax-mon} who introduced them using a different characterization. Streicher has proven both notions to be equivalent~\cite{streicher-cartbifib} which we also discuss.

\begin{figure}
		\[\begin{tikzcd}
			&&&&&&&&& {} \\
			& {} &&&&&&&&& {} \\
			& {\text{Moens/lextensive}} &&&&&&&& {} \\
			& {\text{pre-Moens}} \\
			& {\text{lex BC}} & {\text{gen.~Moens}} \\
			& {\text{BC}} & {\text{lex~cartesian}} \\
			{\text{cocartesian}} & {\text{bicartesian}} & {\text{cartesian}} \\
			& {}
			\arrow[hook, from=5-2, to=6-3]
			\arrow[hook, from=6-3, to=7-3]
			\arrow[hook, from=7-2, to=7-3]
			\arrow[hook', from=7-2, to=7-1]
			\arrow[hook, from=5-2, to=6-2]
			\arrow[hook, from=6-2, to=7-2]
			\arrow[hook, from=5-3, to=6-3]
			\arrow[hook, from=4-2, to=5-2]
			\arrow[hook, from=4-2, to=5-3]
			\arrow[hook, from=3-2, to=4-2]
		\end{tikzcd}\]
	\caption{Logical dependency of notions of fibration}
	\label{fig:notions-fib}
	\end{figure}
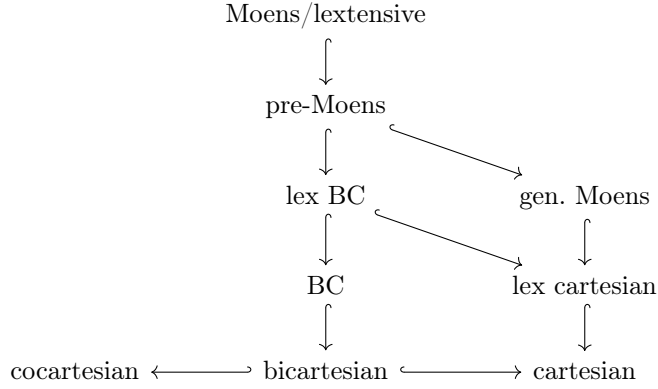

	\newpage

	\subsection{Further context and perspective on the work}\label{ssec:persp}

	\subsubsection{Geometric families of synthetic \texorpdfstring{$\inftyone$}{(∞,1)}-categories}
	Another interesting notion is that of a \emph{geometric fibration}, \ie, a Moens fibration $\pi:E \fibarr B$ that has \emph{small global sections} meaning that the fiberwise terminal element map $\zeta:B \to E$ has a right adjoint. Under the equivalence from Moens' Theorem, these correspond exactly to geometric morphisms between lex categories. This is discussed by Streicher in~\cite[Section~16]{streicher2020fibered} as well as Lietz in~\cite[Section~8]{LietzDip}, with previous considerations made by B\'{e}nabou in~\cite{ben-catlog}. We believe that this can be worked out in simplicial HoTT as well.
	
	\subsubsection{Jibladze's Theorem for higher toposes}
	The geometric fibrations naturally have a context in topos theory~\cite{StrFVGM}, and this should generalize in one way or the other to the setting of higher toposes. We hope that it will be possible to establish a version of Jibladze's Theorem~\cite[Appendix~A]{streicher2020fibered} for $\inftyone$-toposes (in the sense of $\infty$-sheaf toposes \`{a} la Grothendieck--Rezk--Lurie, or possibly the elementary case \`{a} la Shulman--Rasekh~\cite{ShuEHT,RasEHT}). In the classical case, Jibladze's Theorem~\cite{jibladze1989geometric} classifies fibered toposes with internal sums as gluings of lex functors (\ie, the inverse image part of some geometric morphism) to some topos.\footnote{also mentioned in~\cite{elephant2}, but see the remark at the end of~\cite[Appendix~A]{streicher2020fibered}} Note that in the higher case this relies on Moens' Theorem for $\inftyone$-categories which we have established in this article. At present, there is no synthetic notion of $\inftyone$-\emph{topos} within simplicial type theory. Thus,  when turning to the topos-specific aspects one would have to use a different setting, such as a suitable $\infty$-cosmos (or class thereof). For a notion of fibered $\infty$-toposes relying on a (possibly higher) notion of logical morphisms the work by Rasekh~\cite{RasEHT} on elementary higher toposes in an $\infty$-cosmos of $\inftyone$-categories could be guiding. For ensuing considerations of (local) smallness in this context the work by Stenzel~\cite{SteCompr} could be relevant. 
	
	\subsubsection{Moens' Theorem in $\infty$-cosmoses}
	We have shown that going from bicartesian fibrations to Beck--Chevalley and Moens fibrations goes through relatively smoothely in the setting of simplicial homotopy type theory. We find it reasonable to conjecture that it could very well be carried out for bicartesian fibrations internal to an arbitrary $\infty$-cosmos within Riehl--Verity's theory~\cite{RV21,RVyoneda}. This would then yield a model-independent version of Moens' Theorem for extensive bicartesian fibrations of $(\infty,n)$-categories for $1 \le n \le \infty$ (though not formulated in type theory). However, finding the correct definitions in terms of (generalized) co-/cartesian cells might take some work.

\subsection{Structure, contributions, and related work}

This text essentially consists of parts of the author's PhD dissertation, namely~\cite[Section~3.4 and Chapter~5]{jw-phd} plus some additional material. All of this can be found in~\Cref{ssec:more,ssec:gen-moens,ssec:zawadowski}. In~\Cref{sec:syn-fib} we recall the basics of synthetic fibered $\inftyone$-category theory as found in~\cite[Sections~2 and~5]{BW21} and~\cite[Chapter~3]{jw-phd}. The basis for these developments are provided by Riehl--Shulman's account of discrete fibrations of synthetic fibered $\inftyone$-categories~\cite[Section~8]{RS17} and Riehl--Verity's co-/cartesian fibrations in general $\infty$-cosmoses, \cf~\cite[Chapter~5]{RV21} and \cite{RVscratch,RVyoneda}. These were in turn deeply influenced by Gray~\cite{GrayFib} and Street~\cite{StrYon,StrBicat,StrBicatCorr}. A thorough development of an analytic account to (internal) co-/cartesian fibrations is given by Rasekh in~\cite{rasekh2021cartesian,Ras17Cart,RasYonDSp}.

Segal and Rezk spaces were studied by Rezk~\cite{rez01}, Joyal--Tierney~\cite{joyal2007quasi}, Lurie~\cite{lurie2009goodwillie}, Kazhdan--Varshavsky~\cite{KV14}, and Rasekh~\cite{rasekh2021cartesian}. Segal objects are treated by Boavida de Brito~\cite{dB16segal}, Stenzel~\cite{SteSegObj}, and Rasekh~\cite{Ras17Cart,RasUniv}. The ensuing treatments of lex, Beck--Chevalley, and (pre-)Moens fibrations in~\Cref{ssec:lex-fam,sec:bc-fam,sec:moens-fam} are adaptations of the work and presentation by Streicher~\cite[Sections~5 and 6]{streicher2020fibered} to the setting of synthetic $\inftyone$-categories. Several of these results can also be found in the diploma thesis by Lietz~\cite{LietzDip}. In~\Cref{sec:bicart-fam} we briefly discuss \emph{bicartesian families}, \ie~families that are both co- and contravariantly functorial. Of particular interest will be the Artin gluing of a functor between lex Rezk types, discussed in~\Cref{ssec:fam-fib}. Internal sums are discussed via the notion of \emph{Beck--Chevalley family} in~\Cref{sec:bc-fam}. A characterization of those families via cartesianness of the cocartesian transport functor is given by~\Cref{thm:cocartfams-via-transp}, \cf~also \cite[Chapter~3, Theorem~1/(iii)]{LietzDip} and~\cite[Theorem~4.4]{streicher2020fibered}. Next, we prove in~\Cref{prop:bcc-for-pb-pres} that a functor between Rezk types preserves pullback if and only if its gluing fibration has internal sums, after~\cite[Lemma~13.2]{streicher2020fibered}.

The main part of the paper is developed in~\Cref{sec:moens-fam} in the form of a thorough discussion of Moens \aka~\emph{extensive} families. We provide characterizations of disjointness of stable internal sums in~\Cref{prop:char-disj-stable}. A characterization of extensivity in the presence of internal sums is given in~\Cref{prop:ext-sums}. For the most part, this goes through even in the absence of internal sums. The goal of these and further developments in~\Cref{ssec:ext} is then a synthetic version of \emph{Moens' Theorem} in~\Cref{ssec:moens-thm}, \Cref{thm:moens-thm}, which establishes an equivalence between the type of small Moens fibrations over a lex Rezk base $B$ and the type of lex functors from $B$ to some other small lex Rezk type. Classically, this can be found in~\cite[Theorem~15.18]{streicher2020fibered} and \cite[Subsection~4.2, Theorem~5, and Chapter~5, Proposition~12]{LietzDip}. The result gets generalized in~\Cref{ssec:gen-moens}, \Cref{thm:gen-moens-thm}, to the setting where the Beck--Chevalley condition is not necessary, corresponding to the functor from the base only preserving the terminal element but not necessarily pullbacks. This generalization originally is due to Streicher~\cite{streicher-simpl-moens,streicher2020fibered}.

We conclude by showing in~\Cref{ssec:zawadowski} that Zawadowski's notion of \emph{cartesian bifibrations}~\cite{zawadowski-lax-mon} coincides with this generalized notion of Moens fibration. This has been also observed by Streicher~\cite{streicher-cartbifib}, and we provide a slight complementation of his proof, and argue that everything goes through in the setting at hand.

\subsection{Further related work}

An account to Moens \aka~extensive fibrations can be found in~Jacob's textbook~\cite{jacobs-cltt}. The \emph{fibered viewpoint of geometric morphisms} builds crucially on this and has been comprehensively developed by Streicher~\cite{streicher2020fibered,streicher-cartbifib,StrFVGM} and Lietz~\cite{LietzDip}, building on central results by Moens~\cite{MoensPhD}. The material in~\cite{streicher2020fibered}, commenced 1999, is based on B\'{e}nabou's notes~\cite{ben-cat-fib,ben-catlog}. Central aspects of his philosophy are laid out in his essay \emph{Fibered categories and the foundations of naive category theory}~\cite{benabou-foundations}. These perspectives are also reflected in the recent online book by~Sterling--Angiuli~\cite{SterAng-RelCT}. In the context of $\inftyone$-categories Stenzel has adapted some concepts from B\'{e}nabou's work~\cite{SteCompr}. In logic, the fibered view of geometric morphisms has furthermore been used in realizability by Frey~\cite{FreyPhD,FreyMoensFib} and Frey--Streicher~\cite{FreStr-Tripos}, and in categorical modal logic by Doat~\cite{DoatMSc}. Initially, the theory of Grothendieck \aka~cartesian fibrations of $\inftyone$-categories (implemented as quasi-categories) was developed by Joyal~\cite{joyal2007quasi} and Lurie~\cite{LurHTT}. Follow-up groundlaying work on fibrations of (internal) $\inftyone$-categories has been done notably by~Ayala--Francis~\cite{AFfib}, Boavida de Brito~\cite{dB16segal}, Barwick--Dotto--Glasman--Nardin--Shah~\cite{barwick2016}, Barwick--Shah~\cite{BarwickShahFib}, Mazel-Gee~\cite{MazGee-UserCart}, Rezk~\cite{rezk2017stuff}, Cisinski~\cite{CisInfBook}, and Nguyen~\cite{NguyPhD}. Recently, a model-independent theory of internal $\inftyone$-categories and \\ co-/cartesian fibrations internal to an $\inftyone$-topos has been under development by Martini~\cite{mar-yon,MarCocart}, and Martini--Wolf~\cite{mw-lim}. Directed homotopy type theories have been proposed and suggested in various settings by Warren~\cite{War-DTT-IAS}, Licata--Harper~\cite{LH2DTT}, Nuyts~\cite{NuyMSc}, North~\cite{NorthDHoTT}, and Kavvos~\cite{KavQuant}. Parallel approaches in the context of \emph{two-level} type theories have been developed by Voevodsky~\cite{VV-HTS}, Capriotti~\cite{CapriottiPhD}, and Annenkov--Capriotti--Kraus--Sattler~\cite{2ltt}. This is put in a wider perspective by Buchholtz~\cite{B19} in an essay on higher structures in homotopy type theory. In the setting of simplicial type theory an account to directed univalence has been given by Cavallo--Riehl--Sattler~\cite{CRS18}, and in a bicubical setting by Weaver--Licata~\cite{WL19}. A development of limits and colimits in simplicial HoTT has been established by Bardomiano Mart\'{i}nez~\cite{BM21}. Two-sided cartesian fibrations in this setting have been treated in~\cite{W22-2sCart}. A prototype proof assistant for simplicial type theory has been developed by Kudasov~\cite{KudRzk}.

\section{Synthetic fibered \texorpdfstring{$\inftyone$}{(∞,1)}-category theory}\label{sec:syn-fib}

\subsection{Synthetic \texorpdfstring{$\inftyone$}{(∞,1)}-categories}
We work in Riehl--Shulman's simplicial homotopy type theory~\cite{RS17}. This is an augmentation of standard homotopy type theory (HoTT)~\cite{hottbook} by simplicial shapes (such as the $n$-simplices $\Delta^n$, boundaries $\partial \Delta^n$, $(n,k)$-horns $\Lambda_k^n$, \ldots). HoTT has semantics in any given Grothendieck--Rezk--Lurie $\inftyone$-topos\footnote{more precisely, in a representing \emph{type-theoretic model topos}, \cf~\cite{Shu19,Riehl-Sem,Wei-StrExt}} $\mathscr E$, so the additional layer corresponds to the image of the category whose objects are finite cartesian powers of the interval $\I \defeq \Delta^1$ with morphisms all monotone maps, embedded into $\mathscr E^{\Simplex^{\Op}}$ in the categorical direction. Within the shape layer, we can reason about extensional equality of terms. This is reflected into the type layer, so that one can talk about strict \emph{extension types} defined in~\cite{RS17}, after Lumsdaine--Shulman. Given a shape inclusion $\Phi \subseteq \Psi$ (implicitly, in a common cube context $I$), a family $A: \Psi \to \UU$ of $\UU$-small types, and a partial section $a:\prod_{t:\Phi} A(t)$, we can sonsider the ensuing \emph{extension type}
\[ \exten{t:\Psi}{A(t)}{\Phi}{a}\]
of sections $b:\prod_{t:\Psi} A(t)$ such that $t:\Phi \vdash a(t) \jdeq b(t)$. This allows one for any type $A$ to define \eg~the \emph{hom-type}
\[ \hom_A(a,b) \defeq (a \to_A b) \defeq \ndexten{\Delta^1}{A}{\partial \Delta^1}{[a,b]}\]
for terms $a,b:A$. In the case of a family $P:A \to \UU$ we can analogously define the \emph{dependent hom-type}
\[ \hom_u^P(d,e) \defeq (d \to_u^P e) \defeq \exten{t:\Delta^1}{P(u(t))}{\partial \Delta^1}{[d,e]}\]
for $u:a \to_A b$ and $d:P\,a$, $e:P\,b$. This formalism allows one to conveniently express the Segal and Rezk completeness conditions, \cf~\cite[Sections~5 and 10]{RS17}: A type $A$ is \emph{Segal} if the map
\[ A^\iota : A^{\Delta^2} \to A^{\Lambda_1^2}\]
induced from the inclusion $\iota: \Lambda_1^2 \subseteq \Delta^2$ is an equivalence, \ie~its fibers are homotopy propositions. A Segal type $A$ is \emph{Rezk} if the canonical map
\[ \idtoiso_A:\prod_{x,y:A} (x=_A y) \to (x \cong_A y)\]
defined by path induction is an equivalence. This is precisely the \emph{(local) univalence} condition for $A$~\aka~\emph{Rezk completeness}. Since it is convenient for our purposes to also consider maps whose domain is a type while the codomain is a shape we coerce the shapes to be types as well (making use of the additional strict layer as needed), \cf~\cite[Section~2.4]{BW21}. For a comprehensive introduction to simplicial homotopy type theory \cf~\cite[Sections~2 and~3]{RS17}, or the more compact overview in~\cite[Section~2]{BW21}. The basics of plain homotopy type theory are laid out in~\cite{hottbook,AwoTT,Grayson-UF,ShuLogSp,RijIntro,RieHoTT}.

\subsection{Cocartesian families}

Cocartesian families capture the idea of \emph{covariantly functorial} type families. Namely, let $B$ be a Rezk type and $P:B \to \UU$ be a 
family of Rezk types. If $P$ is cocartesian, then any arrow $u:a \to_B b$ in $B$ induces a transport functor $u_!: P\,a \to P\,b$. 
Cocartesian families are precisely the type-theoretic analogue of cocartesian fibrations which have been well-studied semantically. 
In fact, since in HoTT, type families over a type $B$ correspond to maps with codomain $B$ (\cf~\cite[Theorem~4.8.3]{hottbook} and~\cite[Section~2.5]{BW21}) 
a notion of cocartesian family simultaneously determines a notion of cocartesian fibration, and \emph{vice versa}. Given a family $P:B \to \UU$ (over any type $B$) we write the \emph{unstraightening of $P$} as
\[ \Un_B(P) \defeq \pair{\totalty{P}}{\pi_P}\]
where $\pi_P \defeq \totalty{P} \to B$ denotes the projection from the total type $\totalty{P} \defeq \sum_{b:B} P\,b$. Conversely, given a map $\pi:E \to B$, the \emph{straightening of $P$} is given by the family
\[ \St_B(\pi) \defeq \lambda b.\fib(\pi,b)\]
of homotopy fibers $\fib(\pi,b) \defeq \sum_{e:E} \pi(e) =_B b$. We will often denote a map $\pi:E \to B$ as $\pi:E \fibarr B$, particularly in cases where it actually happens to be a 
functorial fibration (cartesian or cocartesian). In HoTT, any type family $P:B \to \UU$ transforms covariantly \wrt~undirected paths $p:a=_Bb$, and the cocartesian families over 
Rezk types precisely generalize this to the case of directed arrows.

\subsubsection{Cocartesian families}

We want to consider families (or equivalently fibrations) that admit a functorial transport operation. For this to make sense, the underlying structure should be as categorical as possible, meaning that the base type, the total type, and all the fibers should be Rezk. This is captured by what we call \emph{isoinner families}, \ie, type families satisfying the \emph{relative version} of the Segal and Rezk conditions. Note that, in general, if a map is right orthogonal to some fixed map or shape inclusion, all its fibers are, too~\cite[Corollary~3.1.10]{BW21}. Furthermore, any map between types that are both right orthogonal to a given type map or shape inclusion is itself orthogonal to that given map or shape inclusion~\cite[Proposition~3.1.1]{BW21}. Maps between local types are necessarily local, see~\cite[Proposition~3.1.15]{BW21}.

\begin{definition}[(Iso)inner families, \protect{\cite[Def.~4.1.1, Prop.~4.1.2, Def.~4.2.3]{BW21}}]
	Let $P: B \to \UU$ be a family. 
	\begin{enumerate}
		\item We call $P$ \emph{inner} if the square
		\[\begin{tikzcd}
			{\totalty{P}^{\Delta^2}} & {\totalty{P}^{\Lambda_1^2}} \\
			{B^{\Delta^2}} & {B^{\Lambda_1^2}}
			\arrow[from=1-1, to=2-1]
			\arrow[from=1-1, to=1-2]
			\arrow[from=1-2, to=2-2]
			\arrow[from=2-1, to=2-2]
		\end{tikzcd}\]
		is a pullback, induced by restricting the projection $\pi : \totalty{P} \to B$ along the inclusion $\Lambda_1^2 \hookrightarrow \Delta^2$.
		\item An inner family is \emph{isoinner} if every fiber is a Rezk type.
	\end{enumerate}
	
\end{definition}

\begin{prop}[\protect{\cite[Prop.~4.1.4, 4.1.5, and 4.2.6]{BW21}}]
	Let $P: B \to \UU$ be a family.
	\begin{enumerate}
		\item Let $P$ be inner. Then every fiber $P\,b$ is inner, for $b:B$. If $B$ is Segal, then the total type $\totalty{P}$ is Segal.
		\item Let $P$ be isoinner. If $B$ is Rezk, then the total type $\totalty{P}$ is Rezk.
	\end{enumerate}
\end{prop}

Thus, isoinner families are families of synthetic $\inftyone$-categories, but they in general do not admit functorial transport. This motivates our notion of \emph{(co)cartesian} family. Cocartesian families can be defined in terms of the existence of liftings of arrows in the base to \emph{cocartesian arrows} lying over 
them (with prescribed source vertex). This is reminiscent of the classical $1$-categorical notion of 
\emph{Grothendieck fibration}~\cite{streicher2020fibered}, but pertinent to a homotopical context. It constitutes a type-theoretic version of the 
usual notions of Grothendieck or (co-)cartesian fibration of $\inftyone$-categories~\cite{JoyQcat,LurHTT,RVyoneda,RV21,rasekh2021cartesian}.

\begin{definition}[Cocartesian arrow, \protect{\cite[Def.~5.1.1]{BW21}, \cf~\cite[Def.~5.4.1]{RV21}}]\label{def:cocart-arr}
	Let $P:B \to \UU$ an isoinner family over a Rezk type $B$. An arrow $f:e \to_u^P e'$ over an arrow $u: b \to_B b'$ is \emph{$P$-cocartesian} if and only if
	\[\isCocartArr_u^P(f) \defeq \prod_{b'':B} \prod_{v:b' \to b''} \prod_{e'':P\,b''} \prod_{h:e \to^P_{vu} e''} \isContr\Big( \sum_{g: e' \to^P_v e''} h =_{vu}^P gf \Big).\]
	Given an arrow $u: b \to_B b'$ in the base $B$, points $e:P\,b$ and $e':P\,b'$ in the fibers, we denote the type of \emph{$P$-cocartesian arrows over $u$} from $e$ to $e'$ as
	\[ (e \cocartarr^P_u e') \defeq \sum_{f:e \to^P_u e'} \isCocartArr_u^P(f). \]
\end{definition}

\begin{definition}[Cocartesian family, \protect{\cite[Def.~5.2.2]{BW21}}, \cf~\protect{\cite[Def.~5.4.2]{RV21}}]
	For a Rezk type $B$, an isoinner family $P:B \to \UU$ is \emph{cocartesian} if and only if for $u:b \to_B b'$ and $e:P\,b$ the type
	\[ \sum_{e':P\,b'} (e \cocartarr^P_u e')  \]
	is inhabited.
\end{definition}

In fact, cocartesian arrows with a fixed starting point are determined uniquely up to homotopy, implying that cocartesian lifts in the above sense are determined uniquely up to homotopy. For a cocartesian family $P:B \to \UU$, given arrows $u:b \to_B b'$, $v : b' \to_B b''$ in the base, a cocartesian arrow $f : e \cocartarr_u^P e'$, and a dependent arrow $h : e \to^P_{vu} e''$, we denote the unique filler $g : e' \to^P_v e''$ such that $h = g \circ^P f$ as
\[ g \defeq \tyfill^P_f(g).\]
There exists another characterization of cocartesian families precisely in terms of the transport operation mentioned earlier.\footnote{Yet another characterization in terms of the existence of a left adjoint right inverse map to the projection $E^{\Delta^1} \to B^{\Delta^1} \times_B E$, serving as the cocartesian \emph{lifting} map, is given in~\cite[Theorem~5.2.6]{BW21} after~\cite[Proposition~5.2.8(ii)]{RV21}. This is known as the \emph{Chevalley criterion}.} 

\begin{theorem}[Cocart.~families via transport, \protect{\cite[Thm.~5.2.7]{BW21}}, \cf~\protect{\cite[Pr~~5.2.8(ii)]{RV21}}]\label{thm:cocartfams-via-transp}
	Let $B$ be a Rezk type, and $P:B \to \UU$ an isoinner family with associated total type projection $\pi:E \to B$.
	
	Then, $P$ is cocartesian if and only if the map
	\[ \iota \defeq \iota_P : E \to \commaty{\pi}{B}, \quad \iota \, \pair{b}{e} :\jdeq \pair{\id_b}{e}  \]
	has a fibered left adjoint $\tau \defeq \tau_P: \commaty{\pi}{B} \to E$ as indicated in the diagram:
	\[\begin{tikzcd}
		{} && E && {\pi \downarrow B} \\
		\\
		&&& B
		\arrow[""{name=0, anchor=center, inner sep=0}, "\iota"', curve={height=12pt}, from=1-3, to=1-5]
		\arrow["\pi"', two heads, from=1-3, to=3-4]
		\arrow["{\partial_1'}", two heads, from=1-5, to=3-4]
		\arrow[""{name=1, anchor=center, inner sep=0}, "\tau"', curve={height=12pt}, from=1-5, to=1-3]
		\arrow["\dashv"{anchor=center, rotate=-90}, draw=none, from=1, to=0]
	\end{tikzcd}\]
\end{theorem}

\subsubsection{Examples of cocartesian families}

\begin{prop}
	Let $g:C \to A \leftarrow B:f$ be a cospan of Rezk types. Then the codomain projection from the comma object
	\[\begin{tikzcd}
		{f \downarrow g} && {A^{\Delta^1}} \\
		{C \times B} && {A \times A} \\
		C
		\arrow[from=1-1, to=2-1]
		\arrow["{g \times f}"', from=2-1, to=2-3]
		\arrow[from=1-1, to=1-3]
		\arrow["{\langle \partial_1, \partial_0 \rangle}", from=1-3, to=2-3]
		\arrow["\lrcorner"{anchor=center, pos=0.125}, draw=none, from=1-1, to=2-3]
		\arrow[from=2-1, to=3-1]
		\arrow["{\partial_1}"', curve={height=22pt}, from=1-1, to=3-1]
	\end{tikzcd}\]
	is a cocartesian fibration.
\end{prop}

\begin{proof}
	Cf.~\cite[Proposition~5.2.15]{BW21}, or the version for fibered comma types,~\cite[Proposition~3.5]{W22-2sCart}.
\end{proof}

\begin{cor}[Codomain opfibration]\label{prop:cod-cocartfam}
	For any Rezk type $B$, the projection
	\[ \partial_1 : B^{\Delta^1} \to B, ~  \partial_1 :\jdeq \lambda f.f(1).\]
	is a cocartesian fibration, called the \emph{codomain opfibration}.
\end{cor}

The domain projection of a Rezk type is a cocartesian fibration given that the base has pushouts.
\begin{prop}[Domain opfibration]
	If $B$ is a Rezk type that has all pushouts, then the domain projection
	\[ \partial_0 : B^{\Delta^1} \to B, \quad \partial_0(u) :\jdeq u(0)\]
	is a cocartesian fibration.
\end{prop}

\begin{proof}
	Cf.~\cite[Proposition~3.2.10]{BW21}.
\end{proof}

\begin{theorem}[Free cocartesian family]
	Let $\pi:E \fibarr B$ be a map between Rezk types. Then the map $\partial_1': L(\pi) \fibarr B$ defined by
	\[\begin{tikzcd}
		{L(\pi)} && E \\
		{B^{\Delta^1}} && B \\
		B
		\arrow[from=1-1, to=1-3]
		\arrow[two heads, from=1-1, to=2-1]
		\arrow["{\partial_0}"', from=2-1, to=2-3]
		\arrow["\pi", two heads, from=1-3, to=2-3]
		\arrow["{\partial_1}", two heads, from=2-1, to=3-1]
		\arrow["\lrcorner"{anchor=center, pos=0.125}, draw=none, from=1-1, to=2-3]
		\arrow["{\partial_1'}"', curve={height=24pt}, two heads, from=1-1, to=3-1]
	\end{tikzcd}\]
	is a cocartesian family, the \emph{free cocartesian family on $\pi$} or the \emph{cocartesian replacement of $\pi$}. 
\end{theorem}

\begin{proof}
	Cf.~\cite[Theorem~5.2.19]{BW21}, as well as \cite[Theorem~4.3]{GHN17}, \cite[Lemma~3.3.1]{BarwickShahFib}.
\end{proof}

Indeed, the cocartesian replacement can be shown in our setting to satisfy the expected universal property (insofar we can formulate it in this type theory in its present state), \cf~\cite[Proposition~5.2.20]{BW21}, \cite[Theorem~4.5]{GHN17}, and~\cite[Corollary~3.3.4]{BarwickShahFib}. See also \cite{RVexp} for the general context of co-/monadicity.

\subsubsection{Cocartesian functors}
Maps between type families are given by families of maps between the fibers. These are also called \emph{fiberwise maps}. In case of isoinner fibrations these correspond to families of functors between the fibers, called \emph{fibered functor} in that case.

\begin{definition}[Fiberwise map]
	For types $A$ and $B$ consider families $P : B \to \UU$ and $ Q: A \to \UU$. A \emph{fiberwise map} or \emph{family of maps} is given by a pair
	\[ \Big \langle j:A \to B, \varphi : \prod_{a:A} Q\,a \to P\,j\,a \Big \rangle.\]
	In case $P$ and $Q$ are isoinner families over Rezk types $B$ and $A$, we call $\pair{j}{\varphi}$ a \emph{fibered functor}, and we write the corresponding type as
	\[ \FibMap_{A,B}(Q,P) \defeq \sum_{j:A \to B} \prod_{a:A} Q\,a \to P\,j\,a \]
\end{definition}

It is a standard result in homotopy type theory (see \cite[Definition~4.7.5, Theorem~4.7.6]{hottbook} and \cite[Definition~11.1.1, Lemma~11.1.2]{RijIntro}) that, for fixed $j\colon A \to B$, a family of maps $\varphi : \prod_{a:A} Q\,a \to P\,j\,a$ corresponds to a map $\Phi : \widetilde{Q} \to \widetilde{P}$ between the total types together with a homotopy witnessing that the following square between the total type projections commutes:
\[\begin{tikzcd}
	{\totalty{Q}} & {\totalty{P}} \\
	A & B
	\arrow[from=1-1, to=2-1]
	\arrow["j"', from=2-1, to=2-2]
	\arrow["\Phi", from=1-1, to=1-2]
	\arrow[from=1-2, to=2-2]
\end{tikzcd}\]

An appropriate notion of morphism between \emph{cocartesian} families is given by fibered functors that preserve the cocartesian arrows. We call those fibered functors \emph{cocartesian functors}.

\begin{definition}[Cocart.~functors, \protect{\cite[Def.~5.3.2]{BW21}}, \cf~\protect{\cite[Def.~5.3.2]{RV21}}]\label{def:cocart-fun}
	For Rezk types $A$ and $B$, let $Q \colon A \to \UU$ and $P : B \to \UU$ be cocartesian families. Given a fibered functor $\big \langle j:A \to B, \varphi : \prod_{a:A} Q\,a \to P\,j\,a \big \rangle$, we call it a \emph{cocartesian functor} if the map
	\[ \Phi: \totalty{Q} \to \totalty{P}, \quad \Phi \, a \, d :\jdeq \pair{j(a)}{\varphi_{j(a)}(d) } \]
	preserves cocartesian arrows, \ie, the following proposition\footnote{Note that being a cocartesian arrow is a proposition~\cite[Subsubsection~5.1.1]{BW21}.} is satisfied:
	\[ \isCocartFun_{P,Q}(\Phi) :\jdeq \prod_{\substack{u:\Delta^1 \to B \\ f: \Delta^1 \to u^*P}}  \isCocartArr^P_u(f) \to \isCocartArr^Q_{ju}(\varphi_u f).\]
	Here, $\varphi_u f \defeq \lambda t.\varphi_{u\,t}(f\,t)$ is the arrow defined by the action of the dependent function $\varphi$ on the arrow $u$.
	Accordingly, the type of cocartesian functors from $Q$ to $P$ is defined as
	\[ \CocartFun_{A,B}(Q,P) :\jdeq \sum_{\Phi:\FibMap_{A,B}(Q,P)} \isCocartFun_{Q,P}(\Phi).\]
	
	Given $P: B \to \UU$ and $Q: B \to \UU$ over the same base $B$ a \emph{fibered functor} from $Q \to P$ is a dependent function $\varphi : \prod_{x:B} Qx \to Px$. It is a cocartesian functor if and only if
	\[ \isCocartFun_{Q,P}(\pair{\id_B}{\varphi}) \equiv \prod_{\substack{u:\Delta^1 \to B \\ f: \Delta^1 \to u^*Q}} \isCocartArr^Q_u(f) \to \isCocartArr^P_{u}(\varphi_u f).\]
	
	Accordingly, we define the type of cocartesian functors over a common base as
	\[ \CocartFun_{B}(Q,P) :\jdeq \sum_{\varphi:\prod_B Q \to P} \isCocartFun_{Q,P}(\pair{\id_B}{\varphi}).\]
\end{definition}

Cocartesian families and cocartesian functors satisfy analogues of the familiar naturality properties, \cf~\cite[Propositions~5.2.4, 5.2.5, and~5.3.4]{BW21}. There exists a characterization of cocartesian functors analogous to~\Cref{thm:cocartfams-via-transp}:

\begin{theorem}[Characterizations of cocart.~functors, \protect{\cite[Thm.~5.3.19]{BW21}}, \cf~\protect{\cite[Thm.~5.3.4]{RV21}}]\label{thm:char-cocart-fun}
	Let $A$ and $B$ be Rezk types, and consider cocartesian families $P:B \to \UU$ and $Q:A \to \UU$ with total types $E\defeq \totalty{P}$ and $F\defeq \totalty{F}$, respectively.
	
	For a fibered functor\footnote{We are introducing a slight abuse of notation here, conflating between the dependent function and its totalization. We hope that this increases readability rather than cause confusion.} $\Phi\defeq \pair{j}{\varphi}$ giving rise to a square
	\[
	\begin{tikzcd}
		F \ar[r, "\varphi"] \ar[d, "\xi" swap] & E \ar[d, "\pi"] \\
		A \ar[r, "j" swap] & B
	\end{tikzcd}
	\]
	the following are equivalent:
	\begin{enumerate}
		\item The fiberwise map $\Phi$ is a cocartesian functor.
		\item The mate of the induced canonical fibered natural isomorphism is invertible, too:
		\[\begin{tikzcd}
			{F} && {E} & {} & {F} && {E} \\
			{\xi \downarrow A} && {\pi \downarrow B} & {} & {\xi \downarrow A} && {\pi \downarrow B}
			\arrow["{\iota}"', from=1-1, to=2-1]
			\arrow["{\varphi}", from=1-1, to=1-3]
			\arrow["{\iota'}", from=1-3, to=2-3]
			\arrow[Rightarrow, "{=}", from=2-1, to=1-3, shorten <=7pt, shorten >=7pt]
			\arrow["{\rightsquigarrow}" description, from=1-4, to=2-4, phantom, no head]
			\arrow["{\kappa}", from=2-5, to=1-5]
			\arrow["{\varphi}", from=1-5, to=1-7]
			\arrow["{\varphi \downarrow j}"', from=2-5, to=2-7]
			\arrow["{\kappa'}"', from=2-7, to=1-7]
			\arrow[Rightarrow, "{=}"', from=2-7, to=1-5, shorten <=7pt, shorten >=7pt]
			\arrow["{\varphi \downarrow j}"', from=2-1, to=2-3]
		\end{tikzcd}\]
	\end{enumerate}
	
\end{theorem}

Furthermore, cocartesian families and functors satisfy a lot of important closure properties. We can prove this in type theory mirroring \eg~their corresponding $\infty$-cosmological closure properties,~\cite[Definition~1.2.1 and Proposition~6.3.14]{RV21}.

\begin{proposition}[Cosmological closure properties of cocartesian families, \protect{\cite[Proposition~5.3.17]{BW21}}]\label{prop:cocart-cosm-closure}
	Over Rezk bases, it holds that:
	
	Cocartesian families are closed under composition, dependent products, pullback along arbitrary maps, and cotensoring with maps/shape inclusions. Families corresponding to equivalences or terminal projections are always cocartesian.
	
	Between cocartesian families over Rezk bases, it holds that:
	Cocartesian functors are closed under (both horizontal and vertical) composition, dependent products, pullback, sequential limits,\footnote{all three objectwise limit notions satisfying the expected universal properties \wrt~to cocartesian functors} and Leibniz cotensors.
	
	Fibered equivalences and fibered functors into the identity of $\unit$ are always cocartesian.
\end{proposition}

\subsubsection{More on cocartesian and vertical arrows}\label{ssec:more}

Let $P:B \to \UU$ be a cocartesian family over a Rezk $B$. Given an arrow $u:a \to_B b$ and a term $e:P\,a$ we denote the cocartesian lift, determined uniquely up to homotopy, as
\[ P_!(u,e): e \cocartarr_u^P u_!^P(e),\]
somtimes suppressing superscripts whenever they are clear from the context.

We recover the following naturality and functoriality results for cocartesian families and functors.

\begin{proposition}[Functoriality of cocartesian families, \protect{\cite[Proposition~5.2.4]{BW21}}]\label{prop:cocart-functoriality}
	Let $B$ be a Rezk type and $P:B \to \UU$ a cocartesian family. For any $a:B$ and $x:P\,a$ there is an identity
	\[ \coliftarr{P}{\id_b}{x} = \id_x, \]
	and for any $u:\hom_B(a,b)$, $v:\hom_B(b,c)$, there is an identity
	\[ \coliftarr{P}{v \circ u}{x} = \coliftarr{P}{v}{\coliftpt{u}{x}} \circ \coliftarr{P}{u}{x}. \]
\end{proposition}

\begin{proposition}[\protect{\cite[Proposition~5.2.5]{BW21}}]
	Let $B$ be a Rezk type and $P:B \to \UU$ be a cocartesian family. For any arrow $u:\hom_B(a,b)$ and elements $d:P\,a$, $e:P\,b$, we have equivalences between the types of (cocartesian) lifts of arrows (recall~\Cref{def:cocart-arr}) from $d$ to $e$ and maps (isomorphisms, respectively) from $u_!d$ to $e$:
	\[\begin{tikzcd}
		{(d} & {e)} && {(u_!d} & {e)} \\
		{(d} & {e)} && {(u_!d} & {e)}
		\arrow["\equiv", from=1-2, to=1-4]
		\arrow["\equiv", from=2-2, to=2-4]
		\arrow[""{name=0, anchor=center, inner sep=0}, "u"', from=1-1, to=1-2, cocart]
		\arrow[""{name=1, anchor=center, inner sep=0}, "{P\,b}" below, from=2-4, to=2-5]
		\arrow[""{name=2, anchor=center, inner sep=0}, "u"', from=2-1, to=2-2]
		\arrow[""{name=3, anchor=center, inner sep=0}, "{P\,b}" below,equals,  from=1-4, to=1-5]
		\arrow[shorten <=9pt, shorten >=6pt, hook, from=0, to=2]
		\arrow[shorten <=6pt, shorten >=9pt, hook, from=3, to=1]
	\end{tikzcd}\]
\end{proposition}

\begin{proposition}[Naturality of cocartesian liftings, \protect{\cite[Proposition~5.3.4]{BW21}}]\label{prop:nat-cocartlift-arr}
	Let $B$ be a Rezk type, $P:B \to \UU$, $Q:C \to \UU$ cocartesian families, and $\Phi \jdeq \pair{j}{\varphi} : \CocartFun_{B,C}(P,Q)$ a cocartesian functor. Then $\Phi$ commutes with cocartesian lifts, \ie,~for any $u:\hom_B(a,b)$ there is an identification of arrows\footnote{Here, $(ju)^*Q$ denotes the pulled back family $(ju)^*Q \defeq \lambda t.Q(j(u(t))) : \Delta^1 \to \UU$.}
	\[ \varphi \big(P_!(u,d)\big) =_{\Delta^1 \to (ju)^*Q} Q_!(ju,\varphi_ad) \]
	and hence of endpoints
	\[ \varphi_b(u_!^Pd) =_{Q(jb)} (ju)_!^Q(\varphi_ad). \]
	In particular there is a homotopy commutative square:
	\[
	\begin{tikzcd}
		Pa \ar[rr, "\varphi_a"] \ar[d, "\coliftptfammap{P}{u}" swap] & & Qa \ar[d, "\coliftptfammap{Q}{(ju)}"] \\
		Pb \ar[rr, "\varphi_b" swap] && Qb
	\end{tikzcd}
	\]
\end{proposition}

If $f$ is a cocartesian arrow and $h$ some dependent arrow with the same source, we denote the homotopically unique filler $g$ such that $h = gf$ by $g \jdeq \tyfill^P_f(h)$. This gives the following picture:
\[\begin{tikzcd}
	E &&& {u_!\,e} \\
	&& e && {e''} \\
	&&& {b'} \\
	B && b && {b''}
	\arrow[two heads, from=1-1, to=4-1]
	\arrow["vu", from=4-3, to=4-5]
	\arrow["u", from=4-3, to=3-4]
	\arrow["v", from=3-4, to=4-5]
	\arrow["h"', from=2-3, to=2-5]
	\arrow["f", from=2-3, to=1-4, cocart]
	\arrow["{g \defeq \tyfill_{f}(h)}", dashed, from=1-4, to=2-5]
\end{tikzcd}\]

\begin{definition}[Vertical arrows]
	Let $P:B \to \UU$ a family of Rezk types over a Rezk type $B$. A dependent arrow $f$ in $P$ is called \emph{vertical} if $\pi(f)$ is an isomorphism in $B$, where $\pi : \totalty{P} \to B$ denotes the unstraightening of $P$. For $d,e:P\,a$ we write $(d \vertarr^P_a e) \defeq (d \to_{P\,a} e)$ for the type of vertical arrows from $d$ to $e$, or simply just $(d \vertarr e)$ leaving the underlying data implicit.{\tiny {\tiny }}
\end{definition}

Note that by the Rezk condition any biinvertible arrow can be replaced by a (constant) path, \cf~\cite[Proposition~4.2.2]{BW21}. Therefore, the type of vertical arrows in $\totalty{P}$ is equivalent to the type $\sum_{b:B} \Delta^1 \to P\,b$, \cf~\cite[Definition~2.7.1]{jw-phd} or~\cite[Definition~2.13]{W22-2sCart}. Recall also several closure properties of cocartesian arrows that will be important throughout the text.

\begin{prop}[Closedness under composition and right cancelation, \protect{\cite[Prop.~5.1.8]{BW21}}, \cf~\protect{\cite[Lem.~5.1.5]{RV21}}]\label{prop:cocart-arr-closure} Let $P: B \to \UU$ be a cocartesian family over a Rezk type $B$. For arrows $u:\hom_B(b,b')$, $v:\hom_B(b',b'')$, with $b,b',b'':B$, consider dependent arrows $f:\dhom_{P,u}(e,e')$, $g:\dhom_{P,v}(e',e'')$ lying over, for $e:P\,b$, $e':P\,b'$, $e'':P\,b''$.
	\begin{enumerate}
		\item\label{it:cocart-arr-comp} If both $f$ and $g$ are are cocartesian arrows, then so is their composite $g \circ f$.
		\item\label{it:cocart-arr-cancel} If $f$ and $g \circ f$ are cocartesian arrows, then so is $g$.
	\end{enumerate}
\end{prop}

\begin{lem}[\protect{\cite[Prop.~5.1.9]{BW21}}, \cf~\protect{\cite[Lem.~5.1.6]{RV21}}]\label{lem:cocart-arrows-isos}
	Let $P: B \to \UU$ be an inner family over a Segal type $B$.
	\begin{enumerate}
		\item\label{it:depisos-are-over-isos} If $f$ is a dependent isomorphism in $P$ over some morphism $u$ in $B$, then $u$ is itself an isomorphism.
		\item\label{it:depisos-are-cocart} Any dependent isomorphism in $P$ is cocartesian.
		\item\label{it:cocart-arrows-over-ids-are-isos} If $f$ is a cocartesian arrow in $P$ over an identity in $B$, then $f$ is an isomorphism.
	\end{enumerate}
\end{lem}

Moreover, in a cocartesian family vertical arrows are stable under pullback along arbitrary arrows, given they exist. To prove this, we will need some preliminary results. Note that these could presumably be stated in our setting more abstractly introducing a notion of (orthogonal) factorization system, and showing that in a cocartesian family cocartesian arrows as a left class and vertical arrows as a right class yield an instance. However, we will not need the full power of this here so we refrain from developing this result, but see the works by Joyal~\cite[Paragraph~24.13]{JoyQcat}, Lurie~\cite[Example~5.2.8.15]{LurHTT}, Myers~\cite[Proposition~2.4]{MyersCartFS}, as well as the nLab~\cite[3.~Examples]{nLabOFS}, \cite[Proposition~3.1]{nLabWFS}, and Joyal's CatLab~\cite[6.~Examples: More examples]{JoyCatFS}.

\begin{lemma}[Retract closedness of vertical arrows]\label{lem:vertarr-retr}
	Let $\pi:E \fibarr B$ be a map between Rezk types. Then vertical arrows are closed under retracts, meaning: Given a vertical dependent arrow $f: x \vertarr^P y$ in $P \defeq \St_B(\pi)$, then a dependent arrow $f':x ' \to^P y'$ is vertical if there exist dependent squares as follows:
	\[\begin{tikzcd}
		{x'} && x && {x'} \\
		{y'} && y && {y'}
		\arrow["{f'}"', from=1-1, to=2-1]
		\arrow["g", from=1-1, to=1-3]
		\arrow["{g'}", from=2-1, to=2-3]
		\arrow["f"', squiggly, from=1-3, to=2-3]
		\arrow["k", from=1-3, to=1-5]
		\arrow["{k'}", from=2-3, to=2-5]
		\arrow["{f'}", from=1-5, to=2-5]
		\arrow["{\id_x}"{description}, curve={height=-18pt}, Rightarrow, no head, from=1-1, to=1-5]
		\arrow["{\id_y}"{description}, curve={height=18pt}, Rightarrow, no head, from=2-1, to=2-5]
	\end{tikzcd}\]
\end{lemma}

\begin{proof}
	Since $f$ is vertical and by functoriality of $\pi$ we can assume the above diagram to lie over a square in $B$ as follows:
	\[\begin{tikzcd}
		{a'} && a && {a'} \\
		{b'} && a && {b'}
		\arrow["{u'}"', from=1-1, to=2-1]
		\arrow["v", from=1-1, to=1-3]
		\arrow["{v'}", from=2-1, to=2-3]
		\arrow[Rightarrow, no head, from=1-3, to=2-3]
		\arrow["w", from=1-3, to=1-5]
		\arrow["{w'}", from=2-3, to=2-5]
		\arrow["{\id_{a'}}"{description}, curve={height=-18pt}, Rightarrow, no head, from=1-1, to=1-5]
		\arrow["{\id_{b'}}"{description}, curve={height=18pt}, Rightarrow, no head, from=2-1, to=2-5]
		\arrow["{u'}", from=1-5, to=2-5]
	\end{tikzcd}\]
	It suffices to show that $u':a' \to_B b'$ is an isomorphism. Indeed, from the assumption we obtain chains of identifications
	\[ u'(wv') = (u'w)v' = w'v' = \id_{b'}\]
	and
	\[ (wv')u' = w(v'u') = wv = \id_{a'}.\]
\end{proof}

In the following proposition we characterize the vertical arrows in a cocartesian family as (merely) having right liftings against cocartesian arrows.

\begin{proposition}[Right lifting property of vertical arrows]\label{prop:vert-rlp}
	Let $P:B \to \UU$ be a family of Rezk types over a small Rezk type $B$. If $P$ is cocartesian, then a dependent arrow $f:x \to^P_c y$, $c:B$, in $P$ is vertical if and only it \emph{merely has a right lift against every cocartesian arrow}, \ie, if and only if the type
	\[ \prod_{\substack{a,b,c:B \\ u:a \to_B b \\ v:b \to_B c}} \prod_{\substack{x': P\,a \\ y':P\,b}} \prod_{f':x' \cocartarr^P_u y'} \prod_{\substack{g:y' \to^P_v y \\ g': x' \to^P_{vu} x}} \Big \lVert \sum_{h:y' \to^P_{v} x} (hf' = g') \times (fh = g) \Big \rVert_{-1} \]
	is inhabited.\footnote{In fact, if $f$ is vertical the type $\sum_{h:y' \to^P_{v} x} (hf' = g') \times (fh = g)$ of lifts can be shown to be contractible (since $f'$ is cocartesian). Thus, the lifting property could be strengthened to \emph{unique lifting up to homotopy}, \cf~\cite[Definition~5.2.8.1 and Example~5.2.8.15]{LurHTT}.}
\end{proposition}

\begin{proof}
	Consider an arrow $u:a \to_B b$ and a term $c:C$. Let $f:x \vertarr^P_c y$ be vertical and $f':x' \cocartarr^P_u y'$ be cocartesian, for $x:P\,c$, $x':P\,a$, $y:P\,c$, and $y':P\,b$
	
	For any arrow $v:b \to_B c$ together with dependent arrows $g:y' \to^P_v y$, $g': x' \to^P_{vu} x$ we want to find a filler $h:y' \to^P_v x$ as indicated, \ie~such that $hf' = g'$ and $fh = g$:
	\[\begin{tikzcd}
		& {x'} && x \\
		{\totalty{P}} & {y'} && y \\
		& a && c \\
		B & b && c
		\arrow["{f'}"', from=1-2, to=2-2, cocart]
		\arrow["{g'}", from=1-2, to=1-4]
		\arrow["g", from=2-2, to=2-4]
		\arrow["f", squiggly, from=1-4, to=2-4]
		\arrow["u"', from=3-2, to=4-2]
		\arrow["v"', from=4-2, to=4-4]
		\arrow["vu", from=3-2, to=3-4]
		\arrow[Rightarrow, no head, from=3-4, to=4-4]
		\arrow[two heads, from=2-1, to=4-1]
		\arrow["h"{description}, dashed, from=2-2, to=1-4]
	\end{tikzcd}\]
	Since $f'$ is cocartesian, we readily obtain $h := \tyfill_{f}(g'): y' \to_v x$ such that $hf' = g$. To get an identification for $fh = g$, we note that 
	\[ (fh)f' = f(hf') = fg' = gf'\]
	But this means $fh = g$ since $f'$ is cocartesian. This proves one direction.
	
	Conversely, let $h:x \to^P_u z$ be a dependent arrow over $u:a \to_B c$, $a,c:B$ which merely has a lift against each cocartesian arrow. Since $P$ is cocartesian we obtain a factorization as follows, together with a lift $k$ as indicated:
	\[\begin{tikzcd}
		& x && x \\
		{\totalty{P}} & y && z \\
		& a && a \\
		B & c && c
		\arrow["f"', from=1-2, to=2-2, cocart]
		\arrow["g"', squiggly, from=2-2, to=2-4]
		\arrow[Rightarrow, no head, from=1-2, to=1-4]
		\arrow["h", from=1-4, to=2-4]
		\arrow[Rightarrow, no head, from=3-2, to=3-4]
		\arrow["u", from=3-4, to=4-4]
		\arrow[Rightarrow, no head, from=4-2, to=4-4]
		\arrow["u"', from=3-2, to=4-2]
		\arrow[two heads, from=2-1, to=4-1]
		\arrow["k"{description}, dashed, from=2-2, to=1-4]
		\arrow["{\pi(k)}"{description}, from=4-2, to=3-4]
	\end{tikzcd}\]
	 We then find the following:
	 \[\begin{tikzcd}
	 	x && y && x \\
	 	z && z && z
	 	\arrow["f" description, from=1-1, to=1-3]
	 	\arrow["k" description, from=1-3, to=1-5]
	 	\arrow["h"', from=1-1, to=2-1]
	 	\arrow[Rightarrow, no head, from=2-1, to=2-3]
	 	\arrow[Rightarrow, no head, from=2-3, to=2-5]
	 	\arrow["g"', squiggly, from=1-3, to=2-3]
	 	\arrow["h", from=1-5, to=2-5]
	 	\arrow["{\id_x}"{description}, curve={height=-18pt}, Rightarrow, no head, from=1-1, to=1-5]
	 	\arrow["{\id_z}"{description}, curve={height=18pt}, Rightarrow, no head, from=2-1, to=2-5]
	 \end{tikzcd}\]
 	So $h$ is a retract of the vertical arrow $g$, hence is itself vertical by~\Cref{lem:vertarr-retr}.
\end{proof}

This setup is enough to prove now the desired result: in a cocartesian family, whenever the pullback of a vertical arrow exists it is vertical as well.

\begin{theorem}[Pullback stability of vertical arrows in a cocartesian family]\label{thm:pb-of-vert-is-vert-cocart-fam}
	Let $B$ be a Rezk type and $P: B \to \UU$ a cocartesian family. Consider a vertical arrow $f:x \vertarr y$ in $P$. If $g:y' \to y$ is a dependent arrow such that the pullback $f' \defeq g^*f$ exists, then $f'$ is vertical, too:
	\[\begin{tikzcd}
		x' && x \\
		{y'} && y
		\arrow["{f'}"', squiggly, from=1-1, to=2-1]
		\arrow["g"', from=2-1, to=2-3]
		\arrow["{g'}", from=1-1, to=1-3]
		\arrow["f", squiggly, from=1-3, to=2-3]
		\arrow["\lrcorner"{anchor=center, pos=0.125}, draw=none, from=1-1, to=2-3]
	\end{tikzcd}\]
\end{theorem}

\begin{proof}
 By~\Cref{prop:vert-rlp} it suffices to show that $f'$ (merely) has a right lifting against any cocartesian arrow $f'': z' \cocartarr z$, given arrows $z: k\to y'$ and $k':z' \to x'$:
 \[\begin{tikzcd}
 	{z'} && {x'} \\
 	z && {y'}
 	\arrow["{k'}", from=1-1, to=1-3]
 	\arrow["{f''}"', from=1-1, to=2-1, cocart]
 	\arrow["k"', from=2-1, to=2-3]
 	\arrow["{f'}", from=1-3, to=2-3]
 	\arrow[dashed, from=2-1, to=1-3]
 \end{tikzcd}\]
First, \Cref{prop:vert-rlp} gives us a lift $h$ of the composite outer diagram in:
\[\begin{tikzcd}
	z' && {x'} && x \\
	\\
	z && {y'} && y
	\arrow["{k'}", from=1-1, to=1-3]
	\arrow["{f''}"'{pos=0.3}, from=1-1, to=3-1, cocart]
	\arrow["k"', from=3-1, to=3-3]
	\arrow["{f'}"{pos=0.3}, from=1-3, to=3-3]
	\arrow["g"', from=3-3, to=3-5]
	\arrow["f"{pos=0.3}, squiggly, from=1-5, to=3-5]
	\arrow["{g'}", from=1-3, to=1-5]
	\arrow["\lrcorner"{anchor=center, pos=0.125}, draw=none, from=1-3, to=3-5, crossing over]
	\arrow["h" swap, curve={height=6pt}, dashed, from=3-1, to=1-5, crossing over]
\end{tikzcd}\]
Consider now the induced arrow $m: z \to x'$:
\[\begin{tikzcd}
	z \\
	& {x'} && x \\
	& {y'} && y
	\arrow[from=2-2, to=2-4]
	\arrow["g"', from=3-2, to=3-4]
	\arrow["f", squiggly, from=2-4, to=3-4]
	\arrow["h", curve={height=-12pt}, from=1-1, to=2-4]
	\arrow["k"', curve={height=6pt}, from=1-1, to=3-2]
	\arrow["m", dashed, from=1-1, to=2-2]
	\arrow["{f'}", from=2-2, to=3-2]
	\arrow["\lrcorner"{anchor=center, pos=0.125}, draw=none, from=2-2, to=3-4]
\end{tikzcd}\]
We claim that $m$ is the desired lift. By definition, it satisfies $f'm = k$, so we only have to check the condition $mf'' = k'$. To obtain a witness for this, consider the two pullbacks:
\[\begin{tikzcd}
	z &&&& z \\
	& {x'} && x && {x'} && x \\
	& {y'} && y && {y'} && y
	\arrow["{g'}", from=2-2, to=2-4]
	\arrow["{f'}"', from=2-2, to=3-2]
	\arrow["g"', from=3-2, to=3-4]
	\arrow["f", from=2-4, to=3-4]
	\arrow["{f'}"', from=2-6, to=3-6]
	\arrow["g"', from=3-6, to=3-8]
	\arrow["{g'}", from=2-6, to=2-8]
	\arrow["f", from=2-8, to=3-8]
	\arrow["{g'mf''}", curve={height=-12pt}, from=1-1, to=2-4]
	\arrow["{kf''}"', curve={height=12pt}, from=1-1, to=3-2]
	\arrow["{kf''}"', curve={height=12pt}, from=1-5, to=3-6]
	\arrow["{g'k'}", curve={height=-12pt}, from=1-5, to=2-8]
	\arrow[dashed, from=1-1, to=2-2]
	\arrow[dashed, from=1-5, to=2-6]
	\arrow["\lrcorner"{anchor=center, pos=0.125}, draw=none, from=2-6, to=3-8]
	\arrow["\lrcorner"{anchor=center, pos=0.125}, draw=none, from=2-2, to=3-4]
\end{tikzcd}\]
We have $fg' mf'' = fg' k'$. But then the universal property of the pullbacks implies $mf'' = k'$ as desired.
\end{proof}

\subsection{Cartesian families}

Completely dually, one can formulate a theory of \emph{cartesian families} which are \emph{contravariantly functorial} \wrt~to directed paths. That is, for $P:B \to \UU$ a cartesian family, for any arrow $u:b \to a$ and $d:P\,a$ there exists a \emph{cartesian} lift $f \defeq P^*(u,d): u^*d \cartarr d$, satisfying the dual universal property: For any $v:c \to a$, and any $h:e \to^P_{uv} d$ there exists a filler $g:e \to^P_v u^*\,d$, uniquely up to homotopy, \st~$h=P^*(u,d) \circ g$. In particular, this induces a map
\[ u^*:P\,a \to P\,b.\]
Likewise, we have a notion of \emph{cartesian functor}. The fibered adjoint characterization (respectively, the Chevalley condition) turn out to be a \emph{right} adjoint (respectively, right adjoint right inverse (RARI)) condition instead. Furthermore, the cartesian arrows are pullback stable, and any dependent arrow factors as $(\cdot \rightsquigarrow \cdot \cartarr)$. Sometimes, we will distinguish in the notation between cartesian and cocartesian filling by writing $\cartFill_{\ldots}(\ldots)$ or $\cocartFill_{\ldots}(\ldots)$, respectively. As already introduced, vertical arrows (respectively their types) are denoted by a squiggly arrow $\cdot \rightsquigarrow \cdot$.

\begin{proposition}[Pullback stability of cartesian arrows]
	Let $P:B \to \UU$ be a cartesian family over a Rezk type $B$. Then cartesian arrows are stable under pullback, \ie~whenever $f:e' \cartarr^P e$ is a cartesian arrow in $P$, then whenever for another dependent arrow $g:e'' \to^P e$ the pulled back arrow $f' \defeq g^*f: e''' \to^P e''$ exists it is cartesian, too.
\end{proposition}

\begin{proof}
	Consider $f$, $g$, and $g'$ as given by the assumption. Let $d:h \to e''$ be some dependent arrow in $P$. Let $k \jdeq \tyfill_f(gh):d \to e'$ be the filler induced from cartesianness of $f$. This determines a cone $[d,h,k]$ over the cospan $[g,f]$. By assumption, $[e''',f',g']$, with $g' \defeq f^*g:e''' \to e'$ is a pullback, \ie, a terminal cone over $[g,f]$, there exists, uniquely up to homotopy, a gap arrow $m:d \to e'''$ with $g' m = k$ and $f'm=h$ as indicated:
	\[\begin{tikzcd}
		d \\
		& {e'''} && {e'} \\
		& {e''} && e
		\arrow["{f'}", from=2-2, to=3-2]
		\arrow["g"', from=3-2, to=3-4]
		\arrow["f", from=2-4, to=3-4, cart]
		\arrow["{k \defeq \tyfill_f(gh)}", curve={height=-12pt}, dashed, from=1-1, to=2-4]
		\arrow["h"', curve={height=6pt}, from=1-1, to=3-2]
		\arrow["\lrcorner"{anchor=center, pos=0.125}, draw=none, from=2-2, to=3-4]
		\arrow["m", dashed, from=1-1, to=2-2]
		\arrow["{g'}", from=2-2, to=2-4]
	\end{tikzcd}\]
	We now want to show that any arrow $r:d \to e'''$ such that $f'r = h$ is homotopic to $m$. Any such arrow $m$ can be canonically coerced into a cone morphism from $[d,h,k]$ to $[e''',f',g']$. But by terminality of $[e''',f',g']$ as a cone over $[g,f]$, the type of these cone morphisms is contractible, giving an identification $m=r$ of cone morphisms which descends to an identification between the respective arrows.

	All in all, this renders $f':e''' \cartarr e''$ a cartesian arrow.

\end{proof}

\subsection{Lex cartesian families}\label{ssec:lex-fam}

\subsubsection{Lex cartesian families}

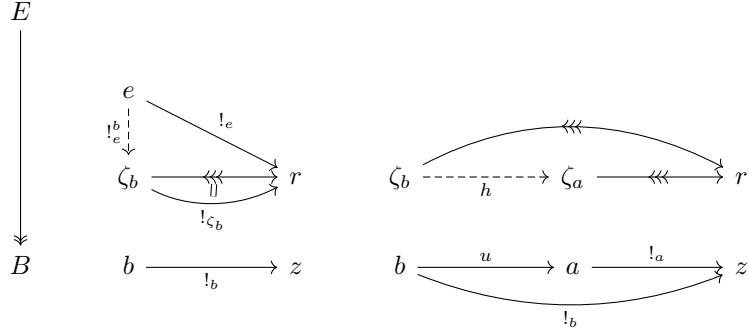
\begin{figure}
	\[\begin{tikzcd}
		E \\
		& e \\
		& {\zeta_b} && r & {\zeta_b} && {\zeta_a} && r \\
		B & b && z & b && a && z
		\arrow["{!_b}"', from=4-2, to=4-4]
		\arrow[two heads, from=1-1, to=4-1]
		\arrow[""{name=0, anchor=center, inner sep=0}, from=3-2, to=3-4, cart]
		\arrow["{!_e}", from=2-2, to=3-4]
		\arrow["{!^b_{e}}"', dashed, from=2-2, to=3-2]
		\arrow[""{name=1, anchor=center, inner sep=0}, "{!_{\zeta_b}}"', curve={height=12pt}, from=3-2, to=3-4]
		\arrow["u", from=4-5, to=4-7]
		\arrow["{!_a}", from=4-7, to=4-9]
		\arrow[from=3-7, to=3-9, cart]
		\arrow[curve={height=-24pt}, from=3-5, to=3-9, cart]
		\arrow["h"', dashed, from=3-5, to=3-7]
		\arrow["{!_b}"', curve={height=18pt}, from=4-5, to=4-9]
		\arrow[shorten <=2pt, shorten >=2pt, Rightarrow, no head, from=0, to=1]
	\end{tikzcd}\]
	\caption{Construction of fiberwise terminal objects}
	\label{fig:fib-term-obj}	
\end{figure}

A specific class that becomes important in the next chapter are the \emph{lex cartesian families}. These admit terminal elements and pullbacks in each fiber, and they get preserved under cartesian reindexing. We recall the notions of terminal element and pullback from~\cite[Definition~9.6]{RS17} and \cite[Definition~5.1.7]{BW21} (or their dual versions, respectively). See also more generally~\cite[Definitions~3.3 and 3.4]{BM21}.

\begin{definition}[Terminal elements and pullbacks in a Rezk type]
	Let $B$ be a Rezk type.
	\begin{enumerate}
		\item Let $z : B$ be term. We call $z$ a \emph{terminal element of $B$} if the type $\commaty{B}{z}$ is contractible.
		\item Let $\tau : B^{\Lambda_2^2}$, called a \emph{cospan} in $B$. A \emph{pullback over $\tau$} is a terminal element in the type of \emph{cones over $\tau$}
		\[ \tau/B \defeq \ndexten{\Delta^1 \times \Delta^1}{B}{\Lambda_2^2}{\tau}.\]
	\end{enumerate}
\end{definition}

We are looking at the preservation properties only for the case of finite limits, see~\cite[Section~3.2]{BM21} for a more general treatment.

\begin{definition}[Preservation of terminal elements and pullbacks]\label{def:pres-term-pb}
	Let $A,B$ be Rezk types.
	\begin{enumerate}
		\item Let $A$ have a terminal object $z$. A functor $F : A \to B$ \emph{preserves terminal elements} if $F(z)$ is terminal in $B$.
		\item  A functor $F : A \to B$  \emph{preserves (all) pullbacks} if for all cospans $\tau : A^{\Lambda_2^2}$ the induced functor
		\[ F \circ - : A/\tau \to B/(F  \tau)\]
			preserves terminal elements.
	\end{enumerate}
\end{definition}

\begin{definition}[Lex Rezk types and lex cartesian families]
\begin{enumerate}
	\item A Rezk type is \emph{lex} if it it has all pullbacks and a terminal object. A functor between Rezk types is a \emph{lex functor} if it preserves terminal elements and pullbacks in the sense of~\Cref{def:pres-term-pb}.
	\item Let $B$ be a lex Rezk type. A cartesian fibration over $B$ is \emph{lex} if all the fibers of $P$ have terminal objects and pullbacks, and both notions are preserved by the reindexing functors.
\end{enumerate}
	
\end{definition}

We will analyze the lexness conditions further in~\Cref{prop:term-obj-fib,prop:pb-fib}.

This is gives the following alternative characterization: a cartesian family $P : B \to \UU$ over a lex type $B$ is lex if and only if the total type $\totalty{P}$ is lex and the unstraightening $\pi_P : \totalty{P} \to B$ is a lex functor. Note that for a functor $f:A \to B$ preserving a terminal object $z:A$ is a propositional condition. If $z'$ denotes the terminal object in $B$ there is a path $f(z) = z'$ if and only if $f(z):B$ is terminal,~\ie~$\prod_{b:B}\isContr(b \to_B f(z))$ if and only if the homotopically unique arrow $!_{f(z)}: f(z) \to z'$ is an isomorphism. We will not discuss this further here, but similar considerations hold for limits in general, by their defining universal property as terminal objects of the respective Rezk types of cones, see~\cite[Subsection~3.2]{BM21}.

In principle, we also think in the synthetic setting there could be a more uniform and abstract treatment of ``$X$-shaped limit fibrations'', for a given shape or type $X$, after~\cite[Definition~8.5.1]{BorHandb2}, but we do not develop this here. Instead we follow the account of~\cite[Section~8]{streicher2020fibered}, adapting it to the synthetic setting.

\subsubsection{Terminal elements}

The overall aim is to recover the standard characterization of lex cartesian fibrations: Fix a base with the desired limits. Then the total type has those limits and they are preserved by the fibration if and only if the fibers each have the respective limits, and the reindexing functors preserve them. First, we consider the case of terminal elements.

\begin{proposition}\label{prop:term-obj-fib}
	Let $P:B \to \UU$ be a cartesian family and $B$ be a Rezk type with terminal object $z:B$. Denote by $\pi:E \fibarr B$ the unstraightening of $P$. Then the following are equivalent:
	\begin{enumerate}
		\item\label{it:term-obj-fib-i} The total Rezk type $E \defeq \widetilde{P}$ has a terminal object $\totalty{z}$, and $\pi$ preserves it, \ie~$\pi(\totalty{z}):B$ is terminal.
		\item\label{it:term-obj-fib-ii} For all $b:B$, the fiber $P\,b$ has a terminal object, and for all arrows $u:b \to_B a$ the functor $u^*: P\,a \to P\,b$ preserves the terminal object.
	\end{enumerate}
\end{proposition}

\begin{proof}
	\begin{description}
		\item[$\ref{it:term-obj-fib-i}\implies\ref{it:term-obj-fib-ii}$] For the visualization of both parts, \cf~\Cref{fig:fib-term-obj}. Denote by $\totalty{z} \defeq \pair{z}{r}:E$ the terminal object of $E$, with $z:B$ terminal. For $b:B$, consider the point $\zeta_b \defeq (!_b)^*r:P\,b$. We claim that this is the terminal object of the fiber $P\,b$. Indeed, consider the canonical arrow $!_r:e \to_{!_b} r$. Then there is a unique arrow $!^b_e:e \to_{P\,b} \zeta_b$ such that $P^*(!_b,r) \circ !^b_e = !_r$. But by terminality of $r$, the cartesian lift $P^*(!_b,r)$ also is propositionally equal to the terminal projection $!_{\zeta_b}: \zeta_b \to r$. Now, for any given map $g:e \to_{P\,b} \zeta_b$ we have that $!_{\zeta_b} \circ g = !_e$, but by the universal property of $!_{\zeta_b}$ there is only a unique such arrow $g$ up to homotopy. Hence, $\zeta_b$ is terminal in $P\,b$.
		
		Let $u:b \to_B a$. We will show that there is a path~$u^*\,\zeta_a = \zeta_b$. As we have just seen, we have $\zeta_a = (!_a)^*(r)$, and similarly for $\zeta_b$. Consider their terminal projections to $r$, which are necessarily cartesian arrows. From this and the identification $!_b = !_a \circ u$ in $B$, we get a unique filler $h:\zeta_b \to_u \zeta_a$. Moreover, $h$ is cartesian by left cancelation, so $h = P^*(u,\zeta_a): \zeta_b \cartarr_u^P \zeta_a$. This establishes the desired path.
		\item[$\ref{it:term-obj-fib-i}\implies\ref{it:term-obj-fib-ii}$] Conversely, consider the section $\zeta:\prod_{b:B} P\,b$ choosing the terminal element in each fiber. Let $b:B$. By assumption, the cartesian lift of the terminal map $!_b:b \to z$ has $\zeta_a$ as its source vertex, up to a path. Let $e:P\,b$ be some point. Since $\zeta_b$ is terminal in $P\,b$, there exists a unique morphism $!_e^b: e \to_{P\,b} \zeta_b$, and post-composition with the cartesian lift $P^*(!_b,\zeta_z): \zeta_b \cartarr \zeta_z$ gives a morphism $t_e: e \to \zeta_z$:
		\[\begin{tikzcd}
			E & e \\
			& {\zeta_b} && {\zeta_z} \\
			B & b && z
			\arrow["{!_b}", from=3-2, to=3-4]
			\arrow[from=2-2, to=2-4, cart]
			\arrow[two heads, from=1-1, to=3-1]
			\arrow["{t_e}", from=1-2, to=2-4]
			\arrow["{!_e^b}"', dashed, from=1-2, to=2-2]
		\end{tikzcd}\]
		Finally, any morphism $f:e \to \zeta_z$, up to homotopy, lies over $!_b: b \to z$, and necessarily has the same factorization again, hence is identified with $t_e$. Therefore, $\pair{z}{\zeta_z}$ defines the (``global'') terminal element of $E$, and we have $\pi(z,\zeta_z) \defeq z$.
	\end{description}
\end{proof}

The end of the proof of~\Cref{prop:term-obj-fib} shows that, with the same preconditions, if either of the conditions is satisfied, the map $\zeta : B \to E$ choosing fiberwise terminal elements presevers terminal elements.
\begin{corollary}\label{cor:terminal-totalty}
	Let $P:B \to \UU$ be a cartesian family and $B$ be a Rezk type with terminal object $z:B$. Denote by $\pi:E \fibarr B$ the unstraightening of $P$ and assume that $E$ has a terminal object and that $\pi$ preserves it. Then $\pair{z}{\zeta_z} : E$ is terminal.
\end{corollary}

\subsubsection{Pullbacks}

We are now turning to the analogous statement for pullbacks, which requires more preparation First, we give two conditions on dependent squares being pullbacks.

\begin{lemma}[\protect{\cite[Lemma~8.1(1)]{streicher2020fibered}}]\label{lem:cart-arr-pb}
	Let $P:B \to \UU$ be a cartesian family over a Rezk type $B$. Then any dependent square in $P$ which lies over a pullback and all of whose sides are cartesian arrows is itself a pullback.
\end{lemma}

\begin{proof}
	Consider a square in $\totalty{P}$ together with a cone, and the fillers $m$, $m'$ over $r$ which are induced by cartesianness of $f'$ and $g'$, respectively:
	
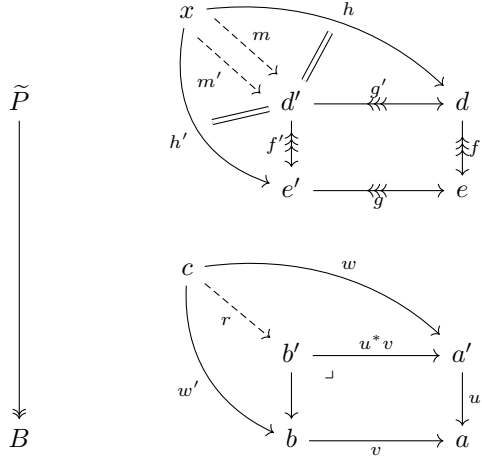
\[\begin{tikzcd}
		&& x \\
		{\widetilde{P}} &&& {d'} && d \\
		&&& {e'} && e \\
		&& c \\
		&&& {b'} && {a'} \\
		B &&& b && a
		\arrow["{f'}"', from=2-4, to=3-4, cart]
		\arrow["g"', from=3-4, to=3-6, cart]
		\arrow["{g'}", from=2-4, to=2-6, cart]
		\arrow["f", from=2-6, to=3-6, cart]
		\arrow[""{name=0, anchor=center, inner sep=0}, "h", curve={height=-18pt}, from=1-3, to=2-6]
		\arrow[""{name=1, anchor=center, inner sep=0}, "{h'}"', curve={height=24pt}, from=1-3, to=3-4]
		\arrow["{m'}"', shift right=2, dashed, from=1-3, to=2-4]
		\arrow["m", shift left=2, dashed, from=1-3, to=2-4]
		\arrow["u", from=5-6, to=6-6]
		\arrow["w", curve={height=-18pt}, from=4-3, to=5-6]
		\arrow["{w'}"', curve={height=18pt}, from=4-3, to=6-4]
		\arrow[from=5-4, to=6-4]
		\arrow["{u^*v}", from=5-4, to=5-6]
		\arrow["\lrcorner"{anchor=center, pos=0.125}, draw=none, from=5-4, to=6-6]
		\arrow["r"', dashed, from=4-3, to=5-4]
		\arrow["v"', from=6-4, to=6-6]
		\arrow[two heads, from=2-1, to=6-1]
		\arrow[shorten >=6pt, Rightarrow, no head, from=2-4, to=1]
		\arrow[shorten >=4pt, Rightarrow, no head, from=2-4, to=0]
	\end{tikzcd}\]
	We have $fh = (fg')m$ and $gh'=g(f'm') = (fg')m'$. But also $fh=gh'$, so $(fg')m = (fg')m'$. But since $fg'$ is cartesian as the composition of two cartesian arrows, we get a homotopy $m=m'$ as desired. This in particular induces gives an identification between the two ensuing cone morphisms defined by $m$ and $m'$, respectively, from $[h',h]$ to $[f',g']$, both being cones over the cospan $[g,f]$.
\end{proof}

\begin{lemma}[\protect{\cite[Lemma~8.1(2)]{streicher2020fibered}}]\label{lem:op-sides-pb}
	Let $P:B \to \UU$ be a cartesian family over a Rezk type $B$. Then any dependent square in $P$ of the form
	\[\begin{tikzcd}
		{e'''} && {e''} \\
		{e'} && e
		\arrow["g"', squiggly, from=1-1, to=2-1]
		\arrow["f"', from=2-1, to=2-3, cart]
		\arrow["{f'}", from=1-1, to=1-3, cart]
		\arrow["{g'}", squiggly, from=1-3, to=2-3]
	\end{tikzcd}\]
	is a pullback.
\end{lemma}

\begin{proof}
	Consider a point $d$ and maps $h:d \to e'$, $h':d \to e''$ such that $g'h' = fh$. By cartesianness of $f'$, there uniquely exists $k:d \to e'''$ such that there is a homotopy $H : f'k = h'$:
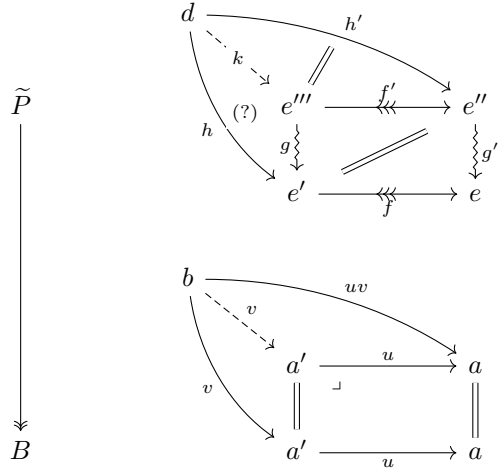
	\[\begin{tikzcd}
		&& d \\
		{\widetilde{P}} &&& {e'''} && {e''} \\
		&&& {e'} && e \\
		&& b \\
		&&& {a'} && a \\
		B &&& {a'} && a
		\arrow["g"', squiggly, from=2-4, to=3-4]
		\arrow["f"', from=3-4, to=3-6, cart]
		\arrow["{f'}", from=2-4, to=2-6, cart]
		\arrow["{g'}", squiggly, from=2-6, to=3-6]
		\arrow["k"{description}, dashed, from=1-3, to=2-4]
		\arrow[""{name=0, anchor=center, inner sep=0}, "{h'}", curve={height=-12pt}, from=1-3, to=2-6]
		\arrow[""{name=1, anchor=center, inner sep=0}, "h"', curve={height=12pt}, from=1-3, to=3-4]
		\arrow[shorten <=10pt, shorten >=10pt, Rightarrow, no head, from=3-4, to=2-6]
		\arrow[two heads, from=2-1, to=6-1]
		\arrow["u", from=5-4, to=5-6]
		\arrow[Rightarrow, no head, from=5-6, to=6-6]
		\arrow[Rightarrow, no head, from=5-4, to=6-4]
		\arrow["u"', from=6-4, to=6-6]
		\arrow["uv", curve={height=-12pt}, from=4-3, to=5-6]
		\arrow["v"', curve={height=12pt}, from=4-3, to=6-4]
		\arrow["v", dashed, from=4-3, to=5-4]
		\arrow["\lrcorner"{anchor=center, pos=0.125}, draw=none, from=5-4, to=6-6]
		\arrow[shorten >=4pt, Rightarrow, no head, from=2-4, to=0]
		\arrow["{(?)}"{description}, Rightarrow, draw=none, from=1, to=2-4]
	\end{tikzcd}\]
	To show that also $h= gk$, it suffices to show that $f(gk) = g'h'$, since also $fh = g'h'$, which taken together then would imply $h=gk$ by cartesianness of $f$. Indeed, by the above we have a chain of paths
	\[ f(gk) = (fg)k = (g'f')k = g'(f'k) = g'h',\]
	which implies the claim that $h=gk$. This means that $k : d \to e'''$ induces a morphism from to $[h,h']$ to $[g,f']$ as cones over the cospan $[f,g']$. The type of such cones is equivalent to $\sum_{r : d \to e'''} (gr = h) \times (f' \times r = h')$. By cartesianness of $f'$ the type $\sum_{r : d \to e'''} (f' \times r = h')$ is contractible with center $\pair{k}{H}$. The argument we gave above yields a homotopy $K : (gk = H)$. By strictification, this exhibits $\angled{k,K,H}$ as a terminal projection from $[h,h']$ to $[g,f']$, hence $[g,f']$ as pullback cone.
\end{proof}

The next lemma presents a sufficient condition for the ``local'' pullbacks being ``global'' pullbacks.
\begin{lemma}[\cf~\protect{\cite[Lemma~8.2]{streicher2020fibered}}]\label{lem:loc-pb-is-pb}
	Let $P:B \to \UU$ be a cartesian family and $B$ be a Rezk type where all pullbacks exist. Assuming that all fibers have pullbacks, and these are preserved by the reindexing functors, we have: A pullback in a fiber $P\,b$ is also a pullback in $\totalty{P}$.
\end{lemma}

\begin{proof}
	Let $\pi:E \fibarr B$ be the unstraightening of $P$. For any $b:B$, consider a pullback square in $P\,b$, together with a cone in $E$, as follows:
	
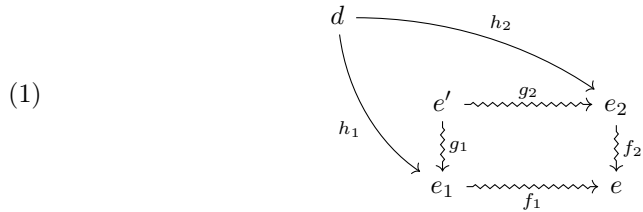
\begin{equation}
		\label{cd:local-pb}
		\begin{tikzcd}
			d \\
			& {e'} && {e_2} \\
			& {e_1} && e
			\arrow["{f_1}"', squiggly, from=3-2, to=3-4]
			\arrow["{f_2}", squiggly, from=2-4, to=3-4]
			\arrow["{g_2}", squiggly, from=2-2, to=2-4]
			\arrow["{g_1}", squiggly, from=2-2, to=3-2]
			\arrow["{h_2}", curve={height=-12pt}, from=1-1, to=2-4]
			\arrow["{h_1}"', curve={height=12pt}, from=1-1, to=3-2]
		\end{tikzcd}
	\end{equation}
	Projecting down we find that $\pi(h_1) = u = \pi(h_2)$ for some $u:b \to_B a$. Consider the factorizations
	\[\begin{tikzcd}
		d && {u^*\,e_n} && {e_n}
		\arrow["{m_n}"', squiggly, from=1-1, to=1-3]
		\arrow["{k_n}"', from=1-3, to=1-5, cart]
		\arrow["{h_n}", curve={height=-24pt}, from=1-1, to=1-5]
	\end{tikzcd}\]
	for $n=1,2$.
	We claim that $(u^*\,f_1) \circ m_1 = (u^*\,f_2) \circ m_2$. To see this, consider the following induced diagram
	\[\begin{tikzcd}
		d && {u^*\,e_2} && {e_2} \\
		{u^*\,e_1} && {u^*\,e} && e \\
		{e_1}
		\arrow["{f_2}", squiggly, from=1-5, to=2-5]
		\arrow[from=1-3, to=1-5, cart, "k_2"]
		\arrow[from=2-3, to=2-5, cart, "k"]
		\arrow[squiggly, from=1-1, to=1-3, "m_2"]
		\arrow[from=2-1, to=2-3, squiggly, "{u^*\,f_1}"]
		\arrow[from=2-1, to=3-1, cart, "k_1"]
		\arrow["{h_1}"', curve={height=24pt}, from=1-1, to=3-1]
		\arrow["{f_1}"', squiggly, from=3-1, to=2-5]
		\arrow[squiggly, from=1-3, to=2-3, "{u^*\,f_2}"]
		\arrow["{h_2}", curve={height=-24pt}, from=1-1, to=1-5]
		\arrow["m_1"' swap, squiggly, from=1-1, to=2-1]
		\arrow[shorten <=10pt, shorten >=10pt, Rightarrow, no head, from=3-1, to=2-3]
		\arrow[shorten <=10pt, shorten >=10pt, Rightarrow, no head, from=2-3, to=1-5]
	\end{tikzcd}\]
	in which the right square commutes. Now, in fact the left sub-square commutes as well because both sides are equalized by the cartesian arrow $k$: By assumption we have $(f_1k_1)m_1 = (f_2k_2)m_2$, \ie~$(k (u^*f_1))m_1 = (k (u^*f_2))m_2$, hence $(u^*f_1)m_1 = (u^*f_2)m_2$ as claimed.
	
	Now, by assumption~the square of vertical arrows in~\ref{cd:local-pb}, is a pullback in $P\,b$, and gets preserved by $u^*:P\,a\to P\,b$. Then the gap map $\ell:d \rightsquigarrow u^*\,e'$ as indicated below is vertical:
	\[\begin{tikzcd}
		d \\
		& {u^*\,e'} && {u^*\,e_2} \\
		& {u^*\,e_1} && {u^*\,e}
		\arrow[squiggly, from=2-2, to=3-2]
		\arrow[squiggly, from=2-2, to=2-4]
		\arrow[squiggly, from=2-4, to=3-4]
		\arrow["{m_2}", curve={height=-12pt}, squiggly, from=1-1, to=2-4]
		\arrow["{m_1}"', curve={height=12pt}, squiggly, from=1-1, to=3-2]
		\arrow["\ell", squiggly, from=1-1, to=2-2, dashed]
		\arrow[squiggly, from=3-2, to=3-4]
	\end{tikzcd}\]
	Then, for $k'\defeq P^*(u,e'): u^*\,e' \cartarr_u e'$ we claim that the mediating arrow for the original diagram~\ref{cd:local-pb} is given by
	\[ \ell' \defeq k' \circ \ell: d \to e'.\]
	Indeed, we find
	\[ g_n \ell' = g_n(k'\ell) = (k_n \circ u^*\,g_n)\ell = k_n m_n = h_n\]
	for $n=1,2$.
	Furthermore, $\ell'$ is unique with this property because its vertical component is determined uniquely up to homotopy as a gap map of a pullback in $P\,b$. This exhibits $\ell'$ together with the two homotopies $g_k \ell' = h_k$ for $k=1$ or $k=2$, as the terminal cone morphism from $[h_1,h_2]$ to $[g_1, g_2]$ witnessing that the square $[g_1,g_2,f_1,f_2]$ is a pullback.
\end{proof}

Finally, we can state the desired characterization.

\begin{proposition}[\cf~\protect{\cite[Theorem~8.3]{streicher2020fibered}}]\label{prop:pb-fib}
	Let $P:B \to \UU$ be a cartesian family and $B$ be a Rezk type where all pullbacks exist. Denote by $\pi:E \fibarr B$ the unstraightening of $P$. Then the following are equivalent:
	\begin{enumerate}
		\item\label{it:pb-fib-i} The total Rezk type $E$ has all pullbacks, and $\pi$ preserves them.
		\item\label{it:pb-fib-ii} For all $b:B$, the fiber $P\,b$ has all pullbacks, and for all arrows $u:b \to_B a$ the functor $u^*: P\,a \to P\,b$ preserves them.
	\end{enumerate}
\end{proposition}

\begin{proof}
	\item[$\ref{it:pb-fib-i} \implies \ref{it:pb-fib-ii}$] Since $\pi$ preserves pullbacks, every pullback, \emph{taken in $E$}, of vertical arrows in a fiber $P\,b$ is a pullback \emph{in $P\,b$}, \ie~given a cone of vertical arrows, the mediating arrow is necessarily vertical as well. What is left to show is that the reindexing functors preserve the pullbacks. Let $u:b \to_B a$. Consider a pullback square in $P\,a$, together with the cartesian liftings of $u$ \wrt~to each point. Then, by~\Cref{lem:op-sides-pb} the ensuing squares are pullbacks, as indicated in:
	\[\begin{tikzcd}
		{} && E & {u^*\,e'''} &&& {e'''} \\
		&&&& {u^*\,e''} &&& {e''} \\
		&&& {u^*\,e'} &&& {e'} \\
		&&&& {u^*\,e} &&& e \\
		&& B && b &&& a
		\arrow[from=1-4, to=1-7, cart]
		\arrow[squiggly, from=1-7, to=3-7]
		\arrow[squiggly, from=1-7, to=2-8]
		\arrow[squiggly, from=2-8, to=4-8]
		\arrow[squiggly, from=3-7, to=4-8]
		\arrow[from=3-4, to=3-7, cart]
		\arrow[dashed, from=1-4, to=2-5, squiggly]
		\arrow[dashed, from=3-4, to=4-5, squiggly]
		\arrow[dashed, from=1-4, to=3-4, squiggly]
		\arrow[two heads, from=1-3, to=5-3]
		\arrow["u", from=5-5, to=5-8]
		\arrow["\lrcorner"{anchor=center, pos=0.125, rotate=-45}, draw=none, from=1-7, to=4-8]
		\arrow["\lrcorner"{anchor=center, pos=0.125}, draw=none, from=2-5, to=4-8]
		\arrow["\lrcorner"{anchor=center, pos=0.125}, shift left=2, draw=none, from=1-4, to=3-7]
		\arrow[from=2-5, to=2-8, cart, crossing over]
		\arrow[from=4-5, to=4-8, cart, crossing over]
		\arrow[dashed, from=2-5, to=4-5, crossing over, squiggly]
	\end{tikzcd}\]
	By~\cite[Remark~26.1.5(ii)]{RijIntro}, we obtain that the left hand square is a pullback, as desired.
	\item[$\ref{it:pb-fib-ii} \implies \ref{it:pb-fib-i}$] Conversely, consider a cospan an $E$, comprised of dependent arrows $(f:d \to e \leftarrow d':f')$. First, we consider their vertical/cartesian-factorizations, $f=gk$, $f'=g'm$. This gives rise to the following situation, which we will readily explain:
	\[\begin{tikzcd}
		{d'''} && {d''} && {d'} \\
		{e''''} && {e'''} && {e''} \\
		d && {e'} && e
		\arrow[""{name=0, anchor=center, inner sep=0}, "k"{description}, squiggly, from=3-1, to=3-3]
		\arrow["g", from=3-3, to=3-5, cart]
		\arrow["{g'}"', from=2-5, to=3-5, cart]
		\arrow["m"', squiggly, from=1-5, to=2-5]
		\arrow["{m'}"', squiggly, from=1-3, to=2-3]
		\arrow["{h'}"', from=2-3, to=3-3, cart]
		\arrow[""{name=1, anchor=center, inner sep=0}, "h", from=2-3, to=2-5, cart]
		\arrow[""{name=2, anchor=center, inner sep=0}, "{h'''}", from=1-3, to=1-5, cart]
		\arrow["f"{description}, curve={height=18pt}, from=3-1, to=3-5]
		\arrow["{f'}"{description}, curve={height=-18pt}, from=1-5, to=3-5]
		\arrow["{h''}"', from=2-1, to=3-1, cart]
		\arrow[""{name=3, anchor=center, inner sep=0}, "{k'}"{description}, squiggly, from=2-1, to=2-3]
		\arrow["\lrcorner"{anchor=center, pos=0.125}, draw=none, from=1-3, to=2-5]
		\arrow["\lrcorner"{anchor=center, pos=0.125}, draw=none, from=2-3, to=3-5]
		\arrow["\lrcorner"{anchor=center, pos=0.125}, draw=none, from=2-1, to=3-3]
		\arrow["\tau"{description}, draw=none, from=2-3, to=3-5]
		\arrow["{\ell'}"', squiggly, from=1-1, to=2-1]
		\arrow["\ell", squiggly, from=1-1, to=1-3]
		\arrow["{\tau'}"{description}, "\lrcorner"{anchor=center, pos=0.125}, draw=none, from=1-1, to=2-3]
		\arrow["\sigma"{description}, Rightarrow, draw=none, from=3, to=0]
		\arrow["{\sigma'}"{description}, Rightarrow, draw=none, from=2, to=1]
	\end{tikzcd}\]
	First of all, the diagram $\tau$ is a pulback by~\Cref{lem:cart-arr-pb}. The (vertical) fillers $k'$ and $m'$, respectively are induced by $h'$ and $h$ being cartesian, respectively. Then by~\Cref{lem:op-sides-pb}, the squares $\sigma$, $\sigma'$ are pullbacks, too. Since the fibers have pullbacks, the square $\tau'$ exists. As the reindexings preserve the local pullbacks, we can apply~\Cref{lem:loc-pb-is-pb}, so $\tau'$ is a pullback in $E$. Altogether, this yields the pullback square of $f'$ along $f$.
\end{proof}

The last part of the proof implies that, under the assumptions of the proposition the map $\zeta : B \to E$ that picks the terminal element in each fiber sends pullbacks to pullbacks.

\begin{corollary}\label{cor:pullback-totalty}
	Let $B$ be a Rezk type with all pullbacks. Consider a cartesian family $P:B \to \UU$ . Denote by $\pi:E \fibarr B$ the unstraightening of $P$ and assume that $E$ has all pullbacks and $\pi$ preserves then. Then $\zeta : B \to E$ sends pullbacks to pullbacks.
\end{corollary}

\begin{proof}
	We first observe that, for any $u:b \to_B a$ we find for $\zeta_u : (\zeta_b \to^P_u \zeta_a)$ the decomposition $\zeta_u = P^*(u,\zeta_e) \circ k$ where $k : \zeta_b \vertarr_b u^*e$ is the filler given by cartesianness. Hence, $\zeta$ maps any (pullback) square in $B$ to a composite square as at the end of the proof of~\Cref{prop:pb-fib}. But by the same arguments, the square exhibits itself as a pullback square.
\end{proof}

Together with the corresponding statement for preservation of terminal elements, we obtain the following corollary, that will be crucial later in constructing the correspondence between lextensive fibrations and lex functors for Moens' Theorem~\Cref{thm:moens-thm}.

\begin{corollary}[Lexness of the fiberwise terminal map]\label{cor:lex-zeta}
	Let $\pi : E \fibarr B$ be a lex cartesian fibration. Then the map $\zeta : B \to E$ that maps $b$ to $\pair{b}{\zeta_b}$ where $\zeta_b$ is the terminal element in the fiber over $b$ is a lex functor.
\end{corollary}

\begin{proof}
	Combine \Cref{cor:terminal-totalty,cor:pullback-totalty}.
\end{proof}

\subsubsection{Pullbacks in slices}

Throughout the paper, we will need to compute pullbacks in slice Rezk types. Recall the (dual) considerations of~\cite[Proposition~5.1.5, Definition~5.1.6, and Definition 5.1.7]{BW21}. We are proving that these pullbacks arise from pullbacks in the base.

\begin{proposition}[Pullbacks in slice Rezk types]
	Let $B$ be a Rezk type with all pullbacks, and $a:B$. Consider morphisms $u:b \to_B a$, $u':b'\to_B a$, and $u'':b''\to_B a$. Let $\kappa \defeq \pair{f}{g}: \Lambda_2^2 \to B/a$ be a cospan in the slice category, given by
\[\begin{tikzcd}
	& {u''} \\
	{u'} & u
	\arrow["g", from=1-2, to=2-2]
	\arrow["f", from=2-1, to=2-2]
\end{tikzcd}\]
for $f : (u' \to_{B/a} u)$ and $g : (u'' \to_{B/a} u)$ (\ie, such that there exist identifications $uf = u'$ and $ug = u''$).
	 Consider a square in $B$:
	 \[\begin{tikzcd}
		 {b'''} & {b''} \\
		 {b'} & b
		 \arrow["g", from=1-2, to=2-2]
		 \arrow["f", from=2-1, to=2-2]
		 \arrow["{g'}"', from=1-1, to=2-1]
		 \arrow["{f'}", from=1-1, to=1-2]
	 \end{tikzcd}\]
	 Assume it gives rise to a cone $\sigma \defeq \angled{u''',f',g'} : (B/a)/\kappa$, \ie, we have:
	\[\begin{tikzcd}
	{u'''} & {u''} \\
	{u'} & u
	\arrow["g", from=1-2, to=2-2]
	\arrow["f", from=2-1, to=2-2]
	\arrow["{g'}"', from=1-1, to=2-1]
	\arrow["{f'}", from=1-1, to=1-2]
	\end{tikzcd}\]
	Then $\sigma$ is terminal if and only if its evaluation at $0$ (\ie, the cone $\angled{b''',f',g'}$ over $\pair{f}{g}$) is a terminal cone in $B$. 
\end{proposition}

\begin{proof}
	One computes $(B/a)/\kappa \simeq \sum_{c : B} \sum_{v:(c \to_B a)} \sum_{\substack{ h:(v \to_{B/a} u') \\ k:(v \to_{B/a} u'')}} (fh = gk)$. We can strictify this data, getting rid of the homotopy. Indeed, this type is equivalent to an extension type whose terms are $3$-cubes in $B$ (whose definition we will not spell out here) that restrict as indicated:
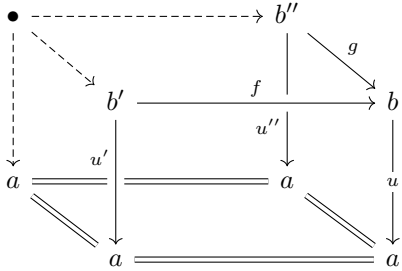
\[\begin{tikzcd}
	\bullet &&& {b''} \\
	& {b'} &&& b \\
	a &&& a \\
	& a &&& a
	\arrow[dashed, from=1-1, to=1-4]
	\arrow[dashed, from=1-1, to=3-1]
	\arrow[Rightarrow, no head, from=3-1, to=3-4]
	\arrow["{u''}"'{pos=0.7}, from=1-4, to=3-4]
	\arrow[dashed, from=1-1, to=2-2]
	\arrow["f", from=2-2, to=2-5, crossing over]
	\arrow[Rightarrow, no head, from=3-1, to=4-2]
	\arrow[Rightarrow, no head, from=4-2, to=4-5]
	\arrow["u"{description}, from=2-5, to=4-5]
	\arrow["{u'}"'{pos=0.3}, from=2-2, to=4-2, crossing over]
	\arrow["g", from=1-4, to=2-5]
	\arrow[Rightarrow, no head, from=3-4, to=4-5]
\end{tikzcd}\]
Morphisms between two cones are given accordingly:
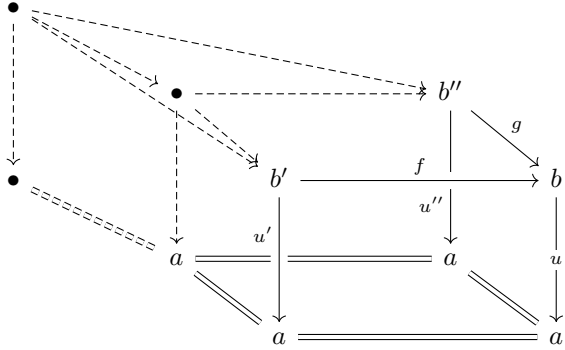
\[\begin{tikzcd}
	\bullet \\
	&& \bullet &&& {b''} \\
	\bullet &&& {b'} &&& b \\
	&& a &&& a \\
	&&& a &&& a
	\arrow[dashed, from=2-3, to=2-6]
	\arrow[dashed, from=2-3, to=4-3]
	\arrow[Rightarrow, no head, from=4-3, to=4-6]
	\arrow["{u''}"'{pos=0.7}, from=2-6, to=4-6]
	\arrow[dashed, from=2-3, to=3-4]
	\arrow["f", from=3-4, to=3-7, crossing over]
	\arrow[Rightarrow, no head, from=4-3, to=5-4]
	\arrow[Rightarrow, no head, from=5-4, to=5-7]
	\arrow["u"{description}, from=3-7, to=5-7]
	\arrow["{u'}"'{pos=0.3}, from=3-4, to=5-4, crossing over]
	\arrow["g", from=2-6, to=3-7]
	\arrow[Rightarrow, no head, from=4-6, to=5-7]
	\arrow[dashed, from=1-1, to=2-6]
	\arrow[dashed, from=1-1, to=3-1]
	\arrow[Rightarrow, dashed, no head, from=3-1, to=4-3]
	\arrow[dashed, from=1-1, to=3-4]
	\arrow[dashed, from=1-1, to=2-3]
\end{tikzcd}\]
	Since $B$ is a lex Rezk type, we can consider the square
\[\begin{tikzcd}
	{b''' \defeq b'\times_b b''} & {b''} \\
	{b'} & b
	\arrow["g", from=1-2, to=2-2]
	\arrow["f", from=2-1, to=2-2]
	\arrow["{{g'}}"', from=1-1, to=2-1]
	\arrow["{{f'}}", from=1-1, to=1-2]
	\arrow["\lrcorner"{anchor=center, pos=0.125}, draw=none, from=1-1, to=2-2]
\end{tikzcd}\]
	in $B$. Setting $u''' \defeq u'g' = u''f' : b''' \to a$ gives rise to a a cone
	\[ \sigma \defeq \angled{b''':B,u''':b''' \to_B a,f':u''' \to_{B/a} u'',g':u''' \to_{B/a} u'} : (B/a)/\kappa.\]
	Let $\tau \defeq \angled{c : B, v:c \to_B a, h:v \to_{B/a} u', k :v \to_{B/a} u''} : (B/a)/\kappa$ be another cone. By the pullback property of the initially considered square in $B$, its mediating map lifts to define a morphism of cones:
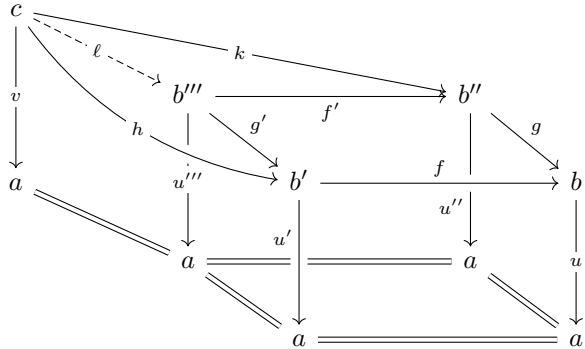
\[\begin{tikzcd}
	c \\
	&& {b'''} &&& {b''} \\
	a &&& {b'} &&& b \\
	&& a &&& a \\
	&&& a &&& a
	\arrow["{f'}", swap, from=2-3, to=2-6]
	\arrow["{u'''}"{description}, from=2-3, to=4-3]
	\arrow[Rightarrow, no head, from=4-3, to=4-6]
	\arrow["{u''}"'{pos=0.7}, from=2-6, to=4-6]
	\arrow["{g'}"', swap, from=2-3, to=3-4]
	\arrow["f", from=3-4, to=3-7, crossing over]
	\arrow[Rightarrow, no head, from=4-3, to=5-4]
	\arrow[Rightarrow, no head, from=5-4, to=5-7]
	\arrow["u"{description}, from=3-7, to=5-7]
	\arrow["{u'}"'{pos=0.3}, from=3-4, to=5-4, crossing over]
	\arrow["g", from=2-6, to=3-7]
	\arrow[Rightarrow, no head, from=4-6, to=5-7]
	\arrow["k"{description}, from=1-1, to=2-6]
	\arrow["v"{description}, from=1-1, to=3-1]
	\arrow[Rightarrow, no head, from=3-1, to=4-3]
	\arrow["h"{description}, from=1-1, to=3-4, curve={height=18pt}, crossing over]
	\arrow["\ell"{description}, dashed, from=1-1, to=2-3]
\end{tikzcd}\]
	Indeed, from a homotopy $g'\ell = h$ we get $u'''\ell = (u'g')\ell = u'h=v$, so $\ell$ really defines a term in $(v \to_{B/a} u''')$. Any such cone in $B/a$ after projection yields a pullback in $B$. Hence $\ell$ is unique up to homotopy with that property.
\end{proof}

\section{Bicartesian families}\label{sec:bicart-fam}

In this section, we consider families that are both cartesian and concartesian, corresponding to \emph{bicartesian fibrations}. Specifically, we are interested in such fibrations satisfying a so-called Beck--Chevalley condition (BCC). This form of the BCC has its origins in the work of B\'{e}nabou--Roubaud leading to their famous characterization of descent data of a fibration. In the $\inftyone$-categorical context such fibrations play a role \eg~in higher ambidexterity, \cf~work by Hopkins--Lurie~\cite{HoLuAmbi} and Heuts's notes~\cite{HeLuAmbi} of Lurie's work. We successively generalize Streicher's exposition and proofs~\cite[Section~15]{streicher2020fibered} to the synthetic $\inftyone$-categorical setting, also making explicit some arguments not detailed in~\emph{op.~cit.}

\subsection{Bicartesian families}

\begin{definition}[Bicartesian family]
	Let $B$ be a Rezk type. A \emph{bicartesian family} is a type family $P:B \to \UU$ which is both cartesian and cocartesian.
\end{definition}

Bicartesian families are hence equipped with both co- and contravariant transport operations for directed arrows. In fact, these induce adjunctions on the fibers.

\begin{proposition}
	Let $P:B \to \UU$ be a bicartesian family. For any $a,b:B$, $u:a \to_B b$, there is an adjunction:
	\[\begin{tikzcd}
		{P\,a} && {P\,b}
		\arrow[""{name=0, anchor=center, inner sep=0}, "{u_!}"{description}, curve={height=-12pt}, from=1-1, to=1-3]
		\arrow[""{name=1, anchor=center, inner sep=0}, "{u^*}"{description}, curve={height=-18pt}, from=1-3, to=1-1]
		\arrow["\dashv"{anchor=center, rotate=-90}, draw=none, from=0, to=1]
	\end{tikzcd}\]
\end{proposition}

\begin{proof}
	For fixed $a,b:B$, we define a pair of maps
	\[ \Phi:(u_!\,d \to e) \rightleftarrows (d \to u^*e) : \Psi,\]
	intended to be quasi-inverse to each other, through
	\[\Phi(g) \defeq \cartFill_{P^*(u,e)}^P(g \circ P_!(u,d)), \quad \Psi(k) \defeq \cocartFill_{P_!(u,d)}^P(P^*(u,e)\circ k).\]
	We write $f \defeq P_!(u,d): d \cocartarr u_!\,d$ and $r \defeq P^*(u,e): u^*e \cartarr e$. For a dependent arrow $g:u_!\,d \to^P e$ we have $g' \defeq \Phi(g) \circ r = g \circ f$ by construction. Next, we find that $\Psi(g') \circ f = r \circ \Phi(g)$. Combining these identities it follows that $\Psi(g') = \Psi(\Phi(g)) = g$ by cocartesianness of $f$. The other roundtrip is analogous.
\end{proof}

We now explain a few important constructions producing bicartesian fibrations.

\subsection{The family fibration}\label{ssec:fam-fib}

For an arbitrary map $\pi:E \fibarr B$ between Rezk types one can produce its \emph{free cocartesian fibration}, ~\cite[Theorem~5.2.19]{BW21}.

This cocartesian fibration turns out to be a \emph{bifibration} in case $B$ has all pullbacks, then called the \emph{family fibration}. To analyze the structure of the family fibration, we first recall the codomain fibration.

\begin{proposition}[Codomain fibration]\label{prop:cod-fib}
	If $B$ is a Rezk type that has all pullbacks, the codomain projection
	\[ \partial_1 : B^{\Delta^1} \to B, \partial_1(f) \defeq u(1)\]
	is a cartesian fibration.
\end{proposition}

\begin{proof}
	This is dual to~\cite[Proposition~5.2.17]{BW21}. In particular, for an arrow $u : a \to_B b$ in $B$, a point in the fiber over $b$ is given by an arrow $v : b' \to_B b$ in $B$. The cartesian lift of $u$ \wrt $v$ is given by the pullback square:
\[\begin{tikzcd}
	{B^{\Delta^1}} & {a \times_b b'} && {b'} \\
	& a && b \\
	B & a && b
	\arrow["{u^*v}"', from=1-2, to=2-2]
	\arrow["u", from=3-2, to=3-4]
	\arrow["u"', from=2-2, to=2-4]
	\arrow["{v^*u}", from=1-2, to=1-4]
	\arrow["v", from=1-4, to=2-4]
	\arrow["\lrcorner"{anchor=center, pos=0.125}, draw=none, from=1-2, to=2-4]
	\arrow["{\partial_1}"', two heads, from=1-1, to=3-1]
\end{tikzcd}\]
\end{proof}

\begin{proposition}\label{prop:fam-fib-cart}
	Let $B$ be a Rezk type with all pullbacks and $\pi:E \fibarr B$ a cartesian fibration, then the free cocartesian fibration
	\[\begin{tikzcd}
		{L(\pi)} && E \\
		{B^{\Delta^1}} && B \\
		B
		\arrow["\pi", two heads, from=1-3, to=2-3]
		\arrow["{\partial_0}"', from=2-1, to=2-3]
		\arrow[from=1-1, to=1-3]
		\arrow["{\partial_1}", two heads, from=2-1, to=3-1]
		\arrow["{\partial_1'}"', curve={height=24pt}, two heads, from=1-1, to=3-1]
		\arrow["\lrcorner"{anchor=center, pos=0.125}, draw=none, from=1-1, to=2-3]
		\arrow[two heads, from=1-1, to=2-1]
	\end{tikzcd}\]
	is itself a also cartesian fibration, hence a bicartesian fibration.
\end{proposition}

\begin{proof}
	For any isoinner fibration $\pi : E \fibarr B$ between Rezk types, the map $\partial_1' : L(\pi) \fibarr B$ is a cocartesian fibration by~\cite[Theorem~5.2.19]{BW21}.

	If (and only if) $B$ has pullbacks the codomain projection $\partial_1 : B^{\Delta^1} \to B$ is a fibration, see~\Cref{prop:cod-fib}. As cartesian fibrations are closed under pullback and composition (dually to~\Cite[Proposition~5.3.17]{BW21}), we get that $\partial_1' \defeq \partial_1 \circ (\partial_0)^* \pi$ is cartesian.
\end{proof}

\begin{definition}[Family fibration]
	Let $B$ be a Rezk type with all pullbacks and $\pi:E \fibarr B$ a cartesian fibration. Then, by~\Cref{prop:fam-fib-cart}, the free cocartesian fibration $\partial_1' : L(\pi) \fibarr B$ is a cartesian fibration, too. This bifibration is called the \emph{family fibration associated to $P$}.
\end{definition}

\begin{proposition}[Cartesian lifts in the family fibration]\label{prop:cartlift-famfib}
	Let $\pi:E \fibarr B$ be a cartesian fibration over a Rezk type $B$ with all pullbacks. The cartesian lift of an arrow $u:a \to b$ in $B$ with respect to $\pair{v:b' \to b}{e:P\,b'}$ in $L(\pi)$ is given by the pullback square of $v$ along $u$, together with the cartesian lift over the upper horizontal arrow as indicated in the following diagram:
\[\begin{tikzcd}
	{L(\pi)} & {(u^*v)^*e} && e \\
	& {a \times_b b'} && {b'} \\
	& a && b \\
	B & a && b
	\arrow["{u^*v}"', from=2-2, to=3-2]
	\arrow["{\partial_1'}"', two heads, from=1-1, to=4-1]
	\arrow["u", from=4-2, to=4-4]
	\arrow["u"', from=3-2, to=3-4]
	\arrow["{v^*u}", from=2-2, to=2-4]
	\arrow["v", from=2-4, to=3-4]
	\arrow["{P^*(u^*v,e)}", from=1-2, to=1-4]
	\arrow["\lrcorner"{anchor=center, pos=0.125}, draw=none, from=2-2, to=3-4]
\end{tikzcd}\]
\end{proposition}

\begin{proof}
	This follows by computing the lifts fiberwisely,~\cf~\cite[Subsection~5.2.3]{BW21}.
\end{proof}

\begin{figure}
	\centering
	\[\begin{tikzcd}
		{L(\pi)} && {(u')^*e} && e \\
		&& {a \times_b b'} && b' \\
		&& {a} && b \\
		B && {a} && b
		\arrow["u", from=4-3, to=4-5]
		\arrow["{v'}"', from=2-3, to=3-3]
		\arrow["u"', from=3-3, to=3-5, swap]
		\arrow["{u'}"', from=2-3, to=2-5]
		\arrow["v", from=2-5, to=3-5]
		\arrow["{P^*(u',e)}", from=1-3, to=1-5, cart]
		\arrow[Rightarrow, dotted, no head, from=1-3, to=2-3]
		\arrow[Rightarrow, dotted, no head, from=1-5, to=2-5]
		\arrow["\lrcorner"{anchor=center, pos=0.125}, draw=none, from=2-3, to=3-5]
		\arrow["{\partial_1'}"', two heads, from=1-1, to=4-1]
	\end{tikzcd}\]
	\caption{Cartesian lifts in the family fibration}\label{fig:cart-lifts-famfib}
\end{figure}

\subsection{The Artin gluing fibration}

\begin{definition}[Artin gluing]
	Let $B,C$ be Rezk types and $F:B \to C$ a functor. Then the map $\gl(F): \commaty{C}{F} \fibarr B$ constructed by pullback
	\[\begin{tikzcd}
		{C \downarrow F} && {C^{\Delta^1}} \\
		B && C
		\arrow["F"', from=2-1, to=2-3]
		\arrow["{\partial_1}", two heads, from=1-3, to=2-3]
		\arrow["{\mathrm{gl}(F)}"', two heads, from=1-1, to=2-1]
		\arrow[from=1-1, to=1-3]
		\arrow["\lrcorner"{anchor=center, pos=0.125}, draw=none, from=1-1, to=2-3]
	\end{tikzcd}\]
	is called the \emph{Artin gluing} (or simply \emph{gluing}) of $F$.
\end{definition}

Since the codomain projection $\partial_1 : C^{\Delta^1} \fibarr C$ is always a cocartesian fibration the gluing $\gl(F)$ is a cocartesian fibration as well. We will be concerned with the case that $C$ has all pullbacks. In this case $\partial_1: C^{\Delta^1} \fibarr C$ also is a cartesian fibration, hence a bifibration, and consequently the same is true for $\gl(F): \commaty{C}{F} \fibarr B$. Hence, from the description of the co-/cocartesian lifts in pullback fibrations~\cite[Proposition~5.2.14]{BW21}, the respective lifts in $\gl(F)$ can be computed as illustrated in Fig.~\ref{fig:lifts-gluing}. Given an arrow $u:b \to b'$ in $B$, a cocartesian lift with respect to an arrow $v:c \to F\,b$ in $C$ is the square with boundary $[v,\id_b,F\,u,F\,u \circ v]$, see the left square of~\Cref{fig:lifts-gluing}. A cartesian lift of $u:b\to_B b'$ with respect to a term $w:c \to_C F\,b'$ in the fiber of $b'$ is given by the pullback square on the right of~\Cref{fig:lifts-gluing}.

A vertical arrow in the gluing fibration is exactly given by a square of the form:
\[\begin{tikzcd}
	{C\downarrow F} && c && {c'} \\
	&& {F\,b} && {F\,b} \\
	B && b && b
	\arrow[from=1-3, to=2-3]
	\arrow[Rightarrow, no head, from=2-3, to=2-5]
	\arrow[from=1-3, to=1-5]
	\arrow[from=1-5, to=2-5]
	\arrow[Rightarrow, no head, from=3-3, to=3-5]
	\arrow[two heads, from=1-1, to=3-1]
\end{tikzcd}\]
\begin{figure}
	\[\begin{tikzcd}
		{C \downarrow F} && c && c && {F\,b \times_{F\,b'} c} && c \\
		&& {F\,b} && {F\,b'} && {F\,b} && {F\,b'} \\
		B && b && b' && b && b'
		\arrow["v"', from=1-3, to=2-3]
		\arrow["{F\,u}"', from=2-3, to=2-5]
		\arrow[Rightarrow, no head, from=1-3, to=1-5]
		\arrow["{Fu \circ v}"{pos=0.6}, from=1-5, to=2-5]
		\arrow[from=1-7, to=2-7]
		\arrow["{F\,u}"', from=2-7, to=2-9]
		\arrow[from=1-7, to=1-9]
		\arrow["w", from=1-9, to=2-9]
		\arrow["\lrcorner"{anchor=center, pos=0.125}, draw=none, from=1-7, to=2-9]
		\arrow["u", from=3-3, to=3-5]
		\arrow["u", from=3-7, to=3-9]
		\arrow[two heads, from=1-1, to=3-1]
	\end{tikzcd}\]
	\caption{Cocartesian and cartesian lifts in the gluing fibration}
	\label{fig:lifts-gluing}
\end{figure}

The gluing bifibration of a functor betweeen lex Rezk types is always lex.

\begin{proposition}[Lexness of Artin gluing]\label{prop:gluing-lex}
	Let $F:B \to C$ be a functor between lex Rezk types. Then the bicartesian fibration $\gl(F): \commaty{C}{F} \fibarr B$ is a lex cartesian fibration.
\end{proposition}

\begin{proof}
	By~\Cref{prop:term-obj-fib,prop:pb-fib} it suffices to argue that all the fibers are lex Rezk types and that the cartesian transition maps are lex functors. Since $C$ is assumed to be lex, $\gl(F)$ is bicartesian (see also~\Cref{prop:cod-fib}), hence the cartesian reindexings are right adjoints so they preserve limits by~\cite[Theorem~3.11]{BM21}. Finite completeness of the fibers can be shown analogously to the $1$-categorical case keeping in mind~\Cref{fig:lifts-gluing} and that the terminal element in a slice Rezk type is given by the identity (see~\cite[Lemma~9.8]{RS17} for the dual version of the statement).
\end{proof}

We will freely make use of~\Cref{prop:gluing-lex} in the following sections whenever necessary.

\section{Beck--Chevalley families}\label{sec:bc-fam}

In categorical logic, Beck--Chevalley conditions say that substitution (\ie~pullback) commutes with existential quantification (\ie~dependent sums). This generalizes to the fibrational setting by considering cartesian arrows in place of substitutions (acting contravariantly), and cocartesian arrows in place of dependent sums (acting covariantly). We refer to~\cite[Section~6]{streicher2020fibered} for more explanation.

\subsection{Beck--Chevalley condition}

\begin{definition}[Beck--Chevalley condition, \protect{\cite[Definition~6.1]{streicher2020fibered}}]\label{def:bcc}
	Let $P:B \to \UU$ be a family over a Rezk type all of whose fibers are Rezk. Then $P$ is said to satisfy the \emph{Beck--Chevalley condition (BCC)} if for any dependent square of the form
	\[\begin{tikzcd}
		& {d'} && {e'} \\
		{\totalty{P}} & d && e \\
		& {a'} && {b'} \\
		B & a && b
		\arrow["{f'}", from=1-2, to=1-4]
		\arrow["{g'}"', from=1-2, to=2-2, cart]
		\arrow["g", from=1-4, to=2-4, cart]
		\arrow["f"', from=2-2, to=2-4, cocart]
		\arrow["{u'}", from=3-2, to=3-4]
		\arrow["u"', from=4-2, to=4-4]
		\arrow["v", from=3-4, to=4-4]
		\arrow["{v'}"', from=3-2, to=4-2]
		\arrow[two heads, from=2-1, to=4-1]
		\arrow["\lrcorner"{anchor=center, pos=0.125}, draw=none, from=3-2, to=4-4]
	\end{tikzcd}\]
	it holds that: if $f$ is cocartesian, and $g,g'$ are cartesian, then $f'$ is cocartesian.
\end{definition}

\begin{proposition}[Dual of the Beck--Chevalley conditions]\label{prop:dual-bcc}
	Let $P:B \to \UU$ be a family over a Rezk type $B$ all of whose fibers are Rezk. Then $P$ satisfies the BCC from~\Cref{def:bcc} if and only if it satisfies the \emph{dual BCC}, which says: Given any dependent square
	\[\begin{tikzcd}
		{d'} && {e'} \\
		d && e
		\arrow["{f'}", from=1-1, to=1-3, cocart]
		\arrow["{g'}"', from=1-1, to=2-1, cart]
		\arrow["g", from=1-3, to=2-3]
		\arrow["f"', from=2-1, to=2-3, cocart]
	\end{tikzcd}\]
	in $P$ over a pullback, then: If $f$ and $f'$ are cocartesian, and $g'$ is cartesian, then $g$ is cartesian as well.
\end{proposition}

\begin{proof}
	Assume, the BCC from~\Cref{def:bcc} is satisfied. We consider a square, factoring the arrow $g$ into a vertical arrow followed by a cartesian arrow, we obtain:
	\[\begin{tikzcd}
		{d'} && {e'} \\
		&&& {v^*e} \\
		d && e
		\arrow["{g'}"', from=1-1, to=3-1, cart]
		\arrow["f", from=3-1, to=3-3, cocart]
		\arrow["{f'}", from=1-1, to=1-3, cocart]
		\arrow["g", from=1-3, to=3-3]
		\arrow["m", squiggly, from=1-3, to=2-4]
		\arrow["k", from=2-4, to=3-3, cart]
		\arrow["{m'}"{description}, from=1-1, to=2-4, crossing over]
	\end{tikzcd}\]
	Then, applying the BCC to the ``smaller'' square (still lying over the same pullback since $m$ is vertical), $m' \defeq m \circ f'$ must be cocartesian. But then $m$ is, too, by right cancelation, as $f'$ is cocartesian. But since it is also vertical, it is an isomorphism, so $g$ is cartesian, as claimed.
	
	The converse direction is analogous.
\end{proof}

\subsection{Beck--Chevalley families}

Preparing the treatment of Moens fibrations, we state a few first results about BC fibrations,~\aka~\emph{fibrations with internal sums}.

\begin{definition}
	A map between $\pi:E \fibarr B$ is a \emph{Beck--Chevalley fibration} or a \emph{cartesian fibration with internal sums} if:
	\begin{enumerate}
		\item The map $\pi$ is a bicartesian fibration, \ie~a cartesian and cocartesian fibration.
		\item The map $\pi$ satisfies the Beck--Chevalley condition.
	\end{enumerate}
\end{definition}

Recall the criterion characterizing cocartesian fibrations via the existence of a fibered left adjoint which acts as the ``cocartesian transport'' functor. The Beck--Chevalley condition is equivalent to this functor being \emph{cartesian}.
\begin{theorem}[Beck--Chevalley fibrations via cartesianness of the cocartesian transport functor, \cf~\protect{\cite[Theorem~6.1]{streicher2020fibered}}, \protect{\cite[Chapter~3, Theorem~1/(iii)]{LietzDip}}]\label{thm:bc-fib-cart-transp}
	Let $P: B \to \UU$ be a cartesian family over a Rezk type $B$ which has all pullbacks. We denote the unstraightening of $P$ by $\pi:E \fibarr B$ which is, in particular, a cartesian fibration. Then $\pi$ is Beck--Chevalley fibration if and only if the following conditions are satisfied:
	\begin{enumerate}
		\item The mediating fibered functor 
		\[ \iota_P: E \to_B \commaty{\pi}{B}, ~ \iota_P(b,e)\defeq \angled{b, \id_b, e} \]
		has a fibered left adjoint:
		\[\begin{tikzcd}
			E && {\pi \downarrow B} \\
			& B
			\arrow["\pi"', two heads, from=1-1, to=2-2]
			\arrow["{\partial_1'}", two heads, from=1-3, to=2-2]
			\arrow[""{name=0, anchor=center, inner sep=0}, "{\tau_\pi}"{description}, curve={height=12pt}, dashed, from=1-3, to=1-1]
			\arrow[""{name=1, anchor=center, inner sep=0}, "{\iota_\pi}"{description}, curve={height=6pt}, from=1-1, to=1-3]
			\arrow["\dashv"{anchor=center, rotate=-91}, draw=none, from=0, to=1]
		\end{tikzcd}\]
		\item The fibered left adjoint $\tau_\pi : \commaty{\pi}{B} \to E$ is a cartesian functor over $B$.
	\end{enumerate}

\end{theorem}

\begin{proof}
	Recall from \Cref{thm:cocartfams-via-transp} that the existence of the fibered left adjoint is equivalent to $\pi:E \fibarr B$ being a \emph{cocartesian} fibration (already without requiring $B$ to have pullbacks and $\pi$ to be a cartesian fibration). For $b:B$ the action of the fiberwise map $\tau \defeq \tau_\pi$ at $b:B$ is given by
	\[ \tau_b(v:a \to b, e:P(a)) \defeq \partial_1 \, v_!(e).\]
	Let $u:a \to b$ be a morphism in $B$. Over $u$, the action on arrows of $\tau$ maps a pair $\pair{\sigma}{f}$ consisting of a commutative square $\sigma$ in $B$ with boundary $[v',u',u,v]$ and a dependent arrow $f:d \to^P_{u'} e$ to the dependent arrow
	\[ \tau_u(\sigma,f) = \cocartFill_{P_!(v,e)\circ f}(P_!(v',e)),\]
	cf.~Fig.~\ref{fig:cocart-transp}.
	\begin{figure}
		\[\begin{tikzcd}
			& d && e && d && e \\
			{} & {a'} && {b'} \\
			& a && b && {v'_!(d)} && {v_!(e)}
			\arrow["{v'}"', from=2-2, to=3-2]
			\arrow[""{name=0, anchor=center, inner sep=0}, "u"', from=3-2, to=3-4]
			\arrow[""{name=1, anchor=center, inner sep=0}, "{u'}", from=2-2, to=2-4]
			\arrow["v", from=2-4, to=3-4]
			\arrow["f", from=1-2, to=1-4]
			\arrow[Rightarrow, dotted, no head, from=1-2, to=2-2]
			\arrow[Rightarrow, dotted, no head, from=1-4, to=2-4]
			\arrow["f", from=1-6, to=1-8]
			\arrow[""{name=2, anchor=center, inner sep=0}, "{P_!(v',d)}"{description, pos=0.3}, from=1-6, to=3-6, cocart]
			\arrow["{\tau_u(\sigma,f)}"', dashed, from=3-6, to=3-8]
			\arrow["{P_!(v,e)}"{description, pos=0.3}, from=1-8, to=3-8, cocart]
			\arrow["\rightsquigarrow"{description}, draw=none, from=2-4, to=2]
			\arrow["\sigma"{description}, Rightarrow, draw=none, from=1, to=0]
		\end{tikzcd}\]
		\caption{Action on arrows of $\tau$ (over $u:a \to b$ in $B$)}
		\label{fig:cocart-transp}
	\end{figure} 
	Recall the description of cartesian lifts in the family fibration from~\Cref{prop:cartlift-famfib}. Then, we claim that $\tau$ mapping these $L(\pi)$-cartesian lifts to $P$-cartesian arrows is equivalent to the Beck--Chevalley condition. To see this, consider an arbitrary cartesian arrow in $\partial_1' : \commaty{\pi}{B} \fibarr B$. This is given by the data of a square in the base and a cartesian arrow as in:
\[\begin{tikzcd}
	{(u')^*e} & e \\
	{a \times_b b'} & {b'} \\
	a & b
	\arrow["f", from=1-1, to=1-2, cart]
	\arrow["{v'}"', from=2-1, to=3-1]
	\arrow["u", from=3-1, to=3-2]
	\arrow["v", from=2-2, to=3-2]
	\arrow["\lrcorner"{anchor=center, pos=0.125}, draw=none, from=2-1, to=3-2]
	\arrow["{u'}", from=2-1, to=2-2]
	\arrow[Rightarrow, dotted, no head, from=1-1, to=2-1]
	\arrow[Rightarrow, dotted, no head, from=1-2, to=2-2]
\end{tikzcd}\]
Applying $\tau_\pi$ to this data yields the arrow $g$ as in:
\[\begin{tikzcd}
	{(u')^*e} & e \\
	{(v')_!(u')^*d} & {v_!(e)}
	\arrow[from=1-1, to=2-1, cocart]
	\arrow["g", dashed, from=2-1, to=2-2]
	\arrow[from=1-2, to=2-2, cocart]
	\arrow["f", from=1-1, to=1-2, cart]
\end{tikzcd}\]
If $\pi \colon E \fibarr B$ is a Beck--Chevalley fibration then (by the dual BCC~\Cref{prop:dual-bcc}) $g$ is cartesian, and so $\tau_\pi$ is a cartesian functor.

On the other hand, assuming $\tau_\pi$ to be a cartesian functor means that the arrow $g$ as above is always cartesian. Since we are considering arbitrary pullbacks in the base together with arbitrary dependent squares lying over them, this entails the dual BCC from~\Cref{prop:dual-bcc} which is equivalent to~\Cref{def:bcc}.
\end{proof}

Next up is a useful result stating that any functor (between Rezk types with pullbacks) preserves pullback if and only if its Artin gluing satisfies the Beck--Chevalley condition.

\begin{proposition}[Internal sums for gluing, \protect{\cite[Lemma~13.2]{streicher2020fibered}}]\label{prop:bcc-for-pb-pres}
	Let $A$ and $B$ be Rezk types with pullbacks and $F:A \to B$ an arbitrary functor (hence all its fibers are Rezk). Then the following are equivalent:
	\begin{enumerate}
		\item\label{it:gl-intsums-i} The functor $F$ preserves pullbacks.
		\item\label{it:gl-intsums-ii} The gluing fibration $\gl(F) \jdeq \partial_1: \commaty{B}{F} \fibarr A$ is a Beck--Chevalley fibration.
	\end{enumerate}
\end{proposition}

\begin{proof}
	We note first that since $B$ has all pullbacks, $\gl(F): \commaty{B}{F} \fibarr A$ is a cartesian fibration since it is a pullback of the fundamental fibration $\partial_1 : B^{\Delta^1} \fibarr B$. But $\partial_1 : B^{\Delta^1} \fibarr B$ is always a cocartesian fibration. Hence, $\gl(F)$ is a bifibration if $B$ has pullbacks.
	
	\begin{description}
		\item[$\ref{it:gl-intsums-i} \implies \ref{it:gl-intsums-ii}$] For a pullback square in $A$ we consider a square lying over in the gluing fibration which is a cube as in Fig.~\ref{fig:gluing-depsqare}
		\begin{figure}
			\[\begin{tikzcd}
				& {a'\times_ac} && {b'\times_bc} \\
				&& c && c \\
				{B \downarrow F} & {a'\defeq b'\times_ba} && {b'} \\
				&& a && b \\
				A & {d'} && {e'} \\
				&& d && e
				\arrow["m"{description}, from=5-2, to=5-4]
				\arrow["g"{description}, from=5-2, to=6-3]
				\arrow["f"{description}, from=6-3, to=6-5]
				\arrow["k"{description}, from=5-4, to=6-5]
				\arrow["\lrcorner"{anchor=center, pos=0.125, rotate=45}, draw=none, from=5-2, to=6-5]
				\arrow["v"{description}, from=3-2, to=4-3]
				\arrow["u"{description}, from=4-3, to=4-5]
				\arrow["w"{description}, from=3-4, to=4-5]
				\arrow["r"{description, pos=0.7}, from=3-2, to=3-4]
				\arrow["{v^*v'}"{description}, from=1-2, to=3-2]
				\arrow["{(v')^*v}"{description}, from=1-2, to=2-3]
				\arrow["{uv'}"{description, pos=0.3}, from=2-5, to=4-5]
				\arrow["{r'}"{description}, from=1-2, to=1-4]
				\arrow["{(uv')^*w}"{pos=0.7}, from=1-4, to=2-5]
				\arrow["{w^*(uv')}"{description, pos=0.8}, from=1-4, to=3-4]
				\arrow["\lrcorner"{anchor=center, pos=0.125, rotate=-45}, draw=none, from=1-4, to=4-5]
				\arrow["\lrcorner"{anchor=center, pos=0.125, rotate=-45}, draw=none, from=1-2, to=4-3]
				\arrow["\lrcorner"{anchor=center, pos=0.125, rotate=45}, draw=none, from=3-2, to=4-5]
				\arrow[two heads, from=3-1, to=5-1]
				\arrow[Rightarrow, no head, from=2-3, to=2-5, crossing over]
				\arrow["{v'}"{description, pos=0.3}, from=2-3, to=4-3, crossing over]
			\end{tikzcd}\]
			\caption{Verifying the Beck--Chevalley condition}
			\label{fig:gluing-depsqare}
		\end{figure}
		
		where the pullback square at the bottom of the cube is the image of the given pullback square in $C$. By composition and right cancelation of pullbacks, the top face of the cube in~\Cref{fig:gluing-depsqare} is also a pullback:
		Then, $r'$ turns out to be an isomorphism. By Rezk-completeness, it can be taken to be the identity $\id_c: c \to c$, exhibiting the back square of the cube as a cocartesian arrow in the glueing fibration, as desired.
		
		\item[$\ref{it:gl-intsums-ii} \implies \ref{it:gl-intsums-i}$] From the Beck--Chevalley condition we obtain, for any pullback square in $A$ a commutative cube above as follows:
		\[\begin{tikzcd}
			& {a'} && {a'} \\
			&& c && c \\
			{B \downarrow F} & {a'} && {b'} \\
			&& a && b \\
			A & {d'} && {e'} \\
			&& d && e
			\arrow["m"{description}, from=5-2, to=5-4]
			\arrow["g"{description}, from=5-2, to=6-3]
			\arrow["f"{description}, from=6-3, to=6-5]
			\arrow["k"{description}, from=5-4, to=6-5]
			\arrow["\lrcorner"{anchor=center, pos=0.125, rotate=45}, draw=none, from=5-2, to=6-5]
			\arrow["v"{description}, from=3-2, to=4-3]
			\arrow["u"{description}, from=4-3, to=4-5]
			\arrow["w"{description}, from=3-4, to=4-5]
			\arrow["r"{description, pos=0.7}, from=3-2, to=3-4]
			\arrow[Rightarrow, no head, from=1-2, to=3-2]
			\arrow[from=1-2, to=2-3]
			\arrow["{uv'}"{description, pos=0.3}, from=2-5, to=4-5]
			\arrow[Rightarrow, no head, from=1-2, to=1-4]
			\arrow["{(uv')^*w}"{pos=0.7}, from=1-4, to=2-5]
			\arrow["{w^*(uv')}"{description, pos=0.75}, from=1-4, to=3-4]
			\arrow["\lrcorner"{anchor=center, pos=0.125, rotate=-45}, draw=none, from=1-4, to=4-5]
			\arrow["\lrcorner"{anchor=center, pos=0.125, rotate=-45}, draw=none, from=1-2, to=4-3]
			\arrow[two heads, from=3-1, to=5-1]
			\arrow[Rightarrow, no head, from=2-3, to=2-5, crossing over]
			\arrow[Rightarrow, no head, from=2-3, to=4-3, crossing over]
		\end{tikzcd}\]
		Since the right outer square is a pullback, by composition also the bottom square is. This shows that $\gl(F)$ preserves pullbacks.
	\end{description}
\end{proof}

\section{Moens families}\label{sec:moens-fam}

We are interested in a specific subclass of BC fibrations that go by the name of \emph{(l)extensive} or \emph{Moens fibrations}. These are a fibrational generalization of (l)extensive categories~\cite{CLWExt}. Ultimately, this leads to Moens' Theorem which says that Moens fibrations over a fixed base type can be identified with lex functors \emph{from} this type into some other lex type. This is crucial to develop the \emph{fibered view of geometric morphisms}, \cf~\cite[Section~15 \emph{et seq}]{streicher2020fibered}, \cite{StrFVGM}. Applications in realizability have been given by Frey in his doctoral thesis~\cite{FreyPhD,FreyMoensFib} and more recently by Frey--Streicher~\cite{FreStr-Tripos}. Again, we are adapting the reasoning from \cite[Section~15]{streicher2020fibered}.

\subsection{(Pre-)Moens families and internal sums}\label{ssec:pre-moens}

Recall from classical $1$-category theory that a category $\mathbb C$ with pullbacks and coproducts is \emph{extensive} (or \emph{lextensive} depending on convention) if and only if, for all small families $(A_i)_{i\in I}$ the induced functor $\prod_{i \in I} \mathbb C/A_i \to \mathbb C/\coprod_{i\in I} A_i$ is an equivalence. This is equivalent to the condition that injections of finite sums are stable under pullback, and for any family of squares
\[\begin{tikzcd}
	{B_k} && B \\
	{A_k} && {\coprod_{i \in I} A_i}
	\arrow["{g_k}", from=1-1, to=1-3]
	\arrow["{f_k}"', from=1-1, to=2-1]
	\arrow[from=2-1, to=2-3]
	\arrow[from=1-3, to=2-3]
\end{tikzcd}\]
all of these are pullbacks if and only if all the maps $g_k: B_k \to B$ exhibit $B$ as the coproduct cone. This generalizes fibrationally as follows.

\begin{definition}[Stable and disjoint internal sums]
	Let $P:B \to \UU$ be a lex fibration with internal sums over a Rezk type $B$. Then $P$ has \emph{stable} internal sums if cocartesian arrows are stable under arbitrary pullbacks. The internal sums of $P$ are \emph{disjoint}\footnote{In a category, a coproduct is \emph{disjoint} if the inclusion maps are monomorphisms, and the intersection of the summands is an initial object.} if for every cocartesian arrow $f:d \cocartarr^P e$ the fibered diagonal is cocartesian, too:
	\[\begin{tikzcd}
		d \\
		& {d \times_ed} && d \\
		& d && e
		\arrow["f"', from=3-2, to=3-4, cocart]
		\arrow[from=2-2, to=2-4]
		\arrow["f", from=2-4, to=3-4, cocart]
		\arrow[curve={height=-24pt}, Rightarrow, no head, from=1-1, to=2-4]
		\arrow[curve={height=18pt}, Rightarrow, no head, from=1-1, to=3-2]
		\arrow[from=2-2, to=3-2]
		\arrow["{\delta_f}", from=1-1, to=2-2, cocart]
		\arrow["\lrcorner"{anchor=center, pos=0.125}, draw=none, from=2-2, to=3-4]
	\end{tikzcd}\]
\end{definition}

\begin{definition}[(Pre-)Moens families]
	Let $B$ be a lex Rezk type. A lex Beck--Chevalley family $P:B \to \UU$ is a \emph{pre-Moens family} if it has stable internal sums. We call a pre-Moens family $P:B \to \UU$ \emph{Moens family} (or \emph{extensive family}  or \emph{pre-geometric family}) if, moreover, all its (stable) internal sums are also disjoint.
\end{definition}

An immediate result is the following.
\begin{lemma}[\cite{streicher2020fibered}, Lemma~15.1]\label{lem:cocart-leg-pb}
	Let $B$ be a lex Rezk type and $\pi:E \fibarr B$ be a Moens fibration. Then, for $d,e,e':E$, and cocartesian morphisms $g:e \cocartarr e'$, in any pullback of the following form, the gap map is cocartesian, too:
	\[\begin{tikzcd}
		d \\
		& {d'} && e \\
		& d && {e'}
		\arrow[from=2-2, to=3-2]
		\arrow["h"', from=3-2, to=3-4]
		\arrow[from=2-2, to=2-4]
		\arrow["g", from=2-4, to=3-4,cocart]
		\arrow["f", curve={height=-12pt}, from=1-1, to=2-4]
		\arrow[curve={height=12pt}, Rightarrow, no head, from=1-1, to=3-2]
		\arrow["\lrcorner"{anchor=center, pos=0.125}, draw=none, from=2-2, to=3-4]
		\arrow["k", dashed, from=1-1, to=2-2, cocart]
	\end{tikzcd}\]
\end{lemma}

\begin{proof}
	Consider the following diagram, arising from canonical factorizations over the pullbacks $d' \jdeq d \times_{e'} e$ and $e \times_e e'$, resp:
	\[\begin{tikzcd}
		d && e \\
		{d'} && {e \times_{e'} e} && e \\
		d && e && {e'}
		\arrow["k"', from=1-1, to=2-1, cocart]
		\arrow[from=2-1, to=2-3]
		\arrow["f", from=1-1, to=1-3]
		\arrow["{\delta_g}"', from=1-3, to=2-3, cocart]
		\arrow[from=2-1, to=3-1]
		\arrow["f"', from=3-1, to=3-3]
		\arrow[from=2-3, to=3-3]
		\arrow[from=2-3, to=2-5]
		\arrow["g"', from=3-3, to=3-5, cocart]
		\arrow["g", from=2-5, to=3-5, cocart]
		\arrow[Rightarrow, no head, from=1-3, to=2-5]
		\arrow["\lrcorner"{anchor=center, pos=0.125}, draw=none, from=1-1, to=2-3]
		\arrow["\lrcorner"{anchor=center, pos=0.125}, draw=none, from=2-1, to=3-3]
		\arrow["\lrcorner"{anchor=center, pos=0.125}, draw=none, from=2-3, to=3-5]
		\arrow[curve={height=24pt}, Rightarrow, no head, from=1-1, to=3-1, crossing over]
		\arrow[curve={height=40pt}, Rightarrow, no head, from=1-3, to=3-3, crossing over]
	\end{tikzcd}\]
	By disjointness of sums, $\delta_g: e \cocartarr e \times_{e'} e$ is cocartesian, and by general pullback stability, so is $k:d \cocartarr d'$.
\end{proof}

The preceding lemma can be used to characterize disjointness given that stable internal sums exist. 

\begin{proposition}[Characterizations of disjointness of stable internal sums, \cite{streicher2020fibered}, Lem.~15.2]\label{prop:char-disj-stable}
	Let $B$ be a lex Rezk type and $P:B \to \UU$ be a pre-Moens family. Then the following are equivalent:\footnote{Streicher \cite{streicher2020fibered} points out that Items~\ref{it:disj-pre-moens-ii}--\ref{it:disj-pre-moens-iv} only require stability of cocartesian arrows along \emph{vertical} maps.}
	\begin{enumerate}
		\item\label{it:disj-pre-moens-i} The family $P$ is a Moens family, \ie~internal sums are disjoint (and stable).
		\item\label{it:disj-pre-moens-ii} Cocartesian arrows in $P$ satisfy left canceling, \ie~if $g$, $g \circ f$ are cocartesian then so is $f$.
		\item\label{it:disj-pre-moens-iii} Cocartesian transport is conservative, \ie~if $k$ is vertical, and both $f$ and $f \circ k$ are cocartesian, then $k$ is an isomorphism.
		\item\label{it:disj-pre-moens-iv} Any dependent square in $P$ of the form
		\[\begin{tikzcd}
			d && e \\
			{d'} && {e'}
			\arrow["f"', squiggly, from=1-1, to=2-1]
			\arrow["g"', from=2-1, to=2-3, cocart]
			\arrow["h", from=1-1, to=1-3, cocart]
			\arrow["k", squiggly, from=1-3, to=2-3]
		\end{tikzcd}\]
		where $f$, $k$ are vertical and $g$, $h$ are cocartesian is a pullback.
	\end{enumerate}	
\end{proposition}

\begin{proof}
	\begin{description}
		\item[$\ref{it:disj-pre-moens-i}  \implies \ref{it:disj-pre-moens-ii}$] Let $f:e \to e'$ and $g:e' \cocartarr e''$ such that $g \circ f: e \cocartarr e''$ is cocartesian. Since $P$ is a Moens family we can apply~\Cref{lem:cocart-leg-pb} to obtain that the gap map $\varphi: e \cocartarr e' \times_{e''} e$ as in
		\[\begin{tikzcd}
			e \\
			& {e'''} && e \\
			& {e'} && {e''}
			\arrow["{g'}", from=2-2, to=3-2]
			\arrow["g"', from=3-2, to=3-4, cocart]
			\arrow[from=2-2, to=2-4]
			\arrow["gf", from=2-4, to=3-4, cocart]
			\arrow[curve={height=-18pt}, Rightarrow, no head, from=1-1, to=2-4]
			\arrow["f"', curve={height=12pt}, from=1-1, to=3-2]
			\arrow["\varphi", dashed, from=1-1, to=2-2, cocart]
			\arrow["\lrcorner"{anchor=center, pos=0.125}, draw=none, from=2-2, to=3-4]
		\end{tikzcd}\]
		is cocartesian. By stability, $g'$ is cocartesian, too, hence so is $f= g' \circ \varphi$.
		\item[$\ref{it:disj-pre-moens-ii}  \implies \ref{it:disj-pre-moens-i}$] By stability of sums, for a cocartesian arrow $f:d \cocartarr e$, the map $f^*f: d \times_e d \to d$ is cocartesian. Then, the gap map $\delta_d: d \to d \times_e d$ is cocartesian, since $\id_d = f^*f \circ \delta_d$.
		\item[$\ref{it:disj-pre-moens-ii}  \implies \ref{it:disj-pre-moens-iii}$] This follows since an arrow that is vertical and cocartesian necessarily is an isomorphism.
		\item[$\ref{it:disj-pre-moens-iii}  \implies \ref{it:disj-pre-moens-ii}$] Let $g$, $f$ be given such that $gf$ exists, and both $gf$ as well as $g$ are cocartesian. Consider the factorization $f = mh$ where $m$ is vertical and $h$ is cocartesian. By right cancelation of cocartesian arrows, since both $gf = (gm)h$ and $h$ are cocartesian, so must be $gm$. 
		\item[$\ref{it:disj-pre-moens-iii}  \implies \ref{it:disj-pre-moens-iv}$] Consider the induced pullback square:
		\[\begin{tikzcd}
			d \\
			& {e''} && e \\
			& {d'} && {e'}
			\arrow["{k'}"', squiggly, from=2-2, to=3-2]
			\arrow["{g'}", from=2-2, to=2-4, cocart]
			\arrow["k", squiggly, from=2-4, to=3-4]
			\arrow["h", curve={height=-18pt}, from=1-1, to=2-4, cocart]
			\arrow["f"', curve={height=18pt}, squiggly, from=1-1, to=3-2]
			\arrow["{f'}", squiggly, from=1-1, to=2-2, dashed]
			\arrow["\lrcorner"{anchor=center, pos=0.125}, draw=none, from=2-2, to=3-4]
			\arrow["g"', from=3-2, to=3-4, cocart]
		\end{tikzcd}\]
		Since $P$ is a bifibration, the vertical arrows are stable under pullback along any arrow, and satisfy left cancelation. Hence, since $k$ is vertical, so is $k'$, and consequently $f'$ as well (since $f$ is). By the assumption in $(\ref{it:disj-pre-moens-iii})$ since $f'$ is vertical and both $g'$ and $h=g' \circ f'$ are cocartesian $f'$ is an isomorphism.
		\item[$\ref{it:disj-pre-moens-iv}  \implies \ref{it:disj-pre-moens-iii}$] Any square of the form
		\[\begin{tikzcd}
			d && e \\
			{d'} && e
			\arrow["k"', squiggly, from=1-1, to=2-1]
			\arrow["f"', from=2-1, to=2-3,  cocart]
			\arrow["fk", from=1-1, to=1-3, cocart]
			\arrow[Rightarrow, no head, from=1-3, to=2-3]
			\arrow["\lrcorner"{anchor=center, pos=0.125}, draw=none, from=1-1, to=2-3]
		\end{tikzcd}\]
		is a pullback by precondition, hence $k$ is an identity.
		
	\end{description}
\end{proof}

\subsection{Extensive internal sums}\label{ssec:ext}

We can now provide a characterization of Moens families among the BC families. In particular, we obtain a fibered version of \emph{Lawvere extensivity} as an alternative characterization for (internal) extensivity. Classically, a category $\mathbb C$ is \emph{Lawvere-extensive} if for any small set $I$, the categories $\mathbb C^I$ and $\mathbb C/\coprod_{i\in I} \unit$ are canonically isomorphic. To prepare, consider first the following construction. Let $B$ be lex Rezk type and $P:B \to \UU$ be a cocartesian family.

\begin{definition}[Terminal transport functor]\label{def:term-transp}
	For a terminal element $z:B$, we define the functor\footnote{In \cite{streicher2020fibered}, the functor $\omega$ is called $\mathbf{\Delta}$.}\footnote{Note that we can suppress the dependency on a specified terminal element $z:B$.}
	\[ \omega_{P,z} \defeq \omega : \widetilde{P} \to P\,z, \quad \omega \defeq \lambda b,e.(!_b)_!(e). \]
\end{definition}
\begin{figure}
	\[\begin{tikzcd}
		e && {\omega(e)} \\
		{e'} && {\omega(e')} \\
		b && z \\
		{b'}
		\arrow["f"', from=1-1, to=2-1]
		\arrow["{P_!(!_{b'},e')}"', from=2-1, to=2-3,cocart]
		\arrow["{P_!(!_b,e)}", from=1-1, to=1-3,cocart]
		\arrow["{\omega(f)}", dashed, from=1-3, to=2-3]
		\arrow["u"', from=3-1, to=4-1]
		\arrow["{!_b}", from=3-1, to=3-3]
		\arrow["{!_{b'}}"', from=4-1, to=3-3]
	\end{tikzcd}\]
	\caption{Action on morphisms of the transport functor $\omega_{P,z}$}
	\label{fig:transp-term}
\end{figure}
The action on arrows of this functor is illustrated in \ref{fig:transp-term}. The arrow $\omega_f$ is vertical over the terminal element $z:B$, for any $f: \Delta^1 \to B$. 

\begin{definition}[Choice of terminal elements]\label{def:term-fib}
	Let $B$ be a Rezk type and $P:B \to \UU$ be family with Rezk fibers such that every fiber has a terminal element. Then we denote, by the Principle of Choice, the section\footnote{We have ocasionally considered a variant of this map as a function $B \to \totalty{P}$. We hope that this mild abuse of notation does not cause any confusion.} choosing fiberwise terminal elements by
	\[ \zeta_P \defeq \zeta: \prod_{b:B} P\,b,\]
	\ie~for any $b:B$ the element $\zeta_b : P\,b$ is terminal.
\end{definition}

We define
\[ \omega' \jdeq \omega'_P: B \to P\,z, \quad \omega'(b)\defeq \omega(\zeta_b) \jdeq (!_b)_!(\zeta_b).\]

We are now ready for the promised characterization.

\begin{proposition}[Stable disjoint sums in terms of extensive sums, \cite{streicher2020fibered}, Lem.~15.3]\label{prop:ext-sums}
	Let $B$ be a lex Rezk type and $P:B \to \UU$ be a Beck--Chevalley family. Then, the following are equivalent:
	\begin{enumerate}
		\item\label{it:ext-sums-stable} The family $P$ is a Moens family, \ie~$P$ has stable disjoint sums.
		\item\label{it:ext-sums-ext} The bicartesian family $P$ has \emph{internally extensive} sums, \ie~for vertical arrows $f:d \to d'$, $k:e \to e'$, cocartesian arrows $g: e \cocartarr e'$, in a square
		
		\[\begin{tikzcd}
			d && e \\
			{d'} && {e'}
			\arrow["f"', squiggly, from=1-1, to=2-1]
			\arrow["g"', from=2-1, to=2-3, cocart]
			\arrow["h", from=1-1, to=1-3]
			\arrow["k", squiggly, from=1-3, to=2-3]
		\end{tikzcd}\]
		the arrow $h:d \to e$ is a cocartesian arrow if and only if the square is a pullback.
		\item\label{it:ext-sums-lawvere} The internal sums in $P$ are \emph{Lawvere-extensive}, \ie~in any square of the form
		\[\begin{tikzcd}
			d && e \\
			{\zeta_a} && {\omega'(a)}
			\arrow["!_d^a"', from=1-1, to=2-1, squiggly]
			\arrow["{P_!(!_a, \zeta_a)}"', from=2-1, to=2-3, cocart]
			\arrow["h", from=1-1, to=1-3]
			\arrow["k", from=1-3, to=2-3, squiggly]
		\end{tikzcd}\]
		where $k:e \vertarr \omega'(a)$ is vertical the arrow $h:d \to e$ is cocartesian if and only if the given square is a pullback. Here, $\zeta$ and $\omega'$ are as defined in~\Cref{def:term-transp,def:term-fib}.
		\item\label{it:ext-sums-transp} Let $z:B$ be a terminal element in $B$. For any $a:B$, the transport functor $(!_a)_!: P\,a \to P\,z$ reflects isomorphisms and $k^*P_!(!_a,\zeta_a)$ is cocartesian in case $k$ is vertical.
	\end{enumerate}
\end{proposition}

Again, as remarked by Streicher, the equivalences between all but the first statement hold already in the case that cocartesian arrows are only stable under pullback along vertical arrows (in a lex bifibration).
\begin{proof}
	\begin{description}
		\item[$(\ref{it:ext-sums-stable}) \implies (\ref{it:ext-sums-ext})$] Consider a square as given in (\ref{it:ext-sums-ext}). If it is a pullback we have an identification $h= k^*g$, and by stability $h$ is cocartesian, too. Conversely, given such a square where $h$ is cocartesian, consider the factorization:
		\[\begin{tikzcd}
			d \\
			& {d''} && e \\
			& {d'} && {e'}
			\arrow["k'", from=2-2, to=3-2, squiggly]
			\arrow["g"', from=3-2, to=3-4, cocart]
			\arrow["{g'}"', from=2-2, to=2-4, cocart]
			\arrow["k", from=2-4, to=3-4, squiggly]
			\arrow["h", curve={height=-12pt}, from=1-1, to=2-4, cocart]
			\arrow["f"', curve={height=12pt}, squiggly, from=1-1, to=3-2]
			\arrow["{f'}"', dashed, from=1-1, to=2-2]
			\arrow["\lrcorner"{anchor=center, pos=0.125}, draw=none, from=2-2, to=3-4]
		\end{tikzcd}\]
		The arrow $g' = k^*g$ is cocartesian by stability of sums. The arrow $k' = g^*k$ is vertical since vertical arrows in a cocartesian fibration are stable under pullback by~\Cref{thm:pb-of-vert-is-vert-cocart-fam}. By the same reason, they are left cancelable, hence $f'$ is vertical. But since $P$ is a Moens family, by Prop.~\ref{prop:char-disj-stable}, cocartesian arrows also satisfy left cancelation, hence $f'$ is cocartesian, too, and thus an equivalence.
		\item[$(\ref{it:ext-sums-ext}) \implies (\ref{it:ext-sums-lawvere})$] The latter is an instance of the former.
		\item[$(\ref{it:ext-sums-lawvere}) \implies (\ref{it:ext-sums-ext})$] First, consider the following diagram:
			
\[\begin{tikzcd}
	d && {d'} \\
	e && {e'} \\
	{\zeta_a} && {\omega'(a)}
	\arrow["k"', squiggly, from=1-1, to=2-1]
	\arrow["m \jdeq !_e^a"', squiggly, from=2-1, to=3-1]
	\arrow[""{name=0, anchor=center, inner sep=0}, "g", from=1-1, to=1-3]
	\arrow[""{name=1, anchor=center, inner sep=0}, "f"', from=2-1, to=2-3, cocart]
	\arrow["{k'}", squiggly, from=1-3, to=2-3]
	\arrow["{m'}", squiggly, from=2-3, to=3-3]
	\arrow[""{name=2, anchor=center, inner sep=0}, "{P_!(!_a,\zeta_a)}"', from=3-1, to=3-3, cocart]
	\arrow["{(*)}"{description}, draw=none, from=0, to=1]
	\arrow["{(**)}"{description}, draw=none, from=1, to=2]
\end{tikzcd}\]

		We will first establish the claim that:
		\[ \text{$g$ is cocartesian iff $(*)$ is a pullback.} \tag{$\dagger$}\]
		By Lawvere extensivity~(\ref{it:ext-sums-lawvere}) the square $(**)$ is a pullback. If $(*)$ is a pullback, then the outer square is a pullback, so by Lawvere extensivity $g$ is cocartesian. Conversely, if $g$ is cocartesian, by Lawvere extensivity the outer square is a pullback, so by canceling $(*)$ is a pullback, too.

		Now, consider a square as follows:
\[\begin{tikzcd}
	y && d \\
	x && e
	\arrow["r"', squiggly, from=1-1, to=2-1]
	\arrow[""{name=0, anchor=center, inner sep=0}, "h", from=1-1, to=1-3]
	\arrow[""{name=1, anchor=center, inner sep=0}, "\ell"', from=2-1, to=2-3, cocart]
	\arrow["k", squiggly, from=1-3, to=2-3]
	\arrow["{(+)}"{description}, draw=none, from=1, to=0]
\end{tikzcd}\]
		We are to show that $(+)$ is a pullback if and only if $h$ is cocartesian. Assume that $h$ is cocartesian. Consider the following pasted diagram, involving the diagram $(*)$ from before:
\[\begin{tikzcd}
	y && d && {d'} \\
	x && e && {e'}
	\arrow["r"', squiggly, from=1-1, to=2-1]
	\arrow[""{name=0, anchor=center, inner sep=0}, "h", from=1-1, to=1-3, cocart]
	\arrow[""{name=1, anchor=center, inner sep=0}, "\ell"', from=2-1, to=2-3, cocart]
	\arrow["k", squiggly, from=1-3, to=2-3]
	\arrow["g", from=1-3, to=1-5, cocart]
	\arrow["f"', from=2-3, to=2-5, cocart]
	\arrow["{k'}", squiggly, from=1-5, to=2-5]
	\arrow["\lrcorner"{anchor=center, pos=0.125}, draw=none, from=1-3, to=2-5]
	\arrow["{(+)}"{description}, draw=none, from=1, to=0]
\end{tikzcd}\]
		The square on the right hand side is a diagram by $(\dagger)$. The outer square gives rise to the following pasted diagram
\[\begin{tikzcd}
	y && {d'} \\
	x && {e'} \\
	{\zeta_b} && {\omega'(b)}
	\arrow["gh", from=1-1, to=1-3, cocart]
	\arrow["r"', squiggly, from=1-1, to=2-1]
	\arrow["f\ell"', from=2-1, to=2-3, cocart]
	\arrow["{k'}", squiggly, from=1-3, to=2-3]
	\arrow["{m''}", squiggly, dashed, from=2-3, to=3-3]
	\arrow["{!_x^b}"', squiggly, from=2-1, to=3-1]
	\arrow[from=3-1, to=3-3, cocart]
\end{tikzcd}\]
		where $m''$ is the (vertical) filler induced by the cocartesian arrow $f\ell$. The lower square is a pullback by Lawvere extensivity~(\ref{it:ext-sums-lawvere}). Then by~$(\dagger)$ the upper composite square is a pullback. Thus, the square $(+)$ is a pullback, too, by canceling.

		Assume, conversely, that the square $(+)$ is a pullback. Then, in the diagram below the square on the right is a pulllback, too, by $(\dagger)$:
\[\begin{tikzcd}
	y && d && {d'} \\
	x && e && {e'}
	\arrow["r"', squiggly, from=1-1, to=2-1]
	\arrow["h", from=1-1, to=1-3]
	\arrow["\ell"', from=2-1, to=2-3, cocart]
	\arrow["k"', squiggly, from=1-3, to=2-3]
	\arrow["g", from=1-3, to=1-5, cocart]
	\arrow["f"', from=2-3, to=2-5, cocart]
	\arrow["{k'}", squiggly, from=1-5, to=2-5]
	\arrow["\lrcorner"{anchor=center, pos=0.125}, draw=none, from=1-1, to=2-3]
	\arrow["\lrcorner"{anchor=center, pos=0.125}, draw=none, from=1-3, to=2-5]
\end{tikzcd}\]
		The composite square is a pullback, too. Then, again, as seen before $gh$ is cocartesian by $(\dagger)$. Then consider the following diagram:
\[\begin{tikzcd}
	y && {y'} && {y''} \\
	&& d && {d'}
	\arrow["{h'}", from=1-1, to=1-3, cocart]
	\arrow["{g'}", from=1-3, to=1-5, cocart]
	\arrow["{h''}", squiggly, from=1-3, to=2-3]
	\arrow["{h'g''}", squiggly, from=1-5, to=2-5]
	\arrow["g", from=2-3, to=2-5, cocart]
	\arrow["h", from=1-1, to=2-3]
	\arrow["gh"', curve={height=40pt}, from=1-1, to=2-5, cocart]
\end{tikzcd}\]
		Since both $g'h'$ and $gh$ are cocartesian, $g''$ is an isomorphism. By $(\dagger)$ the right square is a pullback, so $h''$ is an isomorphism, too. But then $h$ is cocartesian, too.

		\item[$(\ref{it:ext-sums-lawvere}) \implies (\ref{it:ext-sums-transp})$]
		Let $f: d \vertarr d'$ be a vertical arrow in $P\,a$ such that $(!_a)_!(f):(!_a)_!(d) \to (!_a)_!(d')$ is a path. Then:
		\[\begin{tikzcd}
			d && {(!_a)_!(d)} \\
			{d'} && {(!_a)_!(d')}
			\arrow["f"', squiggly, from=1-1, to=2-1]
			\arrow[from=1-1, to=1-3,  cocart]
			\arrow[from=2-1, to=2-3,  cocart]
			\arrow["{(!_a)_!(f)}", Rightarrow, no head, from=1-3, to=2-3]
		\end{tikzcd}\]
		Consider the cube induced by cocartesian filling \wrt~$P_!(!_a,\zeta_a) \circ !_d$ and $P_!(!_a,\zeta_a) \circ !_d'$, respectively:
		\[\begin{tikzcd}
			d && {(!_a)_!(d)} \\
			& {\zeta_a} && {\omega'(a)} \\
			{d'} && {(!_a)_!(d')} \\
			& {\zeta_a} && {\omega'(a)}
			\arrow["{!_{d'}}", squiggly, from=3-1, to=4-2, swap]
			\arrow[from=4-2, to=4-4,  cocart, near end]
			\arrow[from=3-1, to=3-3,  cocart, near end]
			\arrow[squiggly, from=3-3, to=4-4, "\tyfill"]
			\arrow[from=1-1, to=3-1, "f", swap, squiggly]
			\arrow[from=1-1, to=1-3,  cocart, near end]
			\arrow[Rightarrow, no head, from=1-3, to=3-3]
			\arrow[Rightarrow, no head, from=2-2, to=4-2, curve={height=15pt}]
			\arrow["{!_d}", squiggly, from=1-1, to=2-2, swap]
			\arrow[squiggly, from=1-3, to=2-4, "\tyfill"]
			\arrow["\lrcorner"{anchor=center, pos=0.125, rotate=45}, draw=none, from=3-1, to=4-4]
			\arrow["\lrcorner"{anchor=center, pos=0.125, rotate=45}, draw=none, from=1-1, to=2-4]
			\arrow[Rightarrow, no head, from=2-4, to=4-4, crossing over]
			\arrow[from=2-2, to=2-4,  cocart, near end, crossing over]
		\end{tikzcd}\]
		The bottom and top squares are pullbacks by Lawvere extensivity. Then, by~\cite[Proposition~A.2]{W22-2sCart}, the map $f$ is an identity as well.
		
		For $a:B$ and a vertical arrow $k: e \vertarr_z \omega'(a)$ consider the pullback square:
		\[\begin{tikzcd}
			{e'} && e \\
			{\zeta_a} && {\omega'(a)}
			\arrow["{k'}"', from=1-1, to=2-1]
			\arrow["{P_!(!_a,\zeta_a)}"', from=2-1, to=2-3, cocart]
			\arrow["f", from=1-1, to=1-3]
			\arrow["k", squiggly, from=1-3, to=2-3]
			\arrow["\lrcorner"{anchor=center, pos=0.125}, draw=none, from=1-1, to=2-3]
		\end{tikzcd}\]
		Since $P$ is in particular a cocartesian fibration, $k'$ is vertical by~\Cref{thm:pb-of-vert-is-vert-cocart-fam}. By Lawvere extensivity, $f$ is cocartesian.
		\item[$(\ref{it:ext-sums-transp}) \implies (\ref{it:ext-sums-lawvere})$] Given a square as in $(\ref{it:ext-sums-lawvere})$, we see that the arrow $h$ is cocartesian if and only if it is a pullback, by the second condition in $(\ref{it:ext-sums-lawvere})$.
	\end{description}
\end{proof}

\begin{corollary}[Cocartesian transport and slices in Moens families~\protect{\cite[Corollary~15.4]{streicher2020fibered}}]
	Let $B$ be a Rezk type and $P:B \to \UU$ a Moens family.
	
	Then, for all arrows $u:a \to b$ in $B$ and points $d:P\,a$, the functor
	\begin{align*}
		\commaty{u_!}{d} &: \commaty{P\,a}{d} \to \commaty{P\,b}{u_!d}, \\
		\commaty{u_!}{d} & (d', f : d' \vertarr_a d) \defeq \pair{u_! d'}{f' : u_!\,d' \vertarr_b u_!\,d} ,
	\end{align*}
	where $f' \defeq \cocartFill_{P_!(u,d')}(P_!(u,d) \circ f)$,
	is an equivalence.
	In particular,  we have equivalences
	\[ \commaty{u_!}{\zeta_a} : \commaty{P\,a}{\zeta_a} \equiv \commaty{P\,b}{u_!\zeta_a}, \quad \commaty{(!_a)_!}{\zeta_a}:\commaty{P\,a}{\zeta_a} \simeq \commaty{P\,z}{\omega'(a)},\]
	and by terminality we have $P\,a \equiv \commaty{P\,a}{\zeta_a}$.
\end{corollary}

\begin{proof}
	First, note that application of $\commaty{u_!}{d}$ to a vertical arrow $f : d' \vertarr_a d$ gives rise to the square:
	\[\begin{tikzcd}
		{d'} && {u_!\,d'} \\
		d && {u_!\,d}
		\arrow["{P_!(u,d')}", cocart, from=1-1, to=1-3]
		\arrow["f"', squiggly, from=1-1, to=2-1]
		\arrow["{P_!(u,d)}"', cocart, from=2-1, to=2-3]
		\arrow["{f'}", squiggly, from=1-3, to=2-3]
	\end{tikzcd}\]
	Since $P$ is Moens the square is a pullback, by~\Cref{prop:ext-sums}~(\ref{it:ext-sums-ext}).
	
	We claim that an inverse to $\commaty{u_!}{d}$ is given by the pullback functor
	\begin{align*} 
		P_!(u,d)^* & : \commaty{P\,b}{u_!\,d} \to \commaty{P\,a}{d}, \\
		P_!(u,d)^* & (e, g : e \vertarr_b u_!\,d) \defeq \pair{d \times_{u_!\,d} e}{ P_!(u,d)^*(g) : d \times_{u_!\,d} e \vertarr_a d}.
	\end{align*}
	
	To prove $P_!(u,d)^*  \circ \commaty{u_!}{d} = \id_{ \commaty{P\,a}{d}}$, we start with a vertical arrow $f : d' \vertarr_a d$. Pulling back $f' :u_!\, d' \vertarr_b u_!\, d$ along $P_!(u,d)$ gives back $f$ again as just argued.
	
	For the other identification $\commaty{u_!}{d} \circ P_!(u,d)^* = \id_{\commaty{P\,b}{u_!\,d}}$, consider a vertical arrow $g : e \vertarr_b u_!\,d$. Applying $ P_!(u,d)^*$ gives the pullback square:
	\[\begin{tikzcd}
		{d \times_{u_!\,d}e} && e \\
		d && {u_!\,d}
		\arrow["{P_!(u,d)^*g}"', squiggly, from=1-1, to=2-1]
		\arrow["{P_!(u,d)}"', from=2-1, to=2-3, cocart]
		\arrow["{\overline{g}}", from=1-1, to=1-3]
		\arrow["g", squiggly, from=1-3, to=2-3]
		\arrow["\lrcorner"{anchor=center, pos=0.125}, draw=none, from=1-1, to=2-3]
	\end{tikzcd}\]
	Again, by~\Cref{prop:ext-sums}~(\ref{it:ext-sums-ext}), this means that $\overline{g} : d \times_{u_!\,d}e \to e$ is cocartesian, so the above square is identified with:
	\[\begin{tikzcd}
		{d \times_{u_!\,d}e} && e \\
		d && {u_!\,d}
		\arrow["{P_!(u,d)^*g}"', squiggly, from=1-1, to=2-1]
		\arrow["{P_!(u,d)}"', from=2-1, to=2-3, cocart]
		\arrow["{P_!(u,d \times_{u_!\,d} e)}", cocart, from=1-1, to=1-3]
		\arrow["g", squiggly, from=1-3, to=2-3]
		\arrow["\lrcorner"{anchor=center, pos=0.125}, draw=none, from=1-1, to=2-3]
	\end{tikzcd}\]
	But this exhibits $g$ to be (homotopic to) the filler
	\[ \cocartFill_{{P_!(u,d \times_{u_!\,d} e)}}(P_!(u,d) \circ P_!(u,d)^*g),\]
	as desired. 
\end{proof}

\begin{corollary}[Pullback preservation of cocartesian transport in Moens families, \protect{\cite[Corollary~15.5]{streicher2020fibered}}]\label{cor:lex-transp-moens}
	Let $B$ be a Rezk type and $P:B \to \UU$ a Moens family.
	
	Then for all $u:a \to b$ in $B$, the covariant transport functor
	\[ u_!:P\,a \to P\,b\]
	preserves pullbacks.
\end{corollary}

\begin{proof}
	Observe that $u_! : P\,a \to P\,b$ factors as follows:
	\[\begin{tikzcd}
		{P\,a} & {P\,a\downarrow \zeta_a} && {P\,b\downarrow u_!\, \zeta_a} && {P\,b} \\
		\\
		a & {\begin{pmatrix} a \\ \Big \downarrow \\ \zeta_a\end{pmatrix}} && {\begin{pmatrix} u_!\,a \\ \Big \downarrow  \\ u_!\,\zeta_a\end{pmatrix}} && {u_!\,a}
		\arrow["{u!\downarrow \zeta_a}", "\simeq"', from=1-2, to=1-4]
		\arrow["{\partial_0}", from=1-4, to=1-6]
		\arrow["\simeq"{description}, draw=none, from=1-1, to=1-2]
		\arrow["{u_!}"{description}, curve={height=24pt}, from=1-1, to=1-6]
		\arrow[maps to, from=3-4, to=3-6]
		\arrow[maps to, from=3-1, to=3-2]
		\arrow[maps to, from=3-2, to=3-4]
	\end{tikzcd}\]
	The first two maps are equivalences, so they preserve pullbacks. The domain projection 
\end{proof}

As a crucial ingredient for Moens' Theorem, the gluing of pullback-preserving functors always is a Moens fibration.
\begin{proposition}\label{prop:gluing-moens}
	Let $B$ and $C$ be a lex Rezk types and $F:B \to C$ be a pullback-preserving functor. Then $\gl(F): \commaty{C}{F} \fibarr B$ is a Moens fibration.	
\end{proposition}

\begin{proof}
	Since $F$ preserves pullback, $\gl(F)$ is a Beck--Chevalley fibration by~\Cref{prop:bcc-for-pb-pres}. Then, by~\Cref{prop:ext-sums}, it suffices to prove that the internal sums in $\gl(F)$ are extensive. But this follows from considering a dependent cube as given in~\Cref{fig:gluing-moens} and the fact that $k:c \to c'$ is an isomorphism if and only if the top square is a pullback.\footnote{The identities in the cube are part of the prerequisites to prove extensivity.}
\end{proof}

\begin{figure}
	\[\begin{tikzcd}
		&&& c && {c'} \\
		&&&& {F\,b'} && {F\,b'} \\
		{C \downarrow F} &&& {F\,b} && {F\,b'} \\
		&&&& {F\,b} && {F\,b'} \\
		B &&& b && {b'} \\
		&&&& b && {b'}
		\arrow[Rightarrow, no head, from=5-4, to=6-5]
		\arrow[from=5-4, to=5-6]
		\arrow[Rightarrow, no head, from=5-6, to=6-7]
		\arrow[from=6-5, to=6-7]
		\arrow["h"{description, pos=0.3}, from=1-4, to=3-4]
		\arrow[Rightarrow, no head, from=3-4, to=4-5]
		\arrow["{F\,u}"{description}, from=4-5, to=4-7]
		\arrow[Rightarrow, no head, from=3-6, to=4-7]
		\arrow["{F\,u}"{description, pos=0.7}, from=3-4, to=3-6]
		\arrow["f"{description}, from=1-4, to=2-5]
		\arrow["k"{description}, from=1-4, to=1-6]
		\arrow["{f'}"{description}, from=1-6, to=2-7]
		\arrow["{Fu\circ g}"{description, pos=0.3}, from=2-7, to=4-7]
		\arrow["\lrcorner"{anchor=center, pos=0.125, rotate=45}, draw=none, from=5-4, to=6-7]
		\arrow["\lrcorner"{anchor=center, pos=0.125, rotate=45}, draw=none, from=3-4, to=4-7]
		\arrow["{h'}"{description, pos=0.3}, from=1-6, to=3-6]
		\arrow[two heads, from=3-1, to=5-1]
		\arrow["g"{description, pos=0.3}, from=2-5, to=4-5, crossing over]
		\arrow[Rightarrow, no head, from=2-5, to=2-7, crossing over]
	\end{tikzcd}\]
	\caption{Extensivity of sums in $\gl(F)$}
	\label{fig:gluing-moens}
\end{figure}
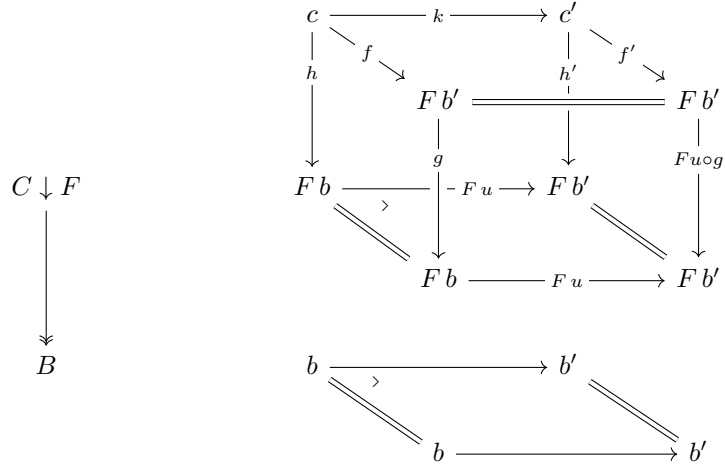

\begin{lemma}[\cf~\protect{\cite[Lemma~15.6]{streicher2020fibered}}]\label{lem:gap-to-vert-cocart}
	Let $B$ be a lex Rezk type and $P:B \to \UU$ be a Moens fibration. Then the gap arrow in any diagram of the form
	
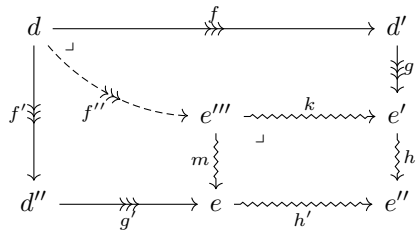
\[\begin{tikzcd}
		d &&&& {d'} \\
		&& {e'''} && {e'} \\
		{d''} && e && {e''}
		\arrow["{f'}"', from=1-1, to=3-1,cocart]
		\arrow["f", from=1-1, to=1-5,cocart]
		\arrow["\lrcorner"{anchor=center, pos=0.125}, draw=none, from=2-3, to=3-5]
		\arrow["{h'}"', from=3-3, to=3-5, squiggly]
		\arrow["{g'}"', from=3-1, to=3-3,cocart]
		\arrow["g", from=1-5, to=2-5, cocart]
		\arrow["h", squiggly, from=2-5, to=3-5]
		\arrow["k", from=2-3, to=2-5, squiggly]
		\arrow["m"', squiggly, from=2-3, to=3-3]
		\arrow["\lrcorner"{anchor=center, pos=0.125}, draw=none, from=1-1, to=2-3]
		\arrow["{f''}"', curve={height=12pt}, dashed, from=1-1, to=2-3, cocart]
	\end{tikzcd}\]
	is cocartesian as well.
\end{lemma}

\begin{proof}
	Since $P$ is a Moens family cocartesian arrows are stable under pullback along vertical arrows (by \emph{internal extensivity}, see~\Cref{prop:ext-sums}, (\ref{it:ext-sums-ext})), and since $P$ is a cocartesian family vertical arrows are stable under pullback along all arrows (by~\Cref{thm:pb-of-vert-is-vert-cocart-fam}). By the Pullback Lemma, this gives rise to the following diagram:
	
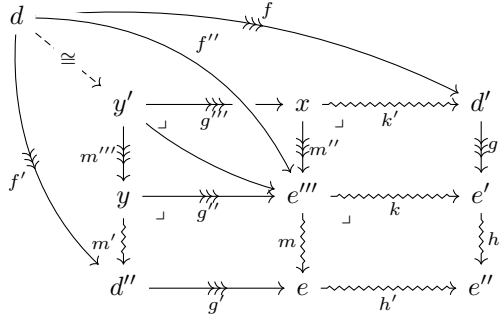
\[\begin{tikzcd}
		d \\
		& {y'} && x && {d'} \\
		& y && {e'''} && {e'} \\
		& {d''} && e && {e''}
		\arrow["{g'}"', from=4-2, to=4-4, cocart]
		\arrow["{h'}"', squiggly, from=4-4, to=4-6]
		\arrow["h", squiggly, from=3-6, to=4-6]
		\arrow["g", from=2-6, to=3-6, cocart]
		\arrow["m"', squiggly, from=3-4, to=4-4]
		\arrow["{m''}"', from=2-4, to=3-4, cocart, swap]
		\arrow["{k'}"', squiggly, from=2-4, to=2-6]
		\arrow["k"', squiggly, from=3-4, to=3-6]
		\arrow["\lrcorner"{anchor=center, pos=0.125}, draw=none, from=3-4, to=4-6]
		\arrow["\lrcorner"{anchor=center, pos=0.125}, draw=none, from=2-4, to=3-6]
		\arrow["{m'}"', squiggly, from=3-2, to=4-2]
		\arrow["{g''}"', from=3-2, to=3-4, cocart]
		\arrow["{m'''}"', from=2-2, to=3-2, cocart]
		\arrow["{g'''}"', from=2-2, to=2-4, cocart]
		\arrow["\lrcorner"{anchor=center, pos=0.125}, draw=none, from=2-2, to=3-4]
		\arrow["\lrcorner"{anchor=center, pos=0.125}, draw=none, from=3-2, to=4-4]
		\arrow["f", curve={height=-18pt}, from=1-1, to=2-6, cocart]
		\arrow["{f'}"', curve={height=18pt}, from=1-1, to=4-2, cocart]
		\arrow["\cong"{description}, dashed, from=1-1, to=2-2]
		\arrow[curve={height=6pt}, from=2-2, to=3-4]
		\arrow["{f''}", from=1-1, to=3-4, crossing over,curve={height=-27pt}]
	\end{tikzcd}\]
	Hence, $f''$ is (up to identification) the composition of two cocartesian arrows, hence cocartesian itself.
\end{proof}

\begin{proposition}[Left exactness of terminal transport, \protect{\cite[Lemma~15.7]{streicher2020fibered}}]\label{prop:ttrans-lex}
	Let $B$ be a lex Rezk type and $P:B \to \UU$ be a Moens family. Then the terminal transport functor $\omega: \widetilde{P} \to P\,z$ is lex.
\end{proposition}

\begin{proof}
	Preservation of the terminal object follows from 
	\[ \omega(z,\zeta_z) \jdeq (!_z)_!(\zeta_z) = \id_{\zeta_z}.\]
	Consider a dependent pullback
	\[\begin{tikzcd}
		{e'''} && {e''} \\
		{e'} && e
		\arrow["{f'}", from=1-1, to=1-3]
		\arrow["{g'}"', from=1-1, to=2-1]
		\arrow["f"', from=2-1, to=2-3]
		\arrow["g", from=1-3, to=2-3]
		\arrow["\lrcorner"{anchor=center, pos=0.125}, draw=none, from=1-1, to=2-3]
	\end{tikzcd}\]
	in $\totalty{P}$, where $f$ lies over an arrow $u$, and $g$ over an arrow $v$ in $B$. Considering the induced diagram
	\[\begin{tikzcd}
		{e'''} &&&& {e''} \\
		&& d && {v_!\,e''} \\
		{e'} && {u_!\,e'} && e
		\arrow["m", squiggly, from=3-3, to=3-5]
		\arrow["k", from=3-1, to=3-3, cocart]
		\arrow["{r'}"', squiggly, from=2-3, to=3-3]
		\arrow["{m'}", squiggly, from=2-3, to=2-5]
		\arrow["\ell"', from=1-5, to=2-5, cocart]
		\arrow["{f'}", from=1-1, to=1-5]
		\arrow["{g'}"', from=1-1, to=3-1]
		\arrow["\lrcorner"{anchor=center, pos=0.125}, draw=none, from=2-3, to=3-5]
		\arrow["r"', squiggly, from=2-5, to=3-5]
		\arrow["h", from=1-1, to=2-3]
		\arrow["\lrcorner"{anchor=center, pos=0.125}, shift right=4, draw=none, from=1-1, to=2-3]
		\arrow["f"{description}, curve={height=14pt}, from=3-1, to=3-5]
		\arrow["g"{description}, curve={height=-18pt}, from=1-5, to=3-5]
		\arrow["\sigma"{description}, draw=none, from=2-3, to=3-5]
	\end{tikzcd}\]
	by~\Cref{lem:gap-to-vert-cocart}, we find that the gap map $h$ is cocartesian.

	The image of the lower left square under $\omega$ is the front square of the following cube:
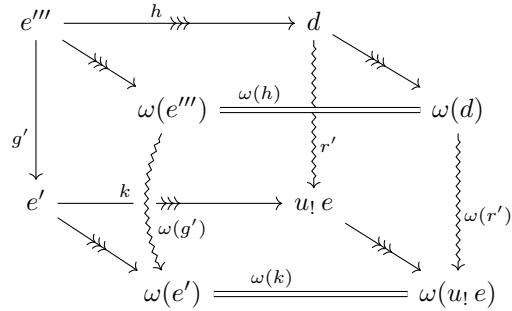
\[\begin{tikzcd}
	{e'''} && d \\
	& {\omega(e''')} && {\omega(d)} \\
	{e'} && {u_!\,e} \\
	& {\omega(e')} && {\omega(u_!\,e)}
	\arrow["h"{pos=0.4}, from=1-1, to=1-3, cocart]
	\arrow["{g'}"'{pos=0.7}, from=1-1, to=3-1]
	\arrow["k"{pos=0.3}, from=3-1, to=3-3, cocart]
	\arrow["{r'}"{pos=0.7}, squiggly, from=1-3, to=3-3]
	\arrow[from=1-1, to=2-2, cocart]
	\arrow["{\omega(h)}"{pos=0.2}, Rightarrow, no head, from=2-2, to=2-4]
	\arrow[from=1-3, to=2-4, cocart]
	\arrow[from=3-1, to=4-2, cocart]
	\arrow["{\omega(k)}"{pos=0.3},Rightarrow, no head, from=4-2, to=4-4]
	\arrow[from=3-3, to=4-4, cocart]
	\arrow["{\omega(r')}"{pos=0.6}, squiggly, from=2-4, to=4-4]
	\arrow["{\omega(g')}"{pos=0.8}, curve={height=12pt}, squiggly, from=2-2, to=4-2, crossing over]
\end{tikzcd}\]
	The arrows $\omega(h)$ and $\omega(k)$ are both isomorphisms: as images of arrows under $\omega$ they are both vertical, and by right cancelation of cocartesian arrows they are cocartesian, too, hence isomorphisms.

	By~\Cref{cor:lex-transp-moens} the pullback square $\sigma$ gets mapped to a pullback square under $\omega$, too (actually, by internal extensivity, all of the squares of the ensuing cube are pullbacks):
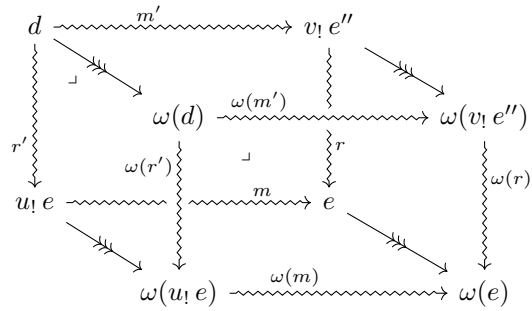
\[\begin{tikzcd}
	d && {v_!\,e''} \\
	& {\omega(d)} && {\omega(v_!\,e'')} \\
	{u_!\,e} && e \\
	& {\omega(u_!\,e)} && {\omega(e)}
	\arrow["{m'}"{pos=0.4}, squiggly, from=1-1, to=1-3]
	\arrow["{r'}"'{pos=0.7}, squiggly, from=1-1, to=3-1]
	\arrow["m"{pos=0.8}, squiggly, from=3-1, to=3-3]
	\arrow["r"{pos=0.7}, squiggly, from=1-3, to=3-3]
	\arrow[from=1-1, to=2-2, cocart]
	\arrow[from=1-3, to=2-4, cocart]
	\arrow[from=3-1, to=4-2, cocart]
	\arrow[from=3-3, to=4-4, cocart]
	\arrow["{\omega(r)}"{pos=0.3}, squiggly, from=2-4, to=4-4]
	\arrow["{\omega(m)}"{pos=0.3}, squiggly, from=4-2, to=4-4]
	\arrow["\lrcorner"{anchor=center, pos=0.125}, draw=none, from=2-2, to=4-4]
	\arrow["{\omega(m')}"{pos=0.2}, squiggly, from=2-2, to=2-4, crossing over]
	\arrow["{\omega(r')}"'{pos=0.2}, squiggly, from=2-2, to=4-2, crossing over]
	\arrow["\lrcorner"{anchor=center, pos=0.125}, shift right=4, draw=none, from=1-1, to=3-3]
\end{tikzcd}\]

	Applying $\omega$ to the upper square in the subdivided diagram gives, by \Cref{cor:lex-transp-moens}:
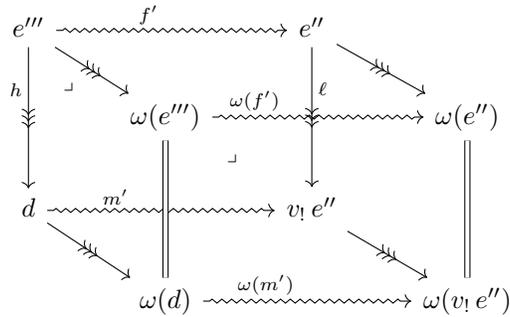
\[\begin{tikzcd}
	{e'''} && {e''} \\
	& {\omega(e''')} && {\omega(e'')} \\
	d && {v_!\,e''} \\
	& {\omega(d)} && {\omega(v_!\,e'')}
	\arrow["h"'{pos=0.3}, from=1-1, to=3-1, cocart]
	\arrow["{m'}"{pos=0.3}, squiggly, from=3-1, to=3-3]
	\arrow["{f'}"{pos=0.4}, squiggly, from=1-1, to=1-3]
	\arrow[from=1-1, to=2-2, cocart]
	\arrow[from=3-1, to=4-2, cocart]
	\arrow["{\omega(m')}"{pos=0.3}, squiggly, from=4-2, to=4-4]
	\arrow[Rightarrow, no head, from=2-4, to=4-4]
	\arrow[Rightarrow, no head, from=2-2, to=4-2]
	\arrow[from=3-3, to=4-4, cocart]
	\arrow[from=1-3, to=2-4, cocart]
	\arrow["{\omega(f')}"{pos=0.2}, squiggly, from=2-2, to=2-4, crossing over]
	\arrow["\lrcorner"{anchor=center, pos=0.125}, shift right=4, draw=none, from=1-1, to=3-3]
	\arrow["\lrcorner"{anchor=center, pos=0.125}, draw=none, from=2-2, to=4-4]
	\arrow["{\ell}"{pos=0.3}, from=1-3, to=3-3, cocart]
\end{tikzcd}\]

	Patching these three squares together yields a diagram of the following shape which contracts to a pullback lying in $P\,z$:
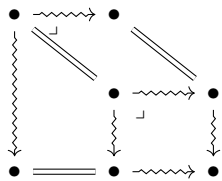
\[\begin{tikzcd}
	\bullet & \bullet \\
	& \bullet & \bullet \\
	\bullet & \bullet & \bullet
	\arrow[Rightarrow, no head, from=3-1, to=3-2]
	\arrow[squiggly, from=3-2, to=3-3]
	\arrow[squiggly, from=2-2, to=3-2]
	\arrow[squiggly, from=2-2, to=2-3]
	\arrow[squiggly, from=2-3, to=3-3]
	\arrow[Rightarrow, no head, from=1-2, to=2-3]
	\arrow[squiggly, from=1-1, to=1-2]
	\arrow[Rightarrow, no head, from=1-1, to=2-2]
	\arrow[squiggly, from=1-1, to=3-1]
	\arrow["\lrcorner"{anchor=center, pos=0.125}, draw=none, from=2-2, to=3-3]
	\arrow["\lrcorner"{anchor=center, pos=0.125}, draw=none, from=1-1, to=2-3]
\end{tikzcd}\]
\end{proof}

\subsection{Moens'~Theorem}\label{ssec:moens-thm}

We are now ready to prove a version of Moens' Theorem characterizing the type of Moens fibrations over a fixed lex base as the type of lex functors from this type into some other lex type. The proof is an adaptation of~\cite[Theorem~15.18]{streicher2020fibered}, \cf~also~\cite[Exercise~9.2.13]{jacobs-cltt}. Ideally, this equivalence should go between two categories, \ie, Rezk types in our case. However, since the version of simplicial type theory at hand lacks the Rezk universe of small Rezk types, we rather construct the respective types as as $\Sigma$-types, which are not expected to be (complete Segal). They have the right objects but not the correct kinds of arrows and higher simplices. However, replacing the $\Sigma$-types we are about to define by the correct Rezk analogues (namely, their Rezk completions) should yield, essentially by the same proof, the improved, ``categorical'' version of Moens' Theorem. After all, we are working in simplicial type theory, so functions between Rezk types are automatically functors. The condition of being a Moens family is propositional, hence this gives rise to a predicate
\[ \isMoensFam : \big( \sum_{B:\UU} \UU^B \big) \to \Prop\]
witnessing that a family $P :B \to \UU$ is a Moens family (over a small Rezk type $B$). We do not spell out a definition of $\isMoensFam$ but it can be read off of our definitions and characterizations of Moens families.

Similarly, we have a predicate $\mathrm{isLex} \colon \Rezk \to \Prop$ on the type $\Rezk \defeq \sum_{A: \UU} \mathrm{isRezk}(A)$ of small Rezk types, witnessing that a small Rezk type $A$ is lex. This gives rise to a subuniverse
\[ \LexRezk \defeq \sum_{A:\UU}  \mathrm{isRezk}(A) \times \mathrm{isLex}(A) \hookrightarrow \Rezk \]
of lex Rezk types. Likewise, for any $B,C : \UU$, our definition of lex exact functor gives rise to a predicate
\[ \mathrm{isLexFun} _{B,C} : (B \to C) \to \Prop,\]
giving rise to the type 
\[  (B \to^\lex C) \defeq \sum_{F: B \to C}  \mathrm{isLexFun} _{B,C}(F) \]
of lex exact functors from $B$ to $C$, which is a subtype of the functor type $B \to C$.
\begin{theorem}[Moens'~Theorem, \protect{\cite[Theorem~15.18]{streicher2020fibered}}, \protect{\cite[Section~5, Proposition~12]{LietzDip}}]\label{thm:moens-thm}
	For a small lex Rezk type $B:\UU$ the type
	\[ \MoensFam(B) \defeq \sum_{P:B \to \UU} \isMoensFam \,P\]
	of $\UU$-small Moens families is equivalent to the type
	\[ B \downarrow^\lex \mathrm{LexRezk} \defeq\sum_{C:\LexRezk} (B \to^\lex C)\]
	of lex functors from $B$ into the type $\LexRezk$ of $\UU$-small lex Rezk types.
\end{theorem}

\begin{proof}
	We define a pair of quasi-inverses
	\[\begin{tikzcd}
		{\mathrm{MoensFam}(B)} &&& {B \downarrow^\lex \mathrm{LexRezk}}
		\arrow[""{name=0, anchor=center, inner sep=0}, "\Phi", curve={height=-18pt}, from=1-1, to=1-4]
		\arrow[""{name=1, anchor=center, inner sep=0}, "\Psi", curve={height=-18pt}, from=1-4, to=1-1]
		\arrow["\simeq"{description}, Rightarrow, draw=none, from=0, to=1]
	\end{tikzcd}\]
	by setting
	\begin{align*} 
		\Phi(P:B \to \UU) & \defeq \omega_P' \jdeq \lambda b.(!_b)_!(\zeta_b): B \to P\,z, \\
		\Psi(F:B \to C)&  \defeq \big(\gl(F):\commaty{C}{F} \to B\big).
	\end{align*} 
	Indeed, the definitions of $\Phi$ and $\Psi$ are well-typed.

	First, we have $\Phi \jdeq \omega' \jdeq \omega \circ \zeta$. The map $\zeta$ is lex by~\Cref{cor:lex-zeta}. Finally, the functions are indeed well-typed due to~\Cref{prop:ttrans-lex,prop:gluing-moens}.
	
	For the first roundtrip, let $P:B \to \UU$ be a Moens family. We have $\Phi(P) \jdeq \omega_P':B \to P\,z$ and $\Psi(\Phi(P)) \jdeq \gl(\omega'): \commaty{P\,z}{\omega'_P} \fibarr B$. We want to give an identification $P =  \gl(\omega'_P)$ in the type of Moens families over $B$ which amounts to a fiberwise equivalence $\widetilde{P} \simeq_B \commaty{P\,z}{\omega_P'}$. Abbreviating $\omega' \defeq \omega'_P$ we are to define a pair of fibered quasi-inverses:
	\[\begin{tikzcd}
		{\widetilde{P}} && {P\,z \downarrow \omega'} \\
		& B
		\arrow[two heads, from=1-1, to=2-2]
		\arrow[two heads, from=1-3, to=2-2]
		\arrow[""{name=0, anchor=center, inner sep=0}, "\varphi"{description}, curve={height=-12pt}, from=1-1, to=1-3]
		\arrow[""{name=1, anchor=center, inner sep=0}, "\psi"{description}, curve={height=-6pt}, from=1-3, to=1-1]
		\arrow["\simeq"{description}, Rightarrow, draw=none, from=0, to=1]
	\end{tikzcd}\]
	
	We introduce the following notation. For $b:B$, $e:P\,b$, consider the following canonical square:
	\[\begin{tikzcd}
		{\widetilde{P}} && e && {(!_b)_!e} \\
		&& {\zeta_b} && {\omega'(b)} \\
		B && b && z
		\arrow["{!_b}", dashed, from=3-3, to=3-5]
		\arrow["{\nu_b}"', from=2-3, to=2-5, cocart]
		\arrow[from=1-1, to=3-1, two heads]
		\arrow["{!_e^b}"', dashed, from=1-3, to=2-3]
		\arrow["{\kappa_e}", from=1-3, to=1-5, cocart]
		\arrow["{\mu_e}", dashed, from=1-5, to=2-5]
	\end{tikzcd}\]
	We denote by $!_e^b:e \to \zeta_b$ the terminal map of $e$ in $P\,b$, and by $\nu_b: \zeta_b \cocartarr_{!_b} \omega'(b)$ the cocartesian lift of $!_b$ \wrt~$\zeta_b:P\,b$.
	
	By $\kappa_e$, we denote the cocartesian lift of $!_b$ \wrt~$e:P\,b$. We abbreviate by $\mu_e:(!_b)_!e \to \omega'(b)$ the filler
	\[ \mu_e \defeq \tyfill_{\kappa_e}(\nu_e \circ !_e^b).\]
	Then, we define 
	\[ \varphi:\prod_{b:B} P\,b \to \commaty{P\,z}{\omega'\,b}, \quad \varphi_b(e)\defeq \mu_{e}: (!_b)_!e \vertarr_z \omega'(b)  \]
	and
	\[ \psi:\prod_{b:B}\commaty{P\,z}{\omega'\,b} \to P\,b, \quad \psi_b(f:e\vertarr_z \omega'(b)) \defeq \zeta_b \times_{\omega'(b)} e : P\,b. \]
	For the first part of the round trip, we take $e:P\,b$ which gets mapped to $\varphi_b(e) = \mu_e: (!_b)_!e \vertarr_z \omega'(b)$. Computing the pullback $\psi_b(\mu_e)$ recovers $e$ by~\Cref{prop:ext-sums}, (\ref{it:ext-sums-lawvere}) (or (\ref{it:ext-sums-ext})).
	The reverse direction is established as follows. Starting with a (vertical) arrow $f:e \vertarr_z \omega'(b)$, we consider the dependent pullback, with a pasted identity of arrows:
	\[\begin{tikzcd}
		&& {(!_b)_!e'} \\
		{e'} && e \\
		{\zeta_b} && {\omega'(b)}
		\arrow[squiggly, from=2-1, to=3-1, "!_e^b", swap]
		\arrow["{\nu_b}"', from=3-1, to=3-3, cocart]
		\arrow[from=2-1, to=2-3, cocart]
		\arrow[squiggly, from=2-3, to=3-3, "f"]
		\arrow["\lrcorner"{anchor=center, pos=0.125}, draw=none, from=2-1, to=3-3]
		\arrow[from=2-1, to=1-3, cocart]
		\arrow[Rightarrow, no head, from=1-3, to=2-3]
	\end{tikzcd}\]
	Then, the arrow $\psi_b(\varphi_b(e))$ is given by the composite
	\[ (!_b)_!e'  \xlongequal{~~} e \xlongrightarrow{f} \omega'(b) \]
	which can be identified with $f$. In sum, we have proven $\Psi \circ \Phi = \id$. We are left with the other direction.
	
	Let $C$ be a lex Rezk type and $F:B \to C$ a lex functor. Then $\Psi(F) = \gl(F): \commaty{C}{F} \fibarr B$. To compute the application of $\Phi$, we have to first specialize the action of the fiberwise terminal element section for $\gl(F)$. It is given by
	\[ \zeta \defeq \lambda b.\angled{b,F\,b, \id_{F\,b}}:\prod_{b:B} \sum_{c:C} c \to F\,b.\]
	Since $F$ is lex we have an identification $F(z) = y$ where $y:C$ is terminal. Then the terminal transport functor of the fibration $\gl(F)$ yields
	\[ \omega' \defeq \pair{F\,b}{F(!_z):F\,b \to y}: B \to \sum_{c:C} c \to y. \]
	We have $\Phi(\Psi(F)) = \omega'_{\St_B(\gl(F))}$, and since $y$ is terminal, we can identify this map with $F$. 
\end{proof}

\subsection{Generalized Moens families}\label{ssec:gen-moens}

Streicher observes at the end of~\cite[Section~5]{streicher2020fibered} that a generalized version of Moens' theorem holds even in the absence of internal sums, \ie~the Beck--Chevalley condition. The idea is the following. Given any functor $F$ between lex Rezk types the gluing $\gl(F)$ is always a lex bifibration that moreover satisfies internal extensivitity in the sense of~\Cref{prop:ext-sums}, \Cref{it:ext-sums-ext}---even in the absence of internal sums. We call such families \emph{generalized Moens}. Now, if $F$ happens to preserve terminal elements we can argue essentially as in~\Cref{thm:moens-thm}. Conversely, for any generalized Moens family $P:B \to \UU$ the functor $\omega_P': B \to P\,z$ preserves the terminal element.

\begin{definition}[Generalized Moens family]\label{def:gen-moens}
	Let $P:B \to \UU$ be a lex bicartesian family over a Rezk type $B$. We call $P$ a \emph{generalized Moens family} if the following two propositions hold:
	\begin{enumerate}
		\item In $P$, cocartesian arrows are stable under pullback along vertical arrows.
		\item Any dependent square whose neighboring and opposing faces are vertical and cocartesian, respectively, as indicated in the figure is a pullback:
		\[\begin{tikzcd}
			d && e \\
			{d'} && e
			\arrow["f"', squiggly, from=1-1, to=2-1]
			\arrow["g"', cocart, from=2-1, to=2-3]
			\arrow["h", cocart, from=1-1, to=1-3]
			\arrow["k", squiggly, from=1-3, to=2-3]
			\arrow["\lrcorner"{anchor=center, pos=0.125}, draw=none, from=1-1, to=2-3]
		\end{tikzcd}\]
	\end{enumerate}
\end{definition}
For lex bicartesian families $P:B \to \UU$ this gives rise to a proposition
\[ \isGMoensFam(P) \]
witnessing that $P$ is generalized Moens. We do not spell out the definition but it is straightforward, see the discussion at the beginning of~\Cref{ssec:moens-thm}

Recall that the gluing of \emph{any} functor between lex Rezk types is always a lex bifibration, \cf~\Cref{prop:gluing-lex}. It is straightforward to check that in a gluing Rezk type a dependent cube as below is a pullback if and only if the arrow $h$ is an isomorphism:
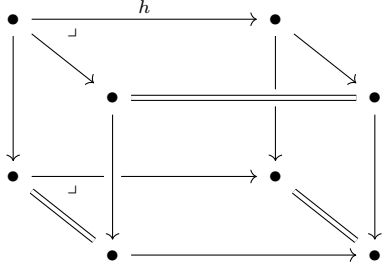
\[\begin{tikzcd}
	\bullet &&& \bullet \\
	& \bullet &&& \bullet \\
	\bullet &&& \bullet \\
	& \bullet &&& \bullet
	\arrow[from=1-1, to=3-1]
	\arrow["h", from=1-1, to=1-4]
	\arrow[from=1-4, to=3-4]
	\arrow[from=1-1, to=2-2]
	\arrow[Rightarrow, no head, from=3-1, to=4-2]
	\arrow[from=4-2, to=4-5]
	\arrow[from=2-5, to=4-5]
	\arrow[from=1-4, to=2-5]
	\arrow[Rightarrow, no head, from=3-4, to=4-5]
	\arrow["\lrcorner"{anchor=center, pos=0.125}, draw=none, from=3-1, to=4-5]
	\arrow["\lrcorner"{anchor=center, pos=0.125}, draw=none, from=1-1, to=2-5]
	\arrow[from=3-1, to=3-4]
	\arrow[from=2-2, to=4-2, crossing over]
	\arrow[Rightarrow, no head, from=2-2, to=2-5, crossing over]
\end{tikzcd}\]
Without detailing this further this establishes the following proposition:

\begin{proposition}\label{prop:gl-gmoens}
	Let $F: C \to B$ be a functor between small lex Rezk types $B$ and $C$. Then the Artin gluing $\gl(F): \commaty{C}{B} \fibarr B$ is generalized Moens.
\end{proposition}

Note that for any Rezk type $B$ with a terminal element and any cocartesian family $P:B \to \UU$ whose fibers have terminal elements the map $\omega'_P: B \to P\,z$ preserves the terminal object. This observation together with~\Cref{prop:gl-gmoens} implies that we can extend the pair of maps from~\Cref{thm:moens-thm}. Noting further that in the proof of~\Cref{thm:moens-thm} we actually do not need $F:B \to C$ to preserve pullbacks but only terminal elements implies Streicher's generalized version of Moens' Theorem. Again, we are looking at the (easy to define) sub-type $(B \to^\ter C)$ of functors $(B \to C)$ that are only demanded to preserve the terminal element.

\begin{theorem}[Generalized Moens'~Theorem, \protect{\cite[Theorem~15.19]{streicher2020fibered}}]\label{thm:gen-moens-thm}
	For a small lex Rezk type $B:\UU$ the type
	\[ \GMoensFam(B) \defeq \sum_{P:B \to \UU} \isGMoensFam \,P\]
	of $\UU$-small Moens families is equivalent to the type
	\[ B \downarrow^\ter \mathrm{LexRezk} \defeq\sum_{C:\LexRezk} (B \to^\ter C)\]
	of terminal element-preserving functors from $B$ into the type $\LexRezk$ of $\UU$-small lex Rezk types.
\end{theorem}

\begin{proof}
	The equivalence is given as in~\Cref{thm:moens-thm}.
\end{proof}

\subsection{Zawadowski's definition}\label{ssec:zawadowski}

In~\cite{zawadowski-lax-mon} Zawadowski introduces a variant of lex bifibrations\footnote{therein called \emph{cartesian bifibrations} but we won't use this term to avoid confusions with our preexisting terminology} where the covariant reindexing $u_! \dashv u^*$ preserves pullbacks and the units and counits of these adjunctions are cartesian natural transformations. Streicher has shown~\cite{streicher-cartbifib} that these conditions are in fact equivalent to the fibration being generalized Moens. We provide a slight complementation of his proof (using the generalized Moens' Theorem).

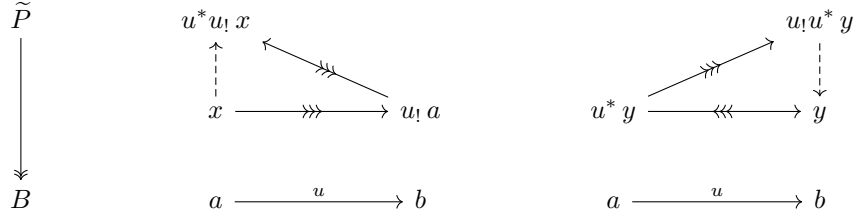
\begin{figure}
\[\begin{tikzcd}
	{\totalty{P}} && {u^* u_!\,x} &&&&&& {u_!u^*\,y} \\
	&& x && {u_!\,a} && {u^*\,y} && y \\
	B && a && b && a && b
	\arrow["u", from=3-3, to=3-5]
	\arrow[two heads, from=1-1, to=3-1]
	\arrow[from=2-3, to=2-5, cocart]
	\arrow[from=2-5, to=1-3, cart]
	\arrow[from=2-7, to=2-9, cart]
	\arrow[dashed, from=1-9, to=2-9]
	\arrow["u", from=3-7, to=3-9]
	\arrow[from=2-7, to=1-9, cocart]
	\arrow[dashed, from=2-3, to=1-3]
\end{tikzcd}\]
\label{fig:bicart}
\caption{Unit and counit from the transport in a bicartesian fibration}
\end{figure}

Recall that in a cartesian bifibration, $P : B \to \UU$ over a Rezk type $B$, for any $u : a \to_B b$ we have an adjunction $u_! \dashv u^*$, with $u_! : P\,a \to P\,b$ and $u^*: P\,b \to P\,a$. This gives rise to the unit
\[ \eta \defeq \eta_{P,u}: \id_{P\,a} \Rightarrow u^*u_!, \quad \eta_x \defeq \cartFill_{P^*(u,u_!\,x)}(P_!(u,x)) : x \vertarr_a^P u^*u_!\,x\]
and the counit, respectively:
\[ \varepsilon \defeq \varepsilon_{P,u}:  u_! u^*  \Rightarrow \id_{P\,b}, \quad \varepsilon_y \defeq \cocartFill_{P_!(u,u^*\,y)}(P^*(u,y)) : u_! u^* y \vertarr_b^P y\]
See~\Cref{fig:bicart} for an illustration.

\begin{theorem}\label{thm:zawadowski}
Let $P:B \to \UU$ be a lex bifibration over a Rezk type $B$. Then $P$ is generalized Moens if and only if the following two propositions are satisfied:
\begin{enumerate}
	\item\label{it:zawad-lex} The cocartesian reindexing maps $u_!:P \,a \to P\,b$ preserve pullbacks, for all $u:a \to_B b$.
	\item\label{it:zawad-cart} The units $\eta \defeq \eta_{P,u}: \id_{P\,a} \Rightarrow u^*u_!$ and counits $\varepsilon \defeq \varepsilon_{P,u} :u_! u^* \Rightarrow \id_{P\,b}$ are \emph{cartesian} natural transformations, for all $u: a \to_B b$. This means for all $f:x\to_{P\,a} x'$, $g:y \to_{P\,b} y'$, the naturality squares are pullbacks:
	\[\begin{tikzcd}
		x && {u^*\,u_!\,x} && {u_!\,u^*\,y} && y \\
		{x'} && {u^*\,u_!\,x'} && {u_!\,u^*\,y'} && {y'}
		\arrow["f"', from=1-1, to=2-1]
		\arrow["{\eta_x}", from=1-1, to=1-3]
		\arrow["{u_!\,u^*\,g}"', from=1-5, to=2-5]
		\arrow["{\varepsilon_{y'}}", from=2-5, to=2-7]
		\arrow["{\varepsilon_{y}}", from=1-5, to=1-7]
		\arrow["g", from=1-7, to=2-7]
		\arrow["\lrcorner"{anchor=center, pos=0.125}, draw=none, from=1-5, to=2-7]
		\arrow["\lrcorner"{anchor=center, pos=0.125}, draw=none, from=1-1, to=2-3]
		\arrow["{\eta_{x'}}", from=2-1, to=2-3]
		\arrow["{u^*u_!\,f}", from=1-3, to=2-3]
	\end{tikzcd}\]
\end{enumerate}
\end{theorem}

\begin{proof}
	In~\cite{streicher-cartbifib} Streicher proves that a lex bifibration satisfying Zawadowski's Conditions~\ref{it:zawad-lex} and~\ref{it:zawad-cart} is generalized Moens. His proof carries over to our synthetic higher setting as well. It therefore remains to show the converse,~\ie, that any given generalized Moens family satisfies~\ref{it:zawad-lex} and~\ref{it:zawad-cart}.
	
	Let $B$ be a lex Rezk type with terminal element $z:B$. From the (proof of) generalized Moens' Theorem we obtain a fibered equivalence
	\[ \totalty{P} \simeq_B \commaty{P\,z}{\omega'}, \]
	for a given generalized Moens family $P:B \to \UU$. Let $u:a \to_B b$ be an arrow in $B$. The transport functors are given by
	\begin{align*}
		u_!(f:x \to_{P\,z} \omega'(a)) & = \omega'(u) \circ f: x \to_{P\,b} \omega'(b), \\
		u^*(g:x \to_{P\,z} \omega'(b)) & = \omega'(u)^* g: y \times_{\omega'(b)} \omega'(a) \to_{P\,b} \omega'(b). \\
	\end{align*}
	We first verify Condition~\ref{it:zawad-lex}. A pullback in $\commaty{P\,z}{\omega'(a)} \simeq P\,z/\omega'(a)$ is given by pullback in $P\,z$ fibered over $\omega'(a)$. The action of $u_!$ on this diagram is given by postcomposing with $\omega'(u):\omega'(a) \to \omega'(b)$:
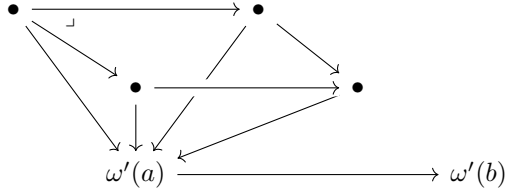
	\[\begin{tikzcd}
		\bullet && \bullet \\
		& \bullet && \bullet \\
		& {\omega'(a)} &&& {\omega'(b)}
		\arrow[from=1-1, to=2-2]
		\arrow[from=1-3, to=2-4]
		\arrow[from=1-1, to=1-3]
		\arrow["\lrcorner"{anchor=center, pos=0.125}, draw=none, from=1-1, to=2-4]
		\arrow[from=1-1, to=3-2]
		\arrow[from=2-2, to=3-2]
		\arrow[from=2-4, to=3-2]
		\arrow[from=3-2, to=3-5]
		\arrow[from=1-3, to=3-2]
		\arrow[from=2-2, to=2-4, crossing over]
	\end{tikzcd}\]
	This, in turn, constitutes a pullback in $P\,z/\omega'(b) \simeq \commaty{P\,z}{\omega'(b)}$ again, as desired.
	
	We now turn to Condition~\ref{it:zawad-cart}. As for the unit, consider the following induced decomposition diagram:
	
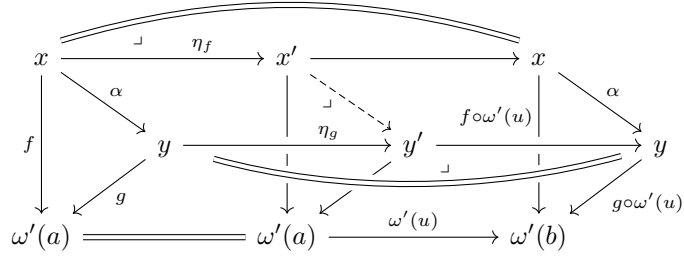
\[\begin{tikzcd}
		x && {x'} && x \\
		& y && {y'} && y \\
		{\omega'(a)} && {\omega'(a)} && {\omega'(b)}
		\arrow["{\eta_f}"{pos=0.7}, from=1-1, to=1-3]
		\arrow[from=1-3, to=1-5]
		\arrow[Rightarrow, no head, from=3-1, to=3-3]
		\arrow[from=1-3, to=3-3]
		\arrow["\omega'(u)", from=3-3, to=3-5]
		\arrow["f\circ \omega'(u)"'{pos=0.3}, from=1-5, to=3-5]
		\arrow["\alpha", from=1-5, to=2-6]
		\arrow[dashed, from=1-3, to=2-4]
		\arrow["\alpha", from=1-1, to=2-2]
		\arrow["\lrcorner"{anchor=center, pos=0.125}, draw=none, from=2-4, to=3-5]
		\arrow["\lrcorner"{anchor=center, pos=0.125}, shift right=2, draw=none, from=1-3, to=3-5]
		\arrow["f"', from=1-1, to=3-1]
		\arrow["{\eta_g}"{pos=0.7}, from=2-2, to=2-4, crossing over]
		\arrow["g", from=2-2, to=3-1, crossing over]
		\arrow[from=2-4, to=3-3, crossing over]
		\arrow[from=2-4, to=2-6, crossing over]
		\arrow["g \circ \omega'(u)", from=2-6, to=3-5, crossing over]
		\arrow[shift left=2, curve={height=-18pt}, Rightarrow, no head, from=1-1, to=1-5, crossing over]
		\arrow[curve={height=18pt}, Rightarrow, no head, from=2-2, to=2-6, crossing over, shorten <=12pt, shorten >=8pt]
		\arrow["\lrcorner"{anchor=center, pos=0.125}, shift left=5, draw=none, from=1-1, to=2-6]
	\end{tikzcd}\]
	In the right-hand prism diagram, the front and back squares are pullbacks by construction, therefore the top one is, too. Since the composite square on the top is a pullback as well, so is the top square of the left-hand prism. This proves the condition.
	
	Analogously, the counit square appears in the following decomposition where the back and front composite squares are pullbacks:
	
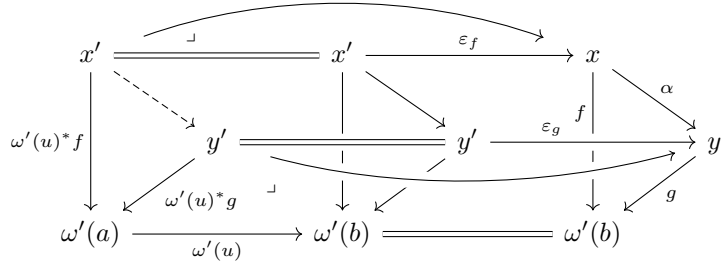
\[\begin{tikzcd}
		x' && x' && x \\
		& y' && y' && y \\
		{\omega'(a)} && {\omega'(b)} && {\omega'(b)}
		\arrow[Rightarrow, no head, from=1-1, to=1-3]
		\arrow["\varepsilon_f", from=1-3, to=1-5]
		\arrow[dashed, from=1-1, to=2-2]
		\arrow[from=1-3, to=2-4]
		\arrow["\alpha", from=1-5, to=2-6]
		\arrow["\omega'(u)^*f"', from=1-1, to=3-1]
		\arrow["{\omega'(u)}"', from=3-1, to=3-3]
		\arrow[from=1-3, to=3-3]
		\arrow[Rightarrow, no head, from=3-3, to=3-5]
		\arrow["g", from=2-6, to=3-5]
		\arrow["\omega'(u)^*g", from=2-2, to=3-1, crossing over]
		\arrow[from=2-4, to=3-3, crossing over]
		\arrow["\lrcorner"{anchor=center, pos=0.125}, shift left=5, draw=none, from=1-1, to=2-6]
		\arrow["f"'{pos=0.3}, from=1-5, to=3-5]
		\arrow[curve={height=-24pt}, from=1-1, to=1-5, crossing over, shorten <=12pt, shorten >=12pt]
		\arrow[Rightarrow, no head, from=2-2, to=2-4, crossing over]
		\arrow[curve={height=18pt}, from=2-2, to=2-6, shorten <=12pt, shorten >=8pt, crossing over]
		\arrow["\varepsilon_g"{pos=0.3}, from=2-4, to=2-6, crossing over]
		\arrow["\lrcorner"{anchor=center, pos=0.125}, shift right=5, draw=none, from=2-2, to=3-5]
	\end{tikzcd}\]
	We only need to argue that the top square of the right-hand prism is a pullback, but this is clear since the composite square is a pullback, and it factors as indicated over the identities.
\end{proof}

\section*{Acknowledgments}
I am grateful for financial support by the US Army Research Office under MURI Grant W911NF-20-1-0082. Furthermore, I thank the MPIM~Bonn for its hospitality and financial support.

I wish to thank Ulrik Buchholtz, Emily Riehl, Thomas Streicher, for many helpful discussions, valuable feedback, and steady guidance.

Ulrik Buchholtz has been my collaborator for the initial work on cocartesian fibrations of synthetic $\inftyone$-categories which was the basis for all the follow-up work in my PhD thesis, including the present text. Emily Riehl has very actively accompanied all stages of my graduate work with essential advice and key insights. Thomas Streicher has been my doctoral advisor. He engaged tirelessly in countless explanations and discussions concerning fibered category theory, in particular from the points of view of B\'{e}nabou, Jibladze, and Moens, as layed out in his excellent lecture notes~\cite{streicher2020fibered} which form an essential basis for the work at hand.

Additionally, Ulrik Buchholtz and Thomas Streicher both pointed out some mistaken proofs in a draft version of this text. I am highly indebted to the anonymous referee for a very detailed review and many helpful corrections and comments regarding the presentation of this material as well as the content of the proofs.

I furthermore thank Mathieu Anel, Tim Campion, Rune Haugseng, Sina Hazratpour, Louis Martini, Maru Sarazola, and Sebastian Wolf for insightful discussions and feedback.

\phantomsection%
\printbibliography[heading=bibintoc]

\end{document}